\documentclass{amsart}

\usepackage[T1]{fontenc}
\usepackage[francais]{babel}
\usepackage{bbm}
\usepackage{amssymb}
\usepackage{amsmath}
\usepackage{amsthm}
\usepackage{mathtools}
\usepackage{hyperref}
\usepackage{graphics}
\usepackage{enumerate}

\usepackage{eulervm}

\usepackage{todonotes}

\newtheorem{proposition}{Proposition}[section]
\newtheorem{conjecture}{Conjecture}[section]
\newtheorem{propriete}{Propriété}[section]

\newtheorem{theoreme}{Théorème}[section]
\newtheorem{corollaire}{Corollaire}[section]
\newtheorem{remarque}{Remarque}[section]
\newtheorem{lemme}{Lemme}[section]

\DeclareMathOperator{\Hom}{\mathnormal{Hom}}

\begin{document}

\title{Formule de Plancherel sur $GL_n \times GL_n \backslash GL_{2n}$}
\date{\today}
\author{Nicolas Duhamel}
\maketitle

\section{Introduction}

Soit $F$ un corps $p$-adique, $G$ un groupe réductif déployé sur $F$ et $X = H \backslash G$ une variété sphérique homogène admettant une mesure invariante. 
Sakellaridis et Venkatesh \cite{sak-ven} introduisent une donnée radicielle associé à $X$, qui n'existe que sous certaines conditions sur $X$. On peut associer à la donnée radicielle duale un groupe réductif complexe $\check{G}_X$ qu'ils appellent le groupe dual de la variété sphérique $X$. On note $G_X$ le groupe réductif déployé sur $F$ dont le groupe dual est $\check{G}_X$, le groupe $G_X$ est associé à la donnée radicielle de $X$.
Sakellaridis et Venkatesh introduisent aussi un morphisme de groupes algébriques $\iota_X : \check{G}_X \times SL_2(\mathbb{C}) \rightarrow \check{G}$ sous certaines hypothèses. L'existence de l'application $\iota_X$ a ensuite été vérifiée par Knop et Schalke \cite{knop} sans ces hypothèses. Supposons que $\iota_X$ est trivial sur $SL_2(\mathbb{C})$.

La correspondance de Langlands locale pour $G$ donne une application surjective $Irr(G) \rightarrow \Phi(G)$ à fibres finies entre l'ensemble $Irr(G)$ des classes d'isomorphisme de représentations irréductibles de $G$ et l'ensemble $\Phi(G) = \{\phi : W_F' \rightarrow {}^LG \text{ admissible}\}$ des paramètres de Langlands, où $W_F'$ est le groupe de Weil-Deligne de $F$. La correspondance de Langlands locale donne une partition de $Irr(G)$ en $L$-paquets
\begin{equation}
Irr(G)  = \bigcup_{\phi \in \Phi(G)} \Pi^G(\phi),
\end{equation}
où $\Pi^G(\phi)$ est l'ensemble des classes d'isomorphisme de représentations qui ont pour paramètre de Langlands $\phi$. La correspondance de Langlands locale est prouvée pour $GL_n(F)$ par Harris-Taylor \cite{harris}, Henniart \cite{henniart2}, Scholze \cite{scholze} et pour les groupes orthogonaux impairs par Arthur \cite{arthur}.
Rappellons la

\begin{conjecture}[Sakellaridis-Venkatesh {\cite[Conjecture 16.2.2]{sak-ven}}]
Il existe un isomorphisme $G$-équivariant de représentations unitaires
\begin{equation}
L^2(X) \simeq \int^{\oplus}_{\Phi_{temp}(G_X)} \mathcal{H}_\phi d\phi,
\end{equation}
où $\Phi_{temp}(G_X)$ est l'ensemble des paramètres de Langlands tempérés de $G_X$ modulo $\check{G}_X$-conjugaison, $d\phi$ est dans la classe naturelle des mesures sur $\Phi_{temp}(G_X)$ et $\mathcal{H}_\phi$ est une somme directe sans multiplicité de représentations dans $\Pi^G(\iota_X \circ \phi)$.
\end{conjecture}

On note $Temp(G)$ (resp. $Temp(G_X)$) l'ensemble des classes d'isomorphisme de représentations irréductibles tempérées de $G$ (resp. $G_X$). Supposons la correspondance de Langlands locale pour $G_X$, on dispose alors d'une correspondance fonctorielle $T_X : Temp(G_X) \rightarrow Temp(G)$. Cette correspondance associe à une classe d'isomorphisme de représentations tempérées de $G_X$ un ensemble fini de classes d'isomorphisme de représentations tempérées de $G$. On obtient alors la

\begin{conjecture}[Sakellaridis-Venkatesh]
On note $d\sigma$ la mesure spectrale sur $G_X$ (voir section \ref{mesures}). Supposons que la mesure spectrale $d\sigma$ sur $G_X$ se descend en une mesure sur $Temp(G_X)\slash{\sim}$, où $\sim$ est la relation d'équivalence "égalité des paramètres de Langlands". Alors il existe un isomorphisme $G$-équivariant de représentations unitaires
\begin{equation}
L^2(X) \simeq \int^{\oplus}_{Temp(G_X)/\sim} \widetilde{T}_X(\sigma) d\sigma,
\end{equation}
où $\widetilde{T}_X(\sigma)$ est une somme directe sans multiplicité de représentations dans $T_X(\sigma)$.
\end{conjecture}

Spécifions maintenant au cas où $G = GL_{2n}$ et $X = GL_n \times GL_n \backslash GL_{2n}$. On a $\check{G}_X = Sp_{2n}$ et $G_X = SO(2n+1)$.
La correspondance de Langlands locale est prouvée pour $G$ et $G_X$ par Harris-Taylor \cite{harris}, Henniart \cite{henniart2} et Arthur \cite{arthur}.
De plus, la mesure $d\sigma$ se descend à $Temp(G_X)\slash{\sim}$. L'essentiel de notre travail consiste alors à prouver le
\begin{theoreme}
\label{conjSV}
Il existe un isomorphisme $GL_{2n}$-équivariant de représentations unitaires
\begin{equation}
L^2(GL_n \times GL_n \backslash GL_{2n}) \simeq \int^{\oplus}_{Temp(SO(2n+1))/\sim} T(\sigma) d\sigma,
\end{equation}
où $d\sigma$ est la mesure spectrale sur $SO(2n+1)$ et $T : {Temp(SO(2n+1))\slash{\sim}} \rightarrow Temp(GL_{2n})$ est l'application de transfert provenant de la correspondance de Langlands locale.
\end{theoreme}

On note $H_n$ le groupe des matrices de la forme $\sigma_n \begin{pmatrix}
1 & X \\
0 & 1
\end{pmatrix}\begin{pmatrix}
g & 0 \\
0 & g
\end{pmatrix} \sigma_n^{-1}$ avec $X \in M_n(F)$ et $g \in GL_n(F)$. L'élément $\sigma_n$ est la matrice associée à la permutation $\bigl(\begin{smallmatrix}
    1 & 2 & \cdots & n & n+1 & n+2 & \cdots & 2n \\
    1 & 3 & \cdots &  2n-1  & 2 & 4 & \cdots & 2n
  \end{smallmatrix}\bigr).$
Soit $\theta$ le caractère sur $H_n(F)$ qui envoie $\sigma_n \begin{pmatrix}
1 & X \\
0 & 1
\end{pmatrix}\begin{pmatrix}
g & 0 \\
0 & g
\end{pmatrix} \sigma_n^{-1}$ sur $\psi(Tr(X))$.
On déduit le théorème précédent d'un résultat analogue sur le modèle de Shalika. Plus précisément, on prouve le
\begin{theoreme}
\label{th1}
Il existe un isomorphisme $GL_{2n}$-équivariant de représentations unitaires
\begin{equation}
L^2(H_n \backslash GL_{2n}, \theta) \simeq \int^{\oplus}_{Temp(SO(2n+1))/\sim} T(\sigma) d\sigma,
\end{equation}
où $T : {Temp(SO(2n+1))\slash{\sim}} \rightarrow Temp(GL_{2n})$ est l'application de transfert provenant de la correspondance de Langlands locale.
\end{theoreme}

Ce dernier est une conséquence de la formule de Plancherel explicite que l'on prouve dans la section \ref{plancherel}. 
Soit $\psi$ un caractère additif non trivial de $F$. On pose $Y_n = H_n \backslash GL_{2n}$. On note $C^\infty_c(Y_n, \theta)$ l'ensemble des fonctions lisse sur $G_n$, $(H_n, \theta)$-équivariante à gauche et à support compact modulo $H_n$. Pour $f \in C^\infty_c(G_{2n})$, on note
\begin{equation}
\varphi_f(y) = \int_{H_n} f(hy) \theta(h)^{-1} dh,
\end{equation}
pour tout $y \in G_{2n}$. L'application $f \in C^\infty_c(G_{2n}) \mapsto \varphi_f \in C^\infty_c(Y_n, \theta)$ est surjective. Soient $\varphi_1, \varphi_2 \in C^\infty_c(Y_n, \theta)$. Il existe $f_1, f_2 \in C^\infty_c(G_{2n})$ telles que que $\varphi_i = \varphi_{f_i}$ pour $i = 1,2$. On pose $f = f_1 * f_2^{*}$, où $f_2^*(g) = \overline{f_2(g^{-1})}$. 
Pour $W \in \mathcal{W}(\pi, \psi)$, on note
\begin{equation}
\beta(W) = \int_{H^P_n \cap N_{2n} \backslash H^P_n} W(\xi_p) \theta(\xi_p)^{-1} d\xi_p.
\end{equation}
où $H^P_n = H_n \bigcap P_n$, on définit la mesure $d\xi_p$ dans la section \ref{formInv}, voir la section \ref{notations} pour les notations $N_{2n}, P_n$ et $\mathcal{W}(\pi, \psi)$. On pose
\begin{equation}
(\varphi_1, \varphi_2)_{Y_n, \pi} = \sum_{W \in \mathcal{B}(\pi, \psi)} \beta(W) \overline{\beta(\pi(\overline{f}_1)\pi(f_2^\vee)W)},
\end{equation}
pour tout $\pi \in T(Temp(SO(2n+1)))$, où $f_2^\vee(g) = f_2(g^{-1})$ et $\mathcal{B}(\pi, \psi)$ est une base orthonormée de $\mathcal{W}(\pi, \psi)$.

On prouve alors la formule de Plancherel explicite sur $H_n \backslash GL_{2n}$ sous la forme du
\begin{theoreme}
\label{introPlanchExp}
On a
\begin{equation}
(\varphi_1, \varphi_2)_{L^2(Y_n, \theta)} = \int_{Temp(SO(2n+1))\slash{\sim}} (\varphi_1, \varphi_2)_{Y_n, T(\sigma)} \frac{|\gamma^*(0, \sigma, Ad, \psi)|}{|S_\sigma|}d\sigma,
\end{equation}
Le facteur $\frac{|\gamma^*(0, \sigma, Ad, \psi)|}{|S_\sigma|}$ est défini dans la section \ref{notations}.
\end{theoreme}
La preuve de ce théorème est purement locale, elle se base sur la théorie des fonctions zêta introduite par Jacquet-Shalika \cite{jacquet-shalika}. Une fois que l'on aura introduit les préliminaires, il s'agit essentiellement de montrer que l'on peut échanger deux intégrales (voir la preuve du théorème \ref{thPlanch}). L'utilisation des résultats de Jacquet-Shalika nous amène à prouver un résultat sur les facteurs gamma.
\begin{theoreme}
Soit $\pi$ une représentation tempérée irréductible de $GL_{2n}(F)$. On note $\gamma^{JS}(s,\pi,\Lambda^2,\psi)$ le facteur gamma de Jacquet-Shalika, voir section \ref{gamma} et $\gamma(s, \pi, \Lambda^2, \psi)$ le facteur gamma d'Artin défini via la correspondance de Langlands.
Alors il existe une constante $c(\pi)$ de module 1 telle que pour tout $s \in \mathbb{C}$, on ait
\begin{equation}
\gamma^{JS}(s,\pi,\Lambda^2,\psi)=c(\pi)\gamma(s,\pi,\Lambda^2,\psi).
\end{equation}
\end{theoreme}

Dans la suite de cette introduction, $F$ désigne un corps de nombres et $\psi$ un caractère non trivial de $\mathbb{A}_F\slash{F}$.
On définit $H_n(\mathbb{A}_F)$ et $\theta$ de la même manière que précédemment de façon globale.

Soit $\pi$ une représentation automorphe cuspidale irréductible de $GL_{2n}(\mathbb{A}_F)$ et $\phi_1, \phi_2 \in C^\infty_c(H_n(\mathbb{A}_F) \backslash GL_{2n}(\mathbb{A}_F), \theta) = \bigotimes_v^{'} C^\infty_c(H_n(F_v) \backslash GL_{2n}(F_v), \theta_v)$. On note $\Sigma \phi_i \in C^\infty([GL_{2n}])$, pour $i=1,2$, la fonction définie par $\Sigma \phi_i(g) = \sum_{x \in H_n(F) \backslash GL_{2n}(F)} \phi_i(xg)$ pour tout $g \in GL_{2n}(\mathbb{A}_F)$. D'autre part, pour $\varphi \in \pi$, on introduit la période globale
\begin{equation}
\mathcal{P}_{H_n, \theta}(\varphi) = \int_{[Z_{2n} \backslash H_n]} \varphi(h) \theta(h) dh,
\end{equation}
où $Z_{2n}$ est le centre de $GL_{2n}$ et les crochets désignent le quotient des points adéliques modulo les points rationnels.

Sakellaridis et Venkatesh conjecturent une factorisation du produit scalaire
\begin{equation}
<(\Sigma \phi_1)_\pi, (\Sigma \phi_2)_\pi>_{Pet} = \int_{[Z_{2n} \backslash GL_{2n}]}(\Sigma \phi_1)_\pi(g)\overline{(\Sigma \phi_2)_\pi(g)} dg,
\end{equation}
 où $(\Sigma \phi_1)_\pi$ est la projection sur $\pi$ de $\Sigma \phi_i$ et $dg$ est la mesure de Tamagawa de $[Z_{2n} \backslash GL_{2n}]$ \cite[section 17.1]{sak-ven}.
 
 Si $\pi$ est le transfert d'une représentation automorphe cuspidale $\sigma$ de $SO(2n+1)(\mathbb{A}_F)$ alors cette factorisation prend la forme suivante
\begin{equation}
\label{introFact}
<(\Sigma \phi_1)_\pi, (\Sigma \phi_2)_\pi>_{Pet} = q \prod_v' <\phi_{1,v}, \phi_{2,v}>_{\sigma_v},
\end{equation}
où $q$ est un rationnel. Cette factorisation est une conséquence de la factorisation de la période globale en produit de périodes locales, produite par Jacquet-Shalika dans le cas qui nous intéresse (voir plus loin). Les quantités $<\phi_{1,v}, \phi_{2,v}>_{\sigma_v}$ sont des formes hermitiennes $(H_n(F_v), \theta_v)$-équivariante. On renvoie à \cite[section 17.5]{sak-ven} pour la signification du produit $\prod_v'$. En effet, le produit n'est pas absolument convergent et on doit l'interpréter comme l'évaluation d'une fonction $L$. Si $\pi$ n'est pas le transfert d'une représentation automorphe cuspidale de $SO(2n+1)(\mathbb{A}_F)$ alors
\begin{equation}
<(\Sigma \phi_1)_\pi, (\Sigma \phi_2)_\pi>_{Pet} = 0.
\end{equation}
Sakellaridis et Venkatesh conjecturent que les formes hermitiennes $<\phi_{1,v}, \phi_{2,v}>_{\sigma_v}$ (pour $\sigma_v$ tempérée) sont reliées aux formes $(\phi_{1,v}, \phi_{2,v})_{Y_{n,v}, \pi_v}$ qui l'on a définit précédemment. Plus précisement,
\begin{conjecture}[Sakellaridis-Venkatesh {\cite[section 17]{sak-ven}}]
\label{introConj}
On a l'égalité
\begin{equation}
<\phi_{1,v}, \phi_{2,v}>_{\sigma_v} = (\phi_{1,v}, \phi_{2,v})_{Y_{n,v}, T(\sigma_v)}.
\end{equation}
\end{conjecture}

De manière duale, la relation \ref{introFact} est équivalente à une factorisation de la période globale $\mathcal{P}_{H_n, \theta}$ en produit de périodes locales $\mathcal{P}_{H_n, \theta, v}$. Cette factorisation est obtenue par Jacquet-Shalika \cite{jacquet-shalika} à travers leur théorie des fonctions zêta que l'on explicite dans la section \ref{gamma}.

La relation \ref{introLien} entre la période locale et la forme $\beta$ va nous permettre d'obtenir une formule de Plancherel explicite sur $L^2(H_n \backslash GL_{2n}, \theta)$ prouvant ainsi la conjecture \ref{introConj}. Plus précisément, pour $\Phi$ une fonction de Schwartz sur $\mathbb{A}_F^n$ et $W_\varphi$ la fonction de Whittaker associée à $\varphi$, on introduit dans la suite des fonctions zêta globales $J(s, W_\varphi, \Phi)$, qui sont reliées à la période globale par la relation
\begin{equation}
Res_{s=1} J(s, W_\varphi, \Phi) = c\mathcal{P}_{H_n, \theta}(\varphi) \widehat{\Phi}(0),
\end{equation}
où $c$ est une constante indépendante de $\varphi$ et $\Phi$.

De plus, ces fonctions zêta globales se décomposent en un produit de fonctions zêta locales, pour $Re(s)$ assez grand, on a
\begin{equation}
J(s, W_\varphi, \Phi) = L^S(s, \pi, \Lambda^2) \prod_{v \in S} J(s, W_v, \Phi_v),
\end{equation}
où $S$ est un ensemble de places suffisamment grand. On obtient alors une factorisation de la période globale sous la forme
\begin{equation}
\mathcal{P}_{H_n, \theta}(\varphi) = \frac{Res_{s=1} L^S(s, \pi, \Lambda^2)}{\widehat{\Phi}^S(0)} \prod_{v \in S} \frac{J(1, W_v, \Phi_v)}{\widehat{\Phi}_v(0)}.
\end{equation}

Supposons que $Res_{s=1} L^S(s, \pi, \Lambda^2) \neq 0$. Le terme $\mathcal{P}_{H_n, \theta}(\varphi)$ ne dépend pas de $\Phi$, on en déduit que les facteurs $\frac{J(1, W_v, \Phi_v)}{\widehat{\Phi}_v(0)}$ ne dépendent pas de $\Phi$. On montrera l'égalité
\begin{equation}
\label{introLien}
J(1, \rho(w_{n,n})\widetilde{W_v}, \widehat{\Phi_v}) = \pm \Phi_v(0) \beta(W_v),
\end{equation} c'est le lemme \ref{zetabeta}, où $w_{n,n} = \sigma_n \begin{pmatrix}
0 & 1 \\
1 & 0
\end{pmatrix}\sigma_n^{-1}$. La forme linéaire $\beta$ nous servira à prouver le théorème \ref{introPlanchExp}.

On commence dans la section \ref{gamma} par prouver une relation sur les facteurs $\gamma$ du carré extérieur. Les sections \ref{seclimite} et \ref{formInv} sont des préliminaires pour le théorème \ref{thPlanch}. On fini dans la section \ref{plancherel} par prouver une formule de Plancherel explicite sur $L^2(H_n \backslash GL_{2n}, \theta)$ et des décompositions de Plancherel abstraite sur $L^2(H_n \backslash GL_{2n}, \theta)$ et $L^2(GL_n \times GL_n \backslash GL_{2n})$.

\subsection{Notations}
\label{notations}
Dans la suite on notera $F$ un corps $p$-adique (sauf dans la section \ref{gamma} où $F$ peut désigner un corps archimédien) et $\psi$ un caractère non trivial de $F$. On note $q_F$ le cardinal du corps résiduel de $F$ et $|.|_F$ (ou simplement |.|) la valeur absolue sur $F$ normalisé par $|\omega|_F = q_F^{-1}$ où $\omega$ est une uniformisante de $F$. 

On notera $G_m$ le groupe $GL_m(F)$ et $PG_m = Z_m(F) \backslash GL_m(F)$. On note $H_n(F)$ le groupe des matrices de la forme $\sigma_n \begin{pmatrix}
1 & X \\
0 & 1
\end{pmatrix}\begin{pmatrix}
g & 0 \\
0 & g
\end{pmatrix} \sigma_n^{-1}$ avec $X \in M_n(F)$ et $g \in GL_n(F)$. L'élément $\sigma_n$ est la matrice associée à la permutation $\bigl(\begin{smallmatrix}
    1 & 2 & \cdots & n & n+1 & n+2 & \cdots & 2n \\
    1 & 3 & \cdots &  2n-1  & 2 & 4 & \cdots & 2n
  \end{smallmatrix}\bigr).$ On note $SO(2m+1)$ la forme déployée du groupe spécial orthogonal sur un espace de dimension $2m+1$. On note $A_n$ le sous-groupe de $G_n$ des matrices diagonales inversibles, $B_n$ le sous groupe des matrices triangulaires supérieures inversibles, $\overline{B}_n$ le sous groupe des matrices triangulaires inférieures inversibles, $N_n$ le sous-groupe de $B_n$ des matrices dont les éléments diagonaux sont $1$, $\overline{N}_n = {}^tN_n$ et $M_n$ l'ensemble des matrices de taille $n \times n$ à coefficients dans $F$. On note $V_n$ le sous-espace de $M_n$ des matrices triangulaires inférieures strictes. On note $U_n$ le groupe des matrices de la forme $\begin{pmatrix}
1_{n-1} & x \\
0 & 1
\end{pmatrix}$ pour $x \in F^{n-1}$ et $P_n = G_{n-1}U_n$ le sous-groupe mirabolique. On note $\delta_{B_n}$ le caractère modulaire de $B_n$. On notera par des lettres gothiques les algèbres de Lie correspondantes et pour $\mathfrak{g}$ une algèbre de Lie $\mathcal{U}(\mathfrak{g})$ désignera l'algèbre enveloppante.

Lorsque $X$ est un espace totalement discontinu, on notera $C^\infty_c(X)$ ou $\mathcal{S}(X)$, l'espace des fonctions localement constante à support compact. Lorsque $G$ est un groupe algébrique réel ou complexe, on note $\mathcal{S}(G)$ l'espace des fonctions $C^\infty$ à décroissance rapide ainsi que toutes ses dérivées tel que défini par Aizenbud et Gourevitch \cite{aizenbud}. De plus, lorsque $\mathbb{A}_K$ est l'anneau des adèles d'un corps de nombres $K$ et $G$ est un groupe algébrique sur $K$, on note $\mathcal{S}(G(\mathbb{A}))$ le produit restreint des espaces $\mathcal{S}(G(K_v))$ lorsque $v$ parcours l'ensemble des places de $K$ i.e. l'ensemble des combinaisons linéaires des fonctions $f = \otimes_v f_v$ avec $f_v \in \mathcal{S}(G(K_v))$ pour tout $v$ et $f_v = \mathbbm{1}_{G(\mathcal{O}_v)}$ sauf pour un nombre fini de $v$, où $\mathcal{O}_v$ est l'anneau des entiers de $K_v$.

Pour $G$ un groupe réductif connexe sur $F$ (dans la suite $G$ sera $GL_{2n}$, $PGL_{2n}$, $SO_{2n+1}$ ou un quotient, sous-groupe de Levi de ces groupes), on note $Temp(G)$ l'ensemble des classes d'isomorphisme de représentations irréductibles tempérées de $G(F)$ et $\Pi_2(G) \subset Temp(G)$ le sous-ensemble des représentations de carré intégrable. 

On peut voir le caractère $\psi$ comme un caractère sur $N_n$ par la formule $\psi(u) = \psi(\sum_{i=1}^{n-1} u_{i,i+1})$,
pour tous $u \in N_n$. On dit qu'une représentation $\pi$ de $G_n$ est générique, si elle est irréductible et $\Hom_{N_n}(\pi, \psi) \neq 0$. Dans ce cas, $\Hom_{N_n}(\pi, \psi)$ est de dimension $1$. Soit $\lambda \in \Hom_{N_n}(\pi, \psi)$ non nul. On note $\mathcal{W}(\pi, \psi)$ le modèle de Whittaker de $\pi$, c'est l'espace des fonctions $W : G_n \rightarrow \mathbb{C}$ telles que $W(g) = \lambda(\pi(g)v)$ pour $v \in V_\pi$, où $V_\pi$ est l'espace sous-jacent à $\pi$. Le groupe $G_n$ agit par translation à droite sur $\mathcal{W}(\pi, \psi)$ et on a $W(ug) = \psi(u)W(g)$ pour $u \in N_n$ et $g \in G_n$. Les représentations tempérées sont génériques.

On note $Z_G$ le centre de $G(F)$ et $A_G$ le tore déployé maximal dans $Z_G$. Soit $M$ un sous-groupe de Levi de $G$ et $\sigma \in \Pi_2(M)$. On note $W(G, M)$ le groupe de Weyl associé au couple $(G,M)$ et $W(G, \sigma)$ le sous-groupe de $W(G, M)$ fixant la classe d'isomorphisme de $\sigma$.
On note $W_F$ (resp. $W_F'$) le groupe de Weil (resp. Weil-Deligne) de $F$ et $\Phi(G) = \{\phi : W_F' \rightarrow {}^LG \text{ admissible dont l'image de $W_F$ est bornée}\}$ l'ensemble des paramètres de Langlands tempérés de $G$ et $Temp(G)/Stab$ le quotient de $Temp(G)$ par la relation d'équivalence $\pi \equiv \pi' \iff \varphi_\pi = \varphi_{\pi'}$, où $\varphi_\pi$ est le paramètre de Langlands associé à $\pi$.

Pour $P=MN$ un sous-groupe parabolique de $G$, on note $i_P^G(\sigma)$ l'induction parabolique normalisée lorsque $\sigma$ est une représentation lisse de $M$ : c'est la représentation régulière à droite de $G$ sur l'espace des fonctions localement constantes $f : G \mapsto \sigma$ qui vérifient $f(mng) = \delta_P(m)^{\frac{1}{2}}\sigma(m)f(g)$ pour tous $m \in M$, $n \in N$ et $g \in G$, où $\delta_P$ est le caractère modulaire de $P$. On note $\rho$ la représentation régulière à droite par $g$ i.e. $\rho(g)f(g_0) = f(g_0g)$ pour tous $g,g_0 \in G$. On notera $\rho(g)$ la translation à droite par $g$ dans d'autres contextes. Lorsque $G = G_n$, $M = G_{n_1} \times ... \times G_{n_k}$ et $P$ est le sous groupe parabolique standard de $G$ de Levi $M$, on note $\pi_1 \times ... \times \pi_k = i_P^G(\pi_1 \boxtimes ... \boxtimes \pi_k)$ pour $\pi_i$ des représentations lisses de $G_{n_i}$. Lorsque $G = SO(2n+1)$, $M=G_{n_1} \times ... \times G_{n_k} \times SO(2m+1)$ et $P$ un sous groupe parabolique de $G$ de Levi $M$, on note $\pi_1 \times ... \times \pi_k \rtimes \sigma_0 = i_P^G(\pi_1 \boxtimes ... \boxtimes \pi_k \boxtimes \sigma_0)$ pour $\pi_i$ des représentations lisses de $G_{n_i}$ et $\sigma_0$ une représentation lisse de $SO(2m+1)$.

 On peut définir une application $\Phi(SO(2m+1)) \rightarrow \Phi(G_{2m})$, rappelons qu'un élément de $\Phi(SO(2m+1))$ est un morphisme admissible $\phi : W_F' \rightarrow {}^L SO(2m+1)$. Or ${}^L SO(2m+1) = Sp_{2m}(\mathbb{C})$, l'application $\Phi(SO(2m+1)) \rightarrow \Phi(G_{2m})$ est définie par l'injection de $Sp_{2m}(\mathbb{C})$ dans $GL_{2m}(\mathbb{C})$. La correspondance de Langlands locale pour $SO(2m+1)$ et pour $GL_{2m}$, nous permettent d'en déduire une application de transfert $T : Temp(SO(2m+1))/Stab \rightarrow Temp(G_{2m})$. 

Dans les mesures de Plancherel, on verra apparaître des termes $|S_{\sigma}|$ pour $\sigma \in Temp(SO(2n+1))$ ou $Temp(PG_{2n})$. On n'explicite pas les ensembles $S_{\sigma}$ et on se contente de donner leur cardinal. Pour $\sigma \in Temp(SO(2n+1))$ sous-représentation de $\pi_1 \times ... \times \pi_l \rtimes \sigma_0$, avec $\pi_i \in \Pi_{2}(G_{n_i})$ et $\sigma_0 \in \Pi_2(SO(2m+1))$, le facteur $|S_{\sigma}|$ est le produit $|S_{\pi_1}|...|S_{\pi_l}||S_{\sigma_0}|$; où $|S_{\sigma_0}|=2^k$ tel que $T(\sigma_0) \simeq \tau_1 \times ... \times \tau_k$ avec $\tau_i \in \Pi_2(G_{m_i})$ et $|S_{\pi_i}| = n_i$.

Pour $\pi \in Temp(G)$ et $r$ une représentation admissible de ${}^L G$, on note $L(s, \pi, r)$ la fonction $L$ associée par la correspondance de Langlands locale et $\gamma(s, \pi, r, \psi)$ le facteur $\gamma$ correspondant. Lorsque $r$ est la représentation standard, on l'omettra. De plus, on note $\gamma^*(0, \pi, r, \psi)$ la régularisation du facteur $\gamma$ en $0$, défini par la relation
\begin{equation}
\gamma^*(0, \pi, r, \psi) = \lim_{s \rightarrow 0^+} \frac{\gamma(s, \pi, r, \psi)}{(s log(q_F))^{n_{\pi,r}}},
\end{equation}
où $n_{\pi,r}$ est l'ordre du zéro de $\gamma(s, \pi, r, \psi)$ en $s=0$.

\subsection{Mesures}
\label{mesures}
On équipe $F$ avec la mesure de Haar $dx$ qui est autoduale par rapport à $\psi$ et $F^\times$ de la mesure de Haar $d^\times x = \frac{dx}{|x|_F}$. 
Pour $m \geq 1$, on équipe $F^m$ de la mesure produit $(dx)^m$ et $(F^\times)^m$ de la mesure $(d^\times x)^m$. On équipe les groupes $M_n$, $U_n$, $N_n$, $\overline{N}_n$ des mesures de Haar "produit des coordonnées". Par exemple, on équipe $M_n$ de la mesure $dX = \prod_{i,j=1}^n dX_{i,j}$ où $dX_{i,j}$ est la mesure de Haar sur $F$ que l'on a fixé précédemment. On équipe $G_n$ de la mesure $dg = |\det g|_F^{-n} \prod_{i,j=1}^n dg_{i,j}$ et $P_n$ la mesure de Haar à droite obtenu comme produit des mesures sur $G_{n-1}$ et sur $U_n$. On équipe $\overline{B}_n$ de la mesure de Haar à droite $\prod_{i=1}^n |b_{i,i}|^{i-n-1} \prod_{\substack{i,j = 1,..,n \\ j \leq i}}db_{i,j}$. On équipe les groupes compact des mesures de Haar de masse totale égale à 1.

On équipe $N_n \backslash G_n$ et $N_n \backslash P_n \simeq N_{n-1} \backslash G_{n-1}$ des mesures quotient. On identifiera ces mesures à des mesures d'un sous-groupe de $G_n$. Plus exactement, $\overline{B}_n$ s'identifie à un ouvert dense de $N_n \backslash G_n$, on obtient alors l'égalité
\begin{equation}
\int_{N_n \backslash G_n} f(g) dg = \int_{\overline{B}_n} f(b)db,
\end{equation}
pour tout $f$ lisse sur $G_n$, $N_n$-équivariante et à support compact modulo $N_n$.

On a l'isomorphisme $P_n \backslash G_n \simeq F^n\backslash{\{0\}}$, on équipe $P_n \backslash G_n$ de la semi-mesure (ou mesure tordue) $dg$ telle que $|\det g|dg$ s'identifie à la mesure $(dx)^n$ sur $F^n$. La mesure tordue sur $P_n \backslash G_n$ n'est pas invariante, ce n'est pas une mesure sur le quotient $P_n \backslash G_n$. Cependant on a la formule d'intégration suivante
\begin{equation}
\int_{G_n} f(g) dg  = \int_{P_n \backslash G_n} \int_{P_n} f(pg) |\det p|^{-1} dp dg,
\end{equation}
 pour tout $f \in \mathcal{S}(G_n)$.
 
Pour $G$ un groupe réductif connexe sur $F$, on fixe un isomorphisme $A_G \simeq (F^\times)^{\dim(A_G)}$ et on équipe $A_G$ de la mesure $(d^\times x)^{dim(A_G)}$ provenant de l'isomorphisme avec $(F^\times)^{\dim(A_G)}$.

Décrivons le choix de la normalisation d'une mesure sur Temp(G). Soit $M$ un sous-groupe de Levi de $G$ et $\sigma \in \Pi_2(M)$. Soit $\widehat{A_M}$ le dual unitaire de $A_M$ et $d\widetilde{\chi}$ la mesure de Haar duale de celle de $A_M$. On équipe alors $\widehat{A_M}$ de la mesure $d\chi$ définie par
\begin{equation}
d\chi = \gamma^*(0, 1, \psi)^{-dim(A_M)}d\widetilde{\chi}.
\end{equation}

La mesure $d\chi$ est indépendante du caractère $\psi$. 

On note $X^*(M)$ le groupe des caractères algébriques de $M$, on dispose alors d'une application $\chi \otimes \lambda \in X^*(M) \otimes i\mathbb{R} \mapsto \sigma \otimes \chi_\lambda \in \Pi_2(M)$ où $\chi_\lambda(g) = |\chi(g)|^\lambda$. On définit alors une base de voisinage de $\sigma$ dans $\Pi_2(M)$ comme l'image d'une base de voisinage de $0$ dans $X^*(M) \otimes i\mathbb{R}$.

Il existe une unique mesure $d\sigma$ sur $\Pi_2(M)$ telle que l'isomorphisme local $\sigma \in \Pi_2(M) \mapsto \omega_{\sigma} \in \widehat{A_M}$ préserve localement les mesures, où $\omega_\sigma$ est le caractère central de $\sigma$ restreint à $A_M$. Soit $P$ un sous groupe parabolique de $G$ de Levi $M$. On définit alors la mesure $d\pi$ sur $Temp(G)$ localement autour de $\pi \simeq i_P^G(\sigma)$ par la formule
\begin{equation}
d\pi  = |W(G, M)|^{-1} (i_P^G)_* d\sigma,
\end{equation}
où $(i_P^G)_* d\sigma$ est la mesure $d\sigma$ poussée en avant en une mesure sur $Temp(G)$ par l'application $i_P^G$. Cette mesure ne dépend pas du choix du groupe parabolique. La mesure $d\pi$ est choisie pour vérifier la relation \ref{mesurePlanch}.

\subsection{Résultats}
\label{resultats}

Soit $F$ un corps $p$-adique et $\psi$ un caractère non trivial de $F$. Rappelons que l'on note $H_n(F)$ le groupe des matrices de la forme $\sigma_n \begin{pmatrix}
1 & X \\
0 & 1
\end{pmatrix}\begin{pmatrix}
g & 0 \\
0 & g
\end{pmatrix} \sigma_n^{-1}$ avec $X \in M_n(F)$ et $g \in GL_n(F)$. L'élément $\sigma_n$ est la matrice associée à la permutation $\bigl(\begin{smallmatrix}
    1 & 2 & \cdots & n & n+1 & n+2 & \cdots & 2n \\
    1 & 3 & \cdots &  2n-1  & 2 & 4 & \cdots & 2n
  \end{smallmatrix}\bigr).$
De plus, $\theta$ est le caractère sur $H_n(F)$ qui envoie $\sigma_n \begin{pmatrix}
1 & X \\
0 & 1
\end{pmatrix}\begin{pmatrix}
g & 0 \\
0 & g
\end{pmatrix} \sigma_n^{-1}$ sur $\psi(Tr(X))$.
Le résultat principal est le
\begin{theoreme}
On a un isomorphisme de représentations unitaires
\begin{equation}
L^2(H_n(F) \backslash GL_{2n}(F), \theta) \simeq \int^{\oplus}_{Temp(SO(2n+1)(F))/Stab} T(\sigma) d\sigma,
\end{equation}
où $T : {Temp(SO(2n+1))\slash{\sim}} \rightarrow Temp(GL_{2n})$ est l'application de transfert provenant de la correspondance de Langlands locale.
\end{theoreme}

De l'isomorphisme $L^2(GL_n(F) \times GL_n(F) \backslash GL_{2n}(F)) \simeq L^2(H_n(F) \backslash G_{2n}(F), \theta)$ $GL_{2n}$-équivariant (lemme \ref{lemmeiso}), on en déduit le
\begin{theoreme}
On a un isomorphisme de représentations unitaires
\begin{equation}
L^2(GL_n(F) \times GL_n(F) \backslash GL_{2n}(F)) \simeq \int^{\oplus}_{Temp(SO(2n+1)(F))/Stab} T(\sigma) d\sigma.
\end{equation}
\end{theoreme}

Rappelons que ces deux décompositions de Plancherel abstraites sont obtenues en prouvant une formule de Plancherel explicite sur $H_n \backslash G_n$.
\begin{theoreme}
On pose $Y_n = H_n \backslash G_{2n}$. Soient $\varphi_1, \varphi_2 \in C^\infty_c(Y_n, \theta)$. Il existe $f_1, f_2 \in C^\infty_c(G_{2n})$ telles que que $\varphi_i = \varphi_{f_i}$ pour $i = 1,2$. On pose
\begin{equation}
(\varphi_1, \varphi_2)_{Y_n, \pi} = \sum_{W \in \mathcal{B}(\pi, \psi)} \beta(W) \overline{\beta(\pi(\overline{f}_1)\pi(f_2^\vee)W)},
\end{equation}
pour tout $\pi \in T(Temp(SO(2n+1)))$. Les notations $\beta$ et $\mathcal{B}(\pi, \psi)$ ont été introduite dans l'introduction. On a alors
\begin{equation}
(\varphi_1, \varphi_2)_{L^2(Y_n, \theta)} = \int_{Temp(SO(2n+1))\slash{\sim}} (\varphi_1, \varphi_2)_{Y_n, T(\sigma)} \frac{|\gamma^*(0, \sigma, Ad, \psi)|}{|S_\sigma|}d\sigma,
\end{equation}
Le facteur $\frac{|\gamma^*(0, \sigma, Ad, \psi)|}{|S_\sigma|}$ est défini dans la section \ref{notations}.
\end{theoreme}

La mesure $\frac{|\gamma^*(0, \sigma, Ad, \psi)|}{|S_\sigma|}d\sigma$ n'est rien d'autre que la mesure de Plancherel pour SO(2n+1). En effet, la mesure de Plancherel d'un groupe réductif $p$-adique $G$ a été calculée par Waldspurger et Harish-Chandra \cite{waldspurger} sous la forme
\begin{equation}
d\mu_G(\sigma) = d(\sigma)j(\sigma)^{-1}d\sigma,
\end{equation}
où $d(\sigma)$ est le degré formel de $\sigma$ et $j(\sigma)$ est un scalaire produit d'opérateurs d'entrelacements (voir \cite{waldspurger}). Le degré formel pour $SO(2n+1)$ a été calculé par Ichino-Lapid-Mao \cite{ichino} et le facteur $j$ pour $SO(2n+1)$ découle de la normalisation des opérateurs d'entrelacements d'Arthur \cite{arthur}.
Finalement, on obtient que la mesure de Plancherel pour $SO(2n+1)$ est
\begin{equation}
d\mu_{SO(2n+1)}(\sigma) = \frac{|\gamma^*(0, \sigma, Ad, \psi)|}{|S_\sigma|}d\sigma.
\end{equation}
On renvoie à l'article de Beuzart-Plessis \cite[Proposition 2.13.2]{beuzart-plessis} pour l'analogue de ce résultat pour les groupes unitaires.

Pour finir, au cours de la preuve de la formule de Plancherel explicite, on aura besoin d'une égalité sur des facteurs gamma définie de deux manière différentes. On prouve le
\begin{theoreme}
\label{egalitegamma}
Soit $\pi$ une représentation tempérée irréductible de $GL_{2n}(F)$. On note $\gamma^{JS}(s,\pi,\Lambda^2,\psi)$ le facteur gamma de Jacquet-Shalika, voir section \ref{gamma} et $\gamma(s, \pi, \Lambda^2, \psi)$ le facteur gamma d'Artin défini via la correspondance de Langlands.
Alors il existe une constante $c(\pi)$ de module 1 telle que pour tout $s \in \mathbb{C}$, on ait
\begin{equation}
\gamma^{JS}(s,\pi,\Lambda^2,\psi)=c(\pi)\gamma(s,\pi,\Lambda^2,\psi).
\end{equation}
\end{theoreme}

\subsection{Remerciments}

Je tiens à remercier Raphaël Beuzart-Plessis qui m'a proposé ce sujet et a su me faire confiance. Il a fait preuve d'une énorme patience lors des phases de recherche et de rédaction. Je tiens à souligner qu'il a pris le temps de me relire un nombre incalculable de fois. Ses remarques et corrections sont la raison pour laquelle cet article n'est pas totalement illisible.

\section{Facteurs $\gamma$ du carré extérieur}
\label{gamma}
Dans cette partie $F$ désigne un corps local de caractéristique $0$ et $\psi$ un caractère non trivial de $F$. Soit $\pi$ une représentation tempérée irréductible de $GL_{2n}(F)$. Jacquet et Shalika ont défini une fonction L du carré extérieur $L_{JS}(s, \pi, \Lambda^2)$ par des intégrales notées $J(s, W, \phi)$, où $W \in \mathcal{W}(\pi, \psi)$ est un élément du modèle de Whittaker de $\pi$ et $\phi \in \mathcal{S}(F^n)$. Matringe a prouvé que, lorsque $F$ est non archimédien, ces intégrales $J(s,W,\phi)$ vérifient une équation fonctionnelle, ce qui permet de définir des facteurs $\gamma$, que l'on note $\gamma^{JS}(s,\pi,\Lambda^2,\psi)$. 

On montre que l'on a encore une équation fonctionnelle lorsque $F$ est archimédien et que les facteurs $\gamma$ sont égaux à une constante de module 1 prés à ceux définis par Shahidi, que l'on note $\gamma^{Sh}(s,\pi,\Lambda^2,\psi)$. Plus exactement, il existe des constantes $c^{Sh}(\pi)$ et $c(\pi)$ de module 1, telles que
\begin{equation}
\gamma^{JS}(s,\pi,\Lambda^2,\psi)=c^{Sh}(\pi)\gamma^{Sh}(s,\pi,\Lambda^2,\psi) = c(\pi)\gamma(s, \pi, \Lambda^2, \psi),
\end{equation}
pour tout $s \in \mathbb{C}$. La dernière égalité est une conséquence de l'égalité des facteurs gamma de Shahidi et d'Artin pour le carré extérieur à une racine de l'unité prés prouvée par Henniart \cite{henniart}. La preuve se fait par une méthode de globalisation, on considère $\pi$ comme une composante locale d'une représentation automorphe cuspidale.

\subsection{Préliminaires}

\subsubsection{Théorie locale}
Les intégrales $J(s, W, \phi)$ sont définies par
 \begin{equation}
\int_{N_n\backslash{G_n}} \int_{V_n} W\left(\sigma_n \begin{pmatrix}
1 & X \\
0 & 1
\end{pmatrix}\begin{pmatrix}
g & 0 \\
0 & g
\end{pmatrix} \sigma_n^{-1} \right)dX\phi(e_ng)|\det g|^s dg
 \end{equation}
pour tous $W \in \mathcal{W}(\pi, \psi)$, $\phi \in \mathcal{S}(F^n)$ et $s \in \mathbb{C}$. L'espace $V_n$ est l'ensemble des matrices triangulaires inférieures strictes, on l'équipe de la mesure de Haar $dX = \prod_{1 \leq j < i \leq n} dX_{i,j}$. L'élément $\sigma_n$ est la matrice associée à la permutation $\bigl(\begin{smallmatrix}
    1 & 2 & \cdots & n & n+1 & n+2 & \cdots & 2n \\
    1 & 3 & \cdots &  2n-1  & 2 & 4 & \cdots & 2n
  \end{smallmatrix}\bigr).$
  
  Jacquet et Shalika ont démontré que ces intégrales convergent pour $Re(s)$ suffisamment grand, plus exactement, on dispose de la
  \begin{proposition}[Jacquet-Shalika \cite{jacquet-shalika}]
  Il existe $\eta > 0$ tel que les intégrales $J(s, W, \phi)$ convergent absolument pour $Re(s) > 1 - \eta$.
  \end{proposition}
  
  Kewat \cite{kewat} montre, lorsque $F$ est p-adique, que ce sont des fractions rationnelles en $q^{s}$ où $q$ est le cardinal du corps résiduel de $F$. On aura aussi besoin d'avoir le prolongement méromorphe de ces intégrales lorsque $F$ est archimédien et d'un résultat de non annulation.
  \begin{proposition}[Belt \cite{belt}, Matringe \cite{matringe}]
  \label{nonzero}
  Fixons $s_0 \in \mathbb{C}$. Il existe $W \in \mathcal{W}(\pi, \psi)$ et $\phi \in \mathcal{S}(F^n)$ tels que $J(s,W,\phi)$ admet un prolongement méromorphe à tout le plan complexe et ne s'annule pas en $s_0$. Si $F=\mathbb{R}$ ou $\mathbb{C}$, le point $s_0$ peut éventuellement être un pôle. Si $F$ est $p$-adique, on peut choisir $W$ et $\phi$ tels que $J(s, W, \phi)$ soit entière.
  \end{proposition}
  
  Lorsque la représentation est non-ramifiée, on peut représenter la fonction L du carré extérieur obtenue par la correspondance de Langlands locale, que l'on note $L(s, \pi, \Lambda^2)$, (qui est égale à celle obtenue par la méthode de Langlands-Shahidi) par ces intégrales.
  \begin{proposition}[Jacquet-Shalika \cite{jacquet-shalika}]
  \label{calculnr}
  Supposons que $F$ est $p$-adique, le conducteur de $\psi$ est l'anneau des entiers $\mathcal{O}_F$ de $F$. Soit $\pi$ une représentation générique non ramifiée de $GL_{2n}(F)$. On note $\phi_0$ la fonction caractéristique de $\mathcal{O}_F^n$ et $W_0 \in \mathcal{W}(\pi, \psi)$ l'unique fonction de Whittaker invariante par $GL_{2n}(\mathcal{O}_F)$ et qui vérifie $W(1)=1$. Alors
   \begin{equation}
   J(s,W_0,\phi_0) = L(s, \pi, \Lambda^2).
    \end{equation}
  \end{proposition}
  
  Pour finir cette section, on énonce l'équation fonctionnelle démontrée par Matringe lorsque $F$ est un corps $p$-adique. Plus précisément, on a la
 \begin{proposition}[Matringe \cite{matringe}]
 \label{funcloc}
 Supposons que $F$ est un corps $p$-adique et $\pi$ générique. Il existe un monôme $\epsilon^{JS}(s,\pi,\Lambda^2,\psi)$ en $q^s$ ou $q^{-s}$, tel que pour tous $W \in \mathcal{W}(\pi,\psi)$ et $\phi \in \mathcal{S}(F^n)$, on ait
 \begin{equation}
 \epsilon^{JS}(s, \pi, \Lambda^2, \psi) \frac{J(s,W,\phi)}{L(s,\pi,\Lambda^2)}  = \frac{J(1-s,\rho(w_{n,n})\widetilde{W},\widehat{\phi})}{L(1-s,\widetilde{\pi},\Lambda^2)},
 \end{equation}
 où $\rho$ désigne la translation à droite, $\widehat{\phi} = \mathcal{F}_\psi(\phi)$ est la transformée de Fourier de $\phi$ par rapport au caractère $\psi$ définie par
 \begin{equation}
 \mathcal{F}_\psi(\phi)(y) = \int_{F^n} \phi(x)\psi(\sum_{i=1}^n x_i y_i)dx
 \end{equation}
 pour tout $y \in F^n$ et $\widetilde{W} \in \mathcal{W}(\widetilde{\pi}, \bar{\psi})$ est la fonction de Whittaker définie par $\widetilde{W}(g) = W(w_n(g^t)^{-1})$ pour tout $g \in GL_{2n}(F)$, avec $w_n$ la matrice associée à la permutation $\bigl(\begin{smallmatrix}
    1 & \cdots & 2n  \\
    2n & \cdots &  1 
  \end{smallmatrix}\bigr)$
  et
 $w_{n,n} = \sigma_n \begin{pmatrix}
0 & 1_n \\
1_n & 0
\end{pmatrix} \sigma_n^{-1}$. On définit alors le facteur $\gamma$ de Jacquet-Shalika par la relation
\begin{equation}
\gamma^{JS}(s,\pi,\Lambda^2,\psi)  = \epsilon^{JS}(s,\pi,\Lambda^2,\psi)\frac{L(1-s,\widetilde{\pi},\Lambda^2)}{L(s,\pi,\Lambda^2)}.
\end{equation}
 \end{proposition}
 
  \subsubsection{Théorie globale}
  La méthode que l'on utilise est une méthode de globalisation. Essentiellement, on verra $\pi$ comme une composante locale d'une représentation automorphe cuspidale. Pour ce faire, on aura besoin de l'équivalent global des intégrales $J(s, W, \phi)$.
  
  Soit $K$ un corps de nombres et $\psi_\mathbb{A}$ un caractère non trivial de $\mathbb{A}_K/K$. Soit $\Pi$ une représentation automorphe cuspidale irréductible de $GL_{2n}(\mathbb{A}_K)$. Pour $\varphi \in \Pi$, on considère
  \begin{equation}
  W_\varphi(g) = \int_{N_{2n}(K)\backslash{N_{2n}(\mathbb{A}_K)}} \varphi(ug)\psi_\mathbb{A}(u)^{-1}du
  \end{equation}
  la fonction de Whittaker associée. On considère $\psi_\mathbb{A}$ comme un caractère de $N_{2n}(\mathbb{A}_K)$ en posant $\psi_\mathbb{A}(u) = \psi_\mathbb{A}(\sum_{i=1}^{2n-1} u_{i,i+1})$. Pour $\Phi \in \mathcal{S}(\mathbb{A}_K^n)$ une fonction de Schwartz, on note $J(s, W_\varphi, \Phi)$ l'intégrale
  \begin{equation}
\int_{N_n\backslash{G_n}} \int_{V_n} W_\varphi \left(\sigma_n \begin{pmatrix}
1 & X \\
0 & 1
\end{pmatrix}\begin{pmatrix}
g & 0 \\
0 & g
\end{pmatrix} \sigma_n^{^-1}\right)dX\Phi(e_ng)|\det g|^s dg
 \end{equation}
 où l'on note $G_n$ le groupe $GL_n(\mathbb{A}_K)$, $B_n$ le sous groupe des matrices triangulaires supérieures, $N_n$ le sous-groupe de $B_n$ des matrices dont les éléments diagonaux sont $1$ et $M_n$ l'ensemble des matrices de taille $n \times n$ à coefficients dans $\mathbb{A}_K$.
 
  Finissons cette section par l'équation fonctionnelle globale démontrée par Jacquet et Shalika.
 \begin{proposition}[Jacquet-Shalika \cite{jacquet-shalika}]
 \label{funcglob}
 Les intégrales $J(s, W_\varphi, \Phi)$ convergent absolument pour $Re(s)$ suffisamment grand. De plus, $J(s, W_\varphi, \Phi)$ admet un prolongement méromorphe à tout le plan complexe et vérifie l'équation fonctionnelle suivante
 \begin{equation}
 J(s,W_\varphi,\Phi)=J(1-s, \rho(w_{n,n})\widetilde{W}_\varphi, \widehat{\Phi}),
 \end{equation}
 où $\widetilde{W}_\varphi(g) = W_\varphi(w_n(g^t)^{-1})$ et $\widehat{\Phi}$ est la transformée de Fourier de $\Phi$ par rapport au caractère $\psi_\mathbb{A}$.
 \end{proposition}
 
 Comme on peut s'y attendre, les intégrales globales sont reliées aux intégrales locales. Plus exactement, si $W_\varphi=\prod_v W_v$ et $\Phi = \prod_v \Phi_v$, où $v$ décrit les places de $K$, on a
 \begin{equation}
 J(s,W_\varphi,\Phi)=\prod_v J(s, W_v, \Phi_v),
 \end{equation}
 pour $Re(s)$ suffisamment grand.
 
 \subsubsection{Globalisation}
 
 Comme la preuve se fait par globalisation, la première chose à faire est de trouver un corps de nombres dont $F$ est une localisation. On dispose du
 \begin{lemme}[Kable \cite{kable}]
 \label{corpsglobal}
 Supposons que $F$ est un corps $p$-adique. Il existe un corps de nombres $k$ et une place $v_0$ telle que $k_{v_0} = F$, où $v_0$ est l'unique place de $k$ au dessus de $p$.
 \end{lemme}
 
Rappelons la topologie que l'on a défini sur $Temp(GL_{2n}(F))$. Soit $M$ un sous-groupe de Levi de $GL_{2n}(F)$, $P$ un parabolique de Levi $M$ et $\sigma \in \Pi_2(M)$. La classe d'isomorphisme de l'induction parabolique normalisé $i^G_P(\sigma)$ est indépendante du parabolique $P$ et on la notera $i^G_M(\sigma)$. On note $X^*(M)$ le groupe des caractères algébriques de $M$, on dispose alors d'une application $\chi \otimes \lambda \in X^*(M) \otimes i\mathbb{R} \mapsto i^G_M(\sigma \otimes \chi_\lambda) \in Temp(GL_{2n}(F))$ où $\chi_\lambda(g) = |\chi(g)|^\lambda$. On définit alors une base de voisinage de $i^G_M(\sigma)$ dans $Temp(GL_{2n}(F))$ comme l'image d'une base de voisinage de $0$ dans $X^*(M) \otimes i\mathbb{R}$.
 
 Cette topologie sur $Temp(GL_{2n}(F))$ nous permet d'énoncer le résultat principal dont on aura besoin pour la méthode de globalisation.
 \begin{proposition}[Beuzart-Plessis {\cite[Théorème 3.7.1]{beuzart-plessis}}, Finis-Lapid-Muller \cite{flm}]
 \label{globalisation}
 Soient $k$ un corps de nombres, $v_0,v_1$ deux places distinctes de $k$ avec $v_1$ non archimédienne. Soit $U$ un ouvert de $Temp(GL_{2n}(k_{v_0}))$. Alors il existe une représentation automorphe cuspidale irréductible $\Pi$ de $GL_{2n}(\mathbb{A}_k)$ telle que $\Pi_{v_0} \in U$ et $\Pi_v$ est non ramifiée pour toute place non archimédienne $v \not \in \{v_0,v_1\}$.
 \end{proposition}
 
 Cette proposition de Beuzart-Plessis est une conséquence des résultats de Finis-Lapid-Muller \cite{flm}.
 
 \subsubsection{Fonctions tempérées}
 \label{fonctiontemperees}
 On aura besoin dans la suite de connaître la dépendance que $J(s, W, \phi)$ lorsque l'on fait varier la représentation $\pi$. Pour ce faire, on introduit la notion de fonction tempérée et on étend la définition de $J(s,W,\phi)$ pour ces fonctions tempérées.
 
On note $K_{2n}$ le sous-groupe compact maximal de $GL_{2n}(F)$ défini par $K_{2n} = GL_{2n}(\mathcal{O}_F)$ lorsque $F$ est p-adique et $K_{2n} = \{ g \in GL_{2n}(F), g \overline{g}^t = I_n\}$ lorsque $F = \mathbb{R}$ ou $\mathbb{C}$, où $\overline{g}$ est la conjuguée complexe.

L'espace des fonctions tempérées $C^w(N_{2n}(F)\backslash{GL_{2n}(F)}, \psi)$ est l'espace des fonctions $f : GL_{2n}(F) \rightarrow \mathbb{C}$ telles que $f(ng) = \psi(n)f(g)$ pour tous $n \in N_{2n}(F)$ et $g \in GL_{2n}(F)$, on impose les conditions suivantes :
 \begin{itemize}
 \item si $F$ est $p$-adique, $f$ est invariante à droite par un sous-groupe compact ouvert $K$ et il existe $d > 0$ et $C > 0$ tels que
 \begin{equation}
 \label{condtempp}
 |f(nak)| \leq C \delta_{B_{2n}}(a)^{\frac{1}{2}} \log(||a||)^d,
 \end{equation}
 où $||a|| = 1 + max(|a_{i,i}|, |a_{i,i}|^{-1})$, pour tous $n \in N_{2n}(F)$, $a \in A_{2n}(F)$ et $k \in K_{2n}$;
 \item si $F$ est archimédien, $f$ est $C^\infty$ et il existe $d > 0$ tel que pour tout $u \in \mathcal{U}(\mathfrak{gl}_{2n}(F))$, il existe $C > 0$ tel que
 \begin{equation}
 \label{condtemparch}
 |(R(u)f)(nak)| \leq C \delta_{B_{2n}}(a)^{\frac{1}{2}} \log(||a||)^d,
 \end{equation}
 pour tous $n \in N_{2n}(F)$, $a \in A_{2n}(F)$, $k \in K_{2n}$.
 \end{itemize}
 
Lorsque $F$ est $p$-adique, pour $d > 0$ et $K$ un sous-groupe compact ouvert de $GL_{2n}$, on note $C^w_d(N_{2n}(F)\backslash{GL_{2n}(F)}, \psi)^K$ l'espace des fonctions tempérées invariante à droite par $K$ et vérifiant la condition \ref{condtempp}. On munit $C^w_d(N_{2n}(F)\backslash{GL_{2n}(F)}, \psi)^K$ de la topologie provenant de la norme $\sup\limits_{a \in A_{2n}, k \in K_{2n}} \frac{|f(ak)|}{\delta_{B_{2n}}(a)^{\frac{1}{2}} \log(||a||)^d}$ qui en fait un espace de Banach. On munit alors $C^w(N_{2n}(F)\backslash{GL_{2n}(F)}, \psi) = \bigcup_{d,K}C^w_d(N_{2n}(F)\backslash{GL_{2n}(F)}, \psi)^K$ de la topologie limite inductive.

Lorsque $F$ est archimédien, pour $d > 0$, on note $C^w_d(N_{2n}(F)\backslash{GL_{2n}(F)}, \psi)$ l'espace des fonctions tempérées vérifiant la condition \ref{condtemparch}. On munit $C^w_d(N_{2n}(F)\backslash{GL_{2n}(F)}, \psi)$ de la topologie provenant des semi-normes $p_u(f) = \sup\limits_{a \in A_{2n}, k \in K_{2n}} \frac{|(R(u)f)(ak)|}{\delta_{B_{2n}}(a)^{\frac{1}{2}} \log(||a||)^d}$ pour $u \in \mathcal{U}(\mathfrak{gl}_{2n}(F))$. On munit alors $C^w(N_{2n}(F)\backslash{GL_{2n}(F)}, \psi) = \bigcup_d C^w_d(N_{2n}(F)\backslash{GL_{2n}(F)}, \psi)$ de la topologie limite inductive.
 
On rappelle la majoration des fonctions tempérées sur la diagonale,
\begin{lemme}[{\cite[Lemme 2.4.3]{beuzart-plessis}}]
\label{majtemp}
Soit $W \in C^w(N_{2n}(F)\backslash{GL_{2n}(F)}, \psi)$. Il existe $d > 0$ tel que pour tout $N \geq 1$, il existe $C > 0$ tel que
\begin{equation}
|W(bk)| \leq C\prod_{i=1}^{2n-1} (1 + |\frac{b_i}{b_{i+1}}|)^{-N}\delta_{B_{2n}}(b)^{\frac{1}{2}}\log(||b||)^d,
\end{equation}
pour tous $b \in A_{2n}(F)$ et $k \in K_{2n}$.
\end{lemme}

\begin{lemme}[{\cite[Lemme 2.4.4]{beuzart-plessis}}]
\label{convergenceAn}
Pour tout $C > 0$, il existe $N$ tel que pour tous $s$ vérifiant $0 < Re(s) < C$ et $d > 0$, l'intégrale
\begin{equation}
\int_{A_n} \prod_{i=1}^{n-1} (1+|\frac{a_i}{a_{i+1}}|)^{-N}(1+|a_n|)^{-N}\log(||a||)^d|\det a|^s da
\end{equation}
converge absolument.
\end{lemme}

On étend la définition des intégrales $J(s, W, \phi)$ aux fonctions tempérées $W$, on montre maintenant la convergence de ces intégrales dans le
\begin{lemme}
\label{convtemp}
Pour $W \in C^w(N_{2n}(F)\backslash{GL_{2n}(F)}, \psi)$ et $\phi \in \mathcal{S}(F^n)$, l'intégrale $J(s, W, \phi)$ converge absolument pour tout $s \in \mathbb{C}$ vérifiant $Re(s) > 0$. De plus, pour tous $\phi \in \mathcal{S}(F^n)$ et $s \in \mathbb{C}$ tels que $Re(s) > 0$, la forme linéaire $W \in C^w(N_{2n}(F)\backslash{GL_{2n}(F)}, \psi) \mapsto J(s, W,  \phi)$ est continue.
\end{lemme}
 
 \begin{proof}
 Soit $G_n = N_nA_nK_n$ la décomposition d'Iwasawa de $G_n$. Il suffit de montrer la convergence de l'intégrale
 \begin{equation}
 \int_{A_n} \int_{K_n} \int_{V_n} \left|W\left(\sigma_n \begin{pmatrix}
1 & X \\
0 & 1
\end{pmatrix}\begin{pmatrix}
ak & 0 \\
0 & ak
\end{pmatrix} \sigma_n^{-1}\right) \phi(e_nak)\right| dX dk \left|\det a\right|^{Re(s)} \delta_{B_n}^{-1}(a) da.
 \end{equation}
 
 On pose $u_X = \sigma_n \begin{pmatrix}
1 & X \\
0 & 1
\end{pmatrix} \sigma_n^{-1}$, ce qui nous permet d'écrire
\begin{equation}
\sigma_n \begin{pmatrix}
1 & X \\
0 & 1
\end{pmatrix}\begin{pmatrix}
a & 0 \\
0 & a
\end{pmatrix} = b u_{a^{-1}Xa} \sigma_n,
\end{equation}
 où $b=diag(a_1,a_1,a_2,a_2,...)$. On effectue le changement de variable $X \mapsto aXa^{-1}$, l'intégrale devient alors
 \begin{equation}
\int_{A_n} \int_{K_n} \int_{V_n} \left|W\left(b u_X \sigma_n \begin{pmatrix}
k & 0 \\
0 & k
\end{pmatrix} \sigma_n^{-1} \right)\phi(e_nak)\right|dX dk |\det a|^{Re(s)} \delta_{B_n}^{-2}(a) da.
 \end{equation}
 
 On écrit $u_X = n_Xt_Xk_X$ la décomposition d'Iwasawa de $u_X$ et on pose $k_\sigma = \sigma_n \begin{pmatrix}
k & 0 \\
0 & k
\end{pmatrix} \sigma_n^{-1}$. Le lemme \ref{majtemp} donne alors
 \begin{equation}
 |W(bt_Xk_Xk_\sigma)| \leq C \prod_{i=1}^{2n-1} (1+  |\frac{t_jb_j}{t_{j+1}b_{j+1}}|)^{-2N} \delta_{B_{2n}}^{\frac{1}{2}}(bt_X)\log(||bt_X||)^d,
 \end{equation}
 où $t_X = diag(t_1, ..., t_{2n})$.
 
 On aura besoin d'inégalités prouvées par Jacquet et Shalika concernant les $t_j$. On dispose de la
 \begin{proposition}[Jacquet-Shalika {\cite[Proposition 4]{jacquet-shalika}}]
 \label{maj_tj}
 On a $|t_k| \geq 1$ lorsque $k$ est impair et $|t_k| \leq 1$ lorsque $k$ est pair. En particulier, $|\frac{t_j}{t_{j+1}}| \geq 1$ lorsque $j$ est impair et $|\frac{t_j}{t_{j+1}}| \leq 1$ lorsque $j$ est pair.
 \end{proposition}
 
 On combine alors cette proposition avec le fait que $\frac{b_j}{b_{j+1}} = 1$ lorsque $j$ est impair et $\frac{b_j}{b_{j+1}} = \frac{a_\frac{j}{2}}{a_{\frac{j}{2}+1}}$ lorsque $j$ est pair. Ce qui nous permet de majorer $(1+  |\frac{t_jb_j}{t_{j+1}b_{j+1}}|)^{-2N}$ par $|\frac{t_j}{t_{j+1}}|^{-2N}$ lorsque $j$ est impair et par $|\frac{t_j}{t_{j+1}}|^{-N}(1+  |\frac{a_{j/2}}{a_{j/2+1}}|)^{-N}$ lorsque $j$ est pair. Ce qui donne
 \begin{equation}
 \begin{split}
 |W(bt_Xk_Xk_\sigma)| &\leq C \prod_{j=1}^{2n-1}|\frac{t_j}{t_{j+1}}|^{-N} \prod_{j=1, \text{j impair}}^{2n-1} |\frac{t_j}{t_{j+1}}|^{-N} \prod_{i=1}^{n-1} (1 + |\frac{a_i}{a_{i+1}}|)^{-N} \delta_{B_{2n}}^{\frac{1}{2}}(bt_X)\log(||bt_X||)^d \\
 &\leq C \prod_{j=1, \text{j impair}}^{2n-1} |\frac{t_j}{t_{j+1}}|^{-N}  \prod_{i=1}^{n-1} (1 + |\frac{a_i}{a_{i+1}}|)^{-N} \delta_{B_{2n}}^{\frac{1}{2}}(bt_X)\log(||bt_X||)^d,
  \end{split}
 \end{equation}
 puisque $\prod_{j=1}^{2n-1}|\frac{t_j}{t_j+1}|^{-N} = |\frac{t_1}{t_{2n}}|^{-N} \leq 1$ d'après la proposition \ref{maj_tj}.
 
 De plus, encore d'après la proposition \ref{maj_tj}, on a
 \begin{equation}
 \prod_{j=1, \text{j impair}}^{2n-1} |\frac{t_j}{t_{j+1}}|^{-N} \leq \prod_{j=1, \text{j impair}}^{2n-1} \frac{1}{|t_j|^N}.
 \end{equation}
 
 Pour finir, on aura besoin de la
 \begin{proposition}[Jacquet-Shalika {\cite[Proposition 5]{jacquet-shalika}}]
 \label{maj_mX}
 Pour $X \in Lie(\overline{N}_n)$, on pose $||X|| = sup_{i,j} |X_{i,j}|$. On pose $m(X) = \sqrt{1+||X||}$ lorsque $F$ est archimédien et $m(X) = sup(1, ||X||)$ lorsque $F$ est non-archimédien. Il existe une constante $\alpha > 0$ telle que
 pour tout $X \in Lie(\overline{N}_n)$, on ait
 \begin{equation}
\prod_{j=1, \text{j impair}}^{2n-1} |t_j| \geq m(X)^\alpha
\end{equation}
 \end{proposition}
 
Grâce à cette proposition, on obtient la majoration
 \begin{equation}
 |W(bt_Xk_Xk_\sigma)| \leq C m(X)^{-\alpha N}  \prod_{i=1}^{n-1} (1 + |\frac{a_i}{a_{i+1}}|)^{-N} \delta_{B_{2n}}^{\frac{1}{2}}(bt_X)\log(||bt_X||)^d.
 \end{equation}
 
 D'autre part, il existe $C' > 0$ tel que
 \begin{equation}
 |\phi(e_nak)| \leq C'(1+|a_n|)^{-N}.
 \end{equation}
 
 L'intégrale $J(s, W, \phi)$ est alors majorée (à une constante près) par le maximum du produit des intégrales
 \begin{equation}
 \int_{V_n} m(X)^{-\alpha N} \delta_{B_{2n}}^{\frac{1}{2}}(t_X)\log(||t_X||)^{d-j} dX
 \end{equation}
 et
 \begin{equation}
 \int_{A_n}  \prod_{i=1}^{n-1} (1+ |\frac{a_i}{a_{i+1}}|)^{-N} (1+|a_n|)^{-N}\log(||b||)^j |\det a|^{Re(s)} \delta_{B_{2n}}^{\frac{1}{2}}(b)\delta_{B_n}^{-2}(a) da,
 \end{equation}
 pour $j$ compris entre $0$ et $d$. La première intégrande est majorée par $m(X)^{-\alpha N + c}$, où 
 $c$ est une constante, on en déduit que la première intégrale converge pour $N$ assez grand et la deuxième pour $N$ assez grand lorsque $Re(s) > 0$ d'après le lemme \ref{convergenceAn} où l'on a utilisé la relation $\delta_{B_{2n}}^{\frac{1}{2}}(b) = \delta_{B_n}^2(a)$.
 \end{proof}
  
 \subsection{Facteurs $\gamma$}
 
 Dans cette partie, on prouve l'égalité entre les facteurs $\gamma^{JS}(., \pi, \Lambda^2, \psi)$ et $\gamma^{Sh}(., \pi, \Lambda^2, \psi)$ à une constante (dépendant de $\pi$) de module 1 près.
 
 On commence à montrer cette égalité pour les facteurs $\gamma$ archimédiens. Pour le moment, les résultats connus ne nous donnent même pas l'existence du facteur $\gamma^{JS}(.,\pi, \Lambda^2, \psi)$ dans le cas archimédien, ce sera une conséquence de la méthode de globalisation.
 
 Soit $\pi$ une représentation tempérée irréductible de $GL_{2n}(F)$. On aura besoin d'un résultat sur la continuité du quotient $\frac{J(1-s, \rho(w_{n,n})\widetilde{W}, \widehat{\phi})}{J(s, W, \phi)}$ lorsque l'on fait varier la représentation $\pi$, on dispose du
 \begin{lemme}
 \label{cont}
 Soient $W_0 \in \mathcal{W}(\pi, \psi)$, $\phi \in \mathcal{S}(F^n)$ et $s \in \mathbb{C}$ tel que $0 < Re(s) < 1$. Supposons que $J(s, W_0, \phi) \neq 0$. Alors il existe une application continue  $\pi' \in Temp(GL_{2n}(F)) \mapsto W_{\pi'} \in C^w(N_{2n}(F)\backslash{GL_{2n}(F)}, \psi)$ et un voisinage $V \subset Temp(GL_{2n}(F))$ de $\pi$ qui vérifient que $W_0 = W_\pi$ et $W_{\pi'} \in \mathcal{W}(\pi', \psi)$ pour tout $\pi' \in Temp(GL_{2n}(F))$. De plus, l'application $\pi' \in V \mapsto \frac{J(1-s, \rho(w_{n,n})\widetilde{W}_{\pi'}, \mathcal{F}_{\psi}(\phi))}{J(s, W_{\pi'}, \phi)}$ est bien définie et continue.
 
 En particulier, si $F$ est un corps $p$-adique, ce quotient est égal à $\gamma^{JS}(s, \pi', \Lambda^2, \psi)$ (proposition \ref{funcloc}); donc $\pi' \in V \mapsto \gamma^{JS}(s, \pi', \Lambda^2, \psi)$ est continue.
 \end{lemme}
 
 \begin{proof}
 On utilise l'existence de bonnes sections $\pi' \mapsto W_{\pi'}$ \cite[Corollaire 2.7.1]{beuzart-plessis2}. La forme linéaire $W \in C^w(N_{2n}(F)\backslash{GL_{2n}(F)}, \psi) \mapsto J(s, W, \phi)$ est continue (lemme \ref{convtemp}), il existe donc un voisinage $V$ de $\pi$ tel que $J(s, W_{\pi'}, \phi) \neq 0$ pour tous $\pi' \in V$. Le quotient $\frac{J(1-s, \rho(w_{n,n})\widetilde{W}_{\pi'}, \mathcal{F}_{\psi}(\phi))}{J(s, W_{\pi'}, \phi)}$ est alors bien définie et continue sur $V$. 
 \end{proof}
 
 On étudie maintenant la dépendance du quotient $\frac{J(1-s, \rho(w_{n,n})\widetilde{W}, \mathcal{F}_\psi(\phi))}{J(s, W, \phi)}$ par rapport au caractère additif $\psi$, où l'on note $\mathcal{F}_\psi$ pour la transformée de Fourier par rapport à $\psi$. Les caractères additifs non-triviaux de $F$ sont de la forme $\psi_\lambda$ avec $\lambda \in F^*$ où $\psi_\lambda(x) = \psi(\lambda x)$. On dispose d'un isomorphisme $W \in \mathcal{W}(\pi, \psi) \mapsto W_\lambda \in \mathcal{W}(\pi, \psi_\lambda)$ donné par $W_\lambda(g) = W(a(\lambda)g)$ pour tout $g \in GL_{2n}(F)$, où $a(\lambda) = diag(\lambda^{2n-1}, \lambda^{2n-2}, ..., \lambda, 1)$.
 
 \begin{lemme}
 \label{depcaradd}
 Soient $\lambda \in F^*$, $W \in \mathcal{W}(\pi, \psi)$, $\phi \in \mathcal{S}(F^n)$ et $s \in \mathbb{C}$ tels que $0 < Re(s) < 1$. Supposons que $J(s, W_\lambda, \phi) \neq 0$. Alors
 \begin{equation}
 \frac{J(1-s, \rho(w_{n,n})\widetilde{(W_\lambda)}, \mathcal{F}_{\psi_\lambda}(\phi))}{J(s, W_\lambda, \phi)} =  \omega_\pi(\lambda)^{2n-1}|\lambda|^{n(2n-1)(s-\frac{1}{2})}\frac{J(1-s, \rho(w_{n,n})\widetilde{W_r}, \mathcal{F}_\psi(\phi))}{J(s, W_r, \phi)},
 \end{equation}
 où $W_r$ est la translation à droite de $W$ par $diag(\lambda, 1, \lambda, 1, ...)$. En particulier, $W_r \in \mathcal{W}(\pi, \psi)$.
 
 Lorsque $F$ est un corps $p$-adique, on en déduit que
 \begin{equation}
 \gamma^{JS}(s, \pi, \Lambda^2, \psi_\lambda) = \omega_\pi(\lambda)^{2n-1}|\lambda|^{n(2n-1)(s-\frac{1}{2})} \gamma^{JS}(s, \pi, \Lambda^2, \psi).
 \end{equation}
 \end{lemme}
 
 \begin{proof}
La mesure de Haar auto-duale pour $\psi_\lambda$ est reliée à la mesure de Haar auto-duale pour $\psi$ par un facteur $|\lambda|^{\frac{n}{2}}$. On en déduit que $\mathcal{F}_{\psi_\lambda}(\phi)(x) = |\lambda|^{\frac{n}{2}}\mathcal{F}_\psi(\phi)(\lambda x)$. Le changement de variable $g \mapsto \lambda^{-1} g$ dans l'intégrale définissant $J(1-s, \rho(w_{n,n})\widetilde{W}, \mathcal{F}_\psi(\phi)(\lambda .))$ donne
 \begin{equation}
 \label{depFourier}
 J(1-s, \rho(w_{n,n})\widetilde{W}, \mathcal{F}_{\psi_\lambda}(\phi)) = |\lambda|^{n(s-\frac{1}{2})}\omega_\pi(\lambda)J(1-s, \rho(w_{n,n})\widetilde{W}, \mathcal{F}_\psi(\phi)).
 \end{equation}

D'autre part,
\begin{equation}
J(s, W_\lambda, \phi) = \int_{N_n\backslash{G_n}} \int_{V_n} W\left(a(\lambda) \sigma_n \begin{pmatrix}
1 & X \\
0 & 1
\end{pmatrix}\begin{pmatrix}
g & 0 \\
0 & g
\end{pmatrix} \sigma_n^{-1} \right)dX\widehat{\phi}(e_ng)|\det g|^s dg.
\end{equation}

On décompose $a(\lambda) \sigma_n$ sous la forme $\sigma_n diag(\lambda b(\lambda), b(\lambda))$ où l'on note $b(\lambda) = diag(\lambda^{2n-2}, \lambda^{2n-4}, ..., 1)$. Après les changements de variables, $X \mapsto b(\lambda)^{-1} X b(\lambda)$, $X \mapsto \lambda^{-1}X$ et $g \mapsto b(\lambda)^{-1}g$ , on obtient la relation
\begin{equation}
\label{deplambda1}
J(s, W_\lambda, \phi) = \delta_{B_n}(b(\lambda))^2|\det b(\lambda)|^{-s}|\lambda|^{-dim(V_n)}J(s, \rho(r(\lambda))W, \phi),
\end{equation}
où $\rho(g)$ est la translation à droite par $g$, $r(\lambda) = \sigma_n diag(\lambda, ..., \lambda, 1, ..., 1) \sigma_n^{-1} = diag(\lambda, 1, \lambda, 1, ...)$.

De plus, pour tout $g \in GL_{2n}(F)$, on a
\begin{equation}
\widetilde{(W_\lambda)}(g) = W(a(\lambda)w_n(g^t)^{-1}) = \widetilde{W}(w_n a(\lambda)^{-1}w_n^{-1}g) = \omega_{\pi}(\lambda)^{2n-1}(\widetilde{W})_\lambda(g),
\end{equation}
où l'on a utilisé la relation $w_na(\lambda)^{-1}w_n^{-1} = \lambda^{-(2n-1)}a(\lambda)$. Ce qui donne, en utilisant la relation \ref{deplambda1}, l'égalité
\begin{equation}
\label{deplambda2}
\begin{split}
J(1-s, \rho(w_{n,n})\widetilde{(W_\lambda)}, \mathcal{F}_{\psi_\lambda}(\phi)) &= \omega_\pi(\lambda)^{2n-1} \delta_{B_n}(b(\lambda))^2|\det b(\lambda)|^{s-1}|\lambda|^{-dim(V_n)} \\
&J(1-s, \rho(r(\lambda))\rho(w_{n,n})\widetilde{W}, \mathcal{F}_{\psi_\lambda}(\phi)).
\end{split}
\end{equation}

Pour finir, on remarque que l'on a pour tout $g \in GL_{2n}(F)$,
\begin{equation}
\begin{split}
\rho(r(\lambda))\rho(w_{n,n})\widetilde{W}(g) &= W\left(w_n \left((g r(\lambda) w_{n,n})^t\right)^{-1} \right) = W\left(w_n (g^t)^{-1} r(\lambda)^{-1} w_{n,n} \right) \\
& = \omega_\pi(\lambda)^{-1} W\left(w_{n}\left((gw_{n,n})^t\right)^{-1}r(\lambda) \right) = \omega_\pi(\lambda)^{-1} \rho(w_{n,n})\widetilde{W_r}(g),
\end{split}
\end{equation}
où $W_r$ est la translation à droite de $W$ par $r(\lambda)$, où l'on a utilisé la relation $r(\lambda)^{-1}w_{n,n} = \lambda^{-1} w_{n,n} r(\lambda)$.

On déduit de \ref{deplambda1} et \ref{deplambda2} la relation suivante
\begin{equation}
\frac{J(1-s, \rho(w_{n,n})\widetilde{(W_\lambda)}, \mathcal{F}_{\psi_\lambda}(\phi))}{J(s, W_\lambda, \phi)} = \omega_\pi(\lambda)^{2n-2} |\lambda|^{n(n-1)(2s-1)}\frac{J(1-s, \rho(w_{n,n})\widetilde{W_r}, \mathcal{F}_{\psi_\lambda}(\phi))}{J(s, W_r, \phi)},
\end{equation}
où l'on a utilisé l'égalité $\det b(\lambda) = \lambda^{n(n-1)}$.
On déduit le lemme grâce à la relation \ref{depFourier}.
 \end{proof}
 
 Les facteurs $\gamma$ de Shahidi du carré extérieur vérifient la même dépendance par rapport au caractère additif $\psi$ (voir Henniart \cite{henniart}). Dans la suite, on pourra donc choisir arbitrairement un caractère additif non trivial, les relations seront alors vérifiées pour tous les caractères additifs, en particulier pour le caractère $\psi$ que l'on a fixé.
 
 \begin{proposition}
 \label{proparch}
 Soit $F = \mathbb{R}$ ou $\mathbb{C}$. Soit $\pi$ une représentation tempérée irréductible de $GL_{2n}(F)$.  Les intégrales $J(s, W, \phi)$ admettent un prolongement méromorphe à $\mathbb{C}$ pour tous $W \in \mathcal{W}(\pi, \psi)$ et $\phi \in \mathcal{S}(F^n)$.
 
 Il existe une fonction méromorphe $\gamma^{JS}(s,\pi,\Lambda^2,\psi)$ telle que pour tous $s \in \mathbb{C}$, $W \in \mathcal{W}(\pi, \psi)$ et $\phi \in \mathcal{S}(F^n)$, on ait
 \begin{equation}
 \gamma^{JS}(s, \pi, \Lambda^2, \psi) J(s, W, \phi) = J(1-s, \rho(w_{n,n})\widetilde{W}, \mathcal{F}_\psi(\phi)).
 \end{equation}
 
 De plus, il existe une constante $c^{Sh}(\pi)$ de module 1 telle que pour tout $s \in \mathbb{C}$,
 \begin{equation}
 \gamma^{JS}(s, \pi, \Lambda^2, \psi) = c^{Sh}(\pi)\gamma^{Sh}(s, \pi, \Lambda^2, \psi).
 \end{equation}
 \end{proposition}
 
 \begin{proof}
 Soit $k$ un corps de nombres, on suppose que $k$ a une seule place archimédienne, elle est réelle (respectivement complexe) lorsque $F=\mathbb{R}$ (respectivement $F=\mathbb{C}$); par exemple, $k=\mathbb{Q}$ si $F=\mathbb{R}$ et $k=\mathbb{Q}(i)$ si $F=\mathbb{C}$. Soient $v \neq v'$ deux places non archimédiennes de caractéristiques résiduelles distinctes, soit $U \subset Temp(GL_{2n}(F))$ un ouvert contenant $\pi$. On choisit un caractère non trivial $\psi_\mathbb{A}$ de $\mathbb{A}_k/k$.
 
 D'après la proposition \ref{globalisation}, il existe une représentation automorphe cuspidale irréductible $\Pi$ de $GL_{2n}(\mathbb{A}_k)$ telle que $\Pi_{\infty} \in U$ et $\Pi_w$ soit non ramifiée pour toute place non archimédienne $w \neq v$.
 
 On choisit des fonctions $W_w \in \mathcal{W}(\Pi_w, (\psi_\mathbb{A})_w)$ et $\phi_w \in \mathcal{S}(k_w)$ dans le but d'appliquer l'équation fonctionnelle globale. On note $S = \{\infty, v\}$ l'ensemble des places où $\Pi$ est ramifiée et $T$ l'ensemble des places où $\psi_\mathbb{A}$ est ramifié. Pour $w \not\in S \cup T$, on prend les fonctions "non ramifiées" qui apparaissent dans la proposition \ref{calculnr}. Pour $w \in S \cup T$, on fait un choix, d'après la proposition \ref{nonzero}, tel que $J(s, W_w, \phi_w) \neq 0$. On pose alors
 \begin{equation}
 W = \prod_w W_w \quad \text{et} \quad \Phi  = \prod_w \phi_w.
 \end{equation}
 
 D'après les propositions \ref{funcglob} et \ref{calculnr}, on a
 \begin{equation}
 \label{jacquet-shalika}
 \begin{split}
 & \prod_{w \in S \cup T} J(s, W_w, \phi_w) L^{S \cup T}(s, \Pi, \Lambda^2) \\
 &= \prod_{w \in S \cup T} J(1-s, \rho(w_{n,n})\widetilde{W}_w, \mathcal{F}_{(\psi_\mathbb{A})_w}(\phi_w)) L^{S \cup T}(1-s, \widetilde{\Pi}, \Lambda^2),
 \end{split}
 \end{equation}
 où  $L^{S \cup T}(s, \Pi, \Lambda^2) = \prod_{w \not \in S \cup T} L(s, \Pi_w, \Lambda^2)$ est la fonction L partielle. D'autre part, les facteurs $\gamma$ de Shahidi vérifient une relation similaire (voir Henniart \cite{henniart}),
 \begin{equation}
 \label{shahidi}
 L^{S \cup T}(s, \Pi, \Lambda^2) = \prod_{w \in S \cup T} \gamma^{Sh}(s, \Pi_w, \Lambda^2, (\psi_\mathbb{A})_w) L^{S \cup T}(1-s, \widetilde{\Pi}, \Lambda^2).
 \end{equation}
 
 Les équations (\ref{jacquet-shalika}) et (\ref{shahidi}), en utilisant la proposition \ref{funcloc} pour les places $w \in \{v\} \cup T$, donnent
 \begin{equation}
 \label{equationlocalglobal}
 \begin{split}
 & J(1-s, \rho(w_{n,n})\widetilde{W}_\infty, \mathcal{F}_{(\psi_\mathbb{A})_\infty}(\phi_\infty)) = \\
 & J(s, W_\infty, \phi_\infty)\gamma^{Sh}(s, \Pi_\infty, \Lambda^2, (\psi_\mathbb{A})_\infty) \prod_{w \in \{v\} \cup T} \frac{\gamma^{Sh}(s, \Pi_w, \Lambda^2, (\psi_\mathbb{A})_w)}{\gamma^{JS}(s, \Pi_w, \Lambda^2, (\psi_\mathbb{A})_w)}.
 \end{split}
 \end{equation}
 
 Ce qui prouve la première partie de la proposition pour $\Pi_\infty$, l'existence du facteur $\gamma^{JS}(s, \Pi_\infty, \Lambda^2, (\psi_\mathbb{A})_\infty)$ ainsi que le prolongement méromorphe d'après le lemme \ref{convtemp}.
 
 On s'occupe tout de suite du quotient $\frac{\gamma^{Sh}(s, \Pi_w, \Lambda^2, (\psi_\mathbb{A})_w)}{\gamma^{JS}(s, \Pi_w, \Lambda^2, (\psi_\mathbb{A})_w)}$ lorsque $w \in T$. En effet, $\Pi_w$ est non ramifiée, une combinaison de la proposition \ref{calculnr} et du lemme \ref{depcaradd} va nous permettre de calculer ce quotient. Il existe $\lambda \in F^*$ et un caractère non ramifié $\psi_0$ de $F$ tel que $(\psi_\mathbb{A})_w(x) = \psi_0(\lambda x)$. La remarque suivant le lemme \ref{depcaradd} nous dit que les facteurs $\gamma^{JS}(s, \pi, \Lambda^2, \psi)$ et $\gamma^{Sh}(s, \pi, \Lambda^2, \psi)$ ont la même dépendance par rapport au caractère additif. On en déduit que
 \begin{equation}
 \frac{\gamma^{Sh}(s, \Pi_w, \Lambda^2, (\psi_\mathbb{A})_w)}{\gamma^{JS}(s, \Pi_w, \Lambda^2, (\psi_\mathbb{A})_w)} = \frac{\gamma^{Sh}(s, \Pi_w, \Lambda^2, \psi_0)}{\gamma^{JS}(s, \Pi_w, \Lambda^2, \psi_0)} = 1,
 \end{equation}
 d'après la proposition \ref{calculnr} et le calcul non ramifié des facteurs gamma de Shahidi (voir Henniart \cite{henniart}).
 
 L'équation (\ref{equationlocalglobal}) devient alors
 \begin{equation}
 \begin{split}
 & J(1-s, \rho(w_{n,n})\widetilde{W}_\infty, \mathcal{F}_{(\psi_\mathbb{A})_\infty}(\phi_\infty)) = \\
 & J(s, W_\infty, \phi_\infty)\gamma^{Sh}(s, \Pi_\infty, \Lambda^2, (\psi_\mathbb{A})_\infty) \frac{\gamma^{Sh}(s, \Pi_v, \Lambda^2, (\psi_\mathbb{A})_v)}{\gamma^{JS}(s, \Pi_v, \Lambda^2, (\psi_\mathbb{A})_v)}.
 \end{split}
 \end{equation}
 
 On choisit maintenant pour $U$ une base de voisinage contenant $\pi$, en utilisant le lemme \ref{cont} et la continuité des facteurs $\gamma$ de Shahidi sur $Temp(GL_{2n}(F))$, on en déduit que
 \begin{equation}
 R(s) = \frac{J(1-s, \rho(w_{n,n})\widetilde{W}, \mathcal{F}_{(\psi_\mathbb{A})_\infty}(\phi_\infty))}{J(s, W, \phi)\gamma^{Sh}(s ,\pi, \Lambda^2, \psi)},
 \end{equation}
pour tout $W \in \mathcal{W}(\pi, (\psi_\mathbb{A})_\infty)$, qui est à priori bien définie pour $0 < Re(s) < 1$, est une fonction indépendante de $W$ et de $\phi_\infty$. La fonction $R(s)$ ne dépend pas du choix de la base de voisinage et des choix qui sont fait lors de l'utilisation de la proposition \ref{globalisation}.
 De plus, $R$ est une limite de fractions rationnelles en $q_v^s$, donc $R$ est une fonction périodique de période $\frac{2i\pi}{\log q_v}$.
 
  En réutilisant le même raisonnement en une place $v'$ de caractéristique résiduelle distincte de celle de $v$, on voit que $R$ est aussi périodique de période $\frac{2i\pi}{\log q_{v'}}$.
 
  La fonction $R$ est donc une fonction périodique de période $\frac{2i\pi}{\log q_v}$ et $\frac{2i\pi}{\log q_{v'}}$ avec $q_v$ et $q_{v'}$ premier entre eux; ce qui est impossible sauf si $R$ est constante. Ce qui nous permet de voir qu'il existe une constante $c^{Sh}(\pi)=R$ telle que
 \begin{equation}
 \gamma^{JS}(s, \pi, \Lambda^2, (\psi_\mathbb{A})_\infty) = c^{Sh}(\pi)\gamma^{Sh}(s, \pi, \Lambda^2, (\psi_\mathbb{A})_\infty),
 \end{equation}
 où l'on a noté $\gamma^{JS}(s, \pi, \Lambda^2, (\psi_\mathbb{A})_\infty) = R(s) \gamma^{Sh}(s, \pi, \Lambda^2, (\psi_\mathbb{A})_\infty)$.
 
 Il ne nous reste plus qu'à montrer que la constante $c^{Sh}(\pi)$ est de module 1. Reprenons l'équation fonctionnelle locale archimédienne,
 \begin{equation}
 \label{funcarch}
 \gamma^{JS}(s, \pi, \Lambda^2, \psi) J(s, W, \phi) = J(1-s, \rho(w_{n,n})\widetilde{W}, \mathcal{F}_\psi(\phi)).
 \end{equation}
 
 On utilise maintenant l'équation fonctionnelle sur la représentation $\widetilde{\pi}$ pour transformer le facteur $J(1-s, \rho(w_{n,n})\widetilde{W}, \mathcal{F}_\psi(\phi))$, ce qui nous donne
 \begin{equation}
 \gamma^{JS}(s, \pi, \Lambda^2, \psi) J(s, W, \phi) = \frac{J(s, W, \mathcal{F}_{\bar{\psi}}(\mathcal{F}_\psi(\phi)))}{\gamma^{JS}(1-s, \widetilde{\pi}, \Lambda^2, \bar{\psi})}.
 \end{equation}
 
 Puisque $\mathcal{F}_{\bar{\psi}}(\mathcal{F}_\psi(\phi)) = \phi$, on obtient donc la relation 
 \begin{equation}
 \gamma^{JS}(s, \pi, \Lambda^2, \psi)\gamma^{JS}(1-s, \widetilde{\pi}, \Lambda^2, \bar{\psi}) = 1.
 \end{equation}
 
 D'autre part, en conjuguant l'équation \ref{funcarch}, on obtient
 \begin{equation}
 \overline{\gamma^{JS}(s, \pi, \Lambda^2, \psi)} = \gamma^{JS}(\bar{s}, \bar{\pi}, \Lambda^2, \bar{\psi}).
 \end{equation}
 
 Comme $\pi$ est tempérée, $\pi$ est unitaire, donc $\widetilde{\pi} \simeq \bar{\pi}$. On en déduit, pour $s = \frac{1}{2}$,
 \begin{equation}
 |\gamma^{JS}(\frac{1}{2}, \pi, \Lambda^2, \psi)|^2=1.
 \end{equation}
 
 D'autre part, le facteur $\gamma$ de Shahidi vérifie aussi $|\gamma^{Sh}(\frac{1}{2}, \pi, \Lambda^2, \psi)|^2=1$; on en déduit donc que $c^{Sh}(\pi)$ est bien de module 1.
 \end{proof}
 
 \begin{proposition}
 Supposons que $F$ est un corps $p$-adique. Soit $\pi$ une représentation tempérée irréductible de $GL_{2n}(F)$. 
 
 Le facteur $\gamma^{JS}(s,\pi,\Lambda^2,\psi)$ est défini par la proposition \ref{funcloc}. Alors il existe une constante $c^{Sh}(\pi)$ de module 1 telle que pour tout $s \in \mathbb{C}$,
 \begin{equation}
 \gamma^{JS}(s, \pi, \Lambda^2, \psi) = c^{Sh}(\pi)\gamma^{Sh}(s, \pi, \Lambda^2, \psi).
 \end{equation}
 \end{proposition}
 
 \begin{proof}
 D'après le lemme \ref{corpsglobal}, il existe un corps de nombres $k$ et une place $v_0$ telle que $k_{v_0} = F$, où $v_0$ est l'unique place de $k$ au dessus de $p$. Soit $v$ une place non archimédienne et de caractéristique résiduelle distincte de celle de $v_0$. Soit $U \subset Temp(GL_{2n}(F))$ un ouvert contenant $\pi$. On choisit un caractère non trivial $\psi_\mathbb{A}$ de $\mathbb{A}_k/k$.
 
 D'après la proposition \ref{globalisation}, il existe une représentation automorphe cuspidale irréductible $\Pi$ telle que $\Pi_{v_0} \in U$ et $\Pi_w$ soit non ramifiée pour toute place non archimédienne $w \neq v,v_0$.
 
On note $S_\infty$ l'ensemble des places archimédienne, $S = S_\infty \cup \{v,v_0\}$ et $T$ l'ensemble des places où $\psi_\mathbb{A}$ est ramifié. Pour $w \in S \cup T$, on choisit d'après la proposition \ref{nonzero}, des fonctions de Whittaker $W_w$ et des fonctions de Schwartz $\phi_w$ telles que $J(s, W_w, \phi_w) \neq 0$. Pour $w \not \in S \cup T$, on choisit les fonctions "non ramifiées" de la proposition \ref{calculnr}. On pose alors
 $$W = \prod_w W_w \quad \text{et} \quad \Phi  = \prod_w \phi_w.$$
 
 D'après l'équation fonctionnelle globale (proposition \ref{funcglob}), on a
 \begin{equation}
 \begin{split}
 &\prod_{w \in S \cup T} J(s, W_w, \phi_w) L^{S \cup T}(s, \Pi, \Lambda^2)\\
 &= \prod_{w \in S \cup T} J(1-s, \rho(w_{n,n})\widetilde{W}_w, \mathcal{F}_{(\psi_\mathbb{A})_w}(\phi_w)) L^{S \cup T}(1-s, \widetilde{\Pi}, \Lambda^2),
 \end{split}
 \end{equation}
  où $L^{S \cup T}(s, \Pi, \Lambda^2)$ est la fonction L partielle. Les facteurs $\gamma$ de Shahidi vérifient (voir Henniart \cite{henniart})
 \begin{equation}
 L^{S \cup T}(s, \Pi, \Lambda^2) = \prod_{w \in S \cup T}\gamma^{Sh}(s, \Pi_w, \Lambda^2, (\psi_\mathbb{A})_w) L^{S \cup T}(1-s, \widetilde{\Pi}, \Lambda^2).
 \end{equation}
 
 On rappelle que lors de la preuve de la proposition précédente, on a démontré que $\frac{\gamma^{Sh}(s, \Pi_w, \Lambda^2, (\psi_\mathbb{A})_w)}{\gamma^{JS}(s, \Pi_w, \Lambda^2, (\psi_\mathbb{A})_w)} = 1$ pour $w \in T$. En utilisant les propositions \ref{funcloc} et \ref{proparch}, on obtient donc la relation
 \begin{equation}
 \prod_{v_\infty \in S_\infty} c^{Sh}(\Pi_{v_\infty}) \frac{\gamma^{JS}(s, \Pi_v, \Lambda^2, (\psi_\mathbb{A})_v)}{\gamma^{Sh}(s, \Pi_v, \Lambda^2, (\psi_\mathbb{A})_v)}\frac{\gamma^{JS}(s, \Pi_{v_0}, \Lambda^2, \psi)}{\gamma^{Sh}(s, \Pi_{v_0}, \Lambda^2, \psi)} = 1.
 \end{equation}
 
 Le reste du raisonnement est maintenant identique à la fin de la preuve de la proposition \ref{proparch}. Par continuité, le quotient $\frac{\gamma^{JS}(s, \pi, \Lambda^2, \psi)}{\gamma^{Sh}(s, \pi, \Lambda^2, \psi)}$ est une fonction périodique de période $\frac{2i\pi}{\log q_v}$. Or c'est une fraction rationnelle en $q_{v_0}^s$, on obtient que c'est une constante. En évaluant $\gamma^{JS}(s, \pi, \Lambda^2, \psi)$ en $s=\frac{1}{2}$, on montre que cette constante est de module $1$.
 \end{proof}

 \section{Limite spectrale}
 
 \label{seclimite}
 Dans cette partie $F$ est un corps $p$-adique, $\psi$ un caractère additif non trivial. On renvoie à la section \ref{mesures} pour la normalisation des mesures sur $Temp(G)$, pour un groupe $G$ réductif connexe sur $F$.

On note $PG_{2n} = G_{2n}(F)/Z_{2n}(F)$ muni la mesure quotient. Soit $f \in \mathcal{S}(PG_{2n})$, pour $\pi \in Temp(PG_{2n})$, on note $\pi(f) : V_\pi \rightarrow V_\pi$, l'opérateur défini par $\pi(f)v = \int_{PG_{2n}} f(g) \pi(g)v dg$ pour tout $v \in V_\pi$ (cette intégrale est en fait une somme finie), où $V_\pi$ est l'espace sous-jacent à la représentation $\pi$. On définit alors $f_\pi$ par
\begin{equation}
f_\pi(g) = Tr(\pi(g)\pi(f^\vee)),
\end{equation}
pour tout $g \in PG_{2n}$, où $f^{\vee}(x) = f(x^{-1})$.

On note $Ad$ la représentation adjointe de $G_{2n}$ sur $M_{2n}$. On pose alors
\begin{equation}
\gamma(s, \pi, \overline{Ad}, \psi) = \frac{\gamma(s, \pi, Ad, \psi)}{\gamma(s, 1, \psi)},
\end{equation}
pour tous $\pi \in Temp(PG_{2n})$ et $s \in \mathbb{C}$.

\begin{proposition}[Harish-Chandra \cite{waldspurger}, Shahidi \cite{shahidi}, Silberger-Zink \cite{silberger-zink}]
\label{propPlanch}
Il existe une unique mesure $\mu_{PG_{2n}}$ sur $Temp(PG_{2n})$ telle que
\begin{equation}
f(g) = \int_{Temp(PG_{2n})} f_{\pi}(g) d\mu_{PG_{2n}}(\pi),
\end{equation} 
pour tous $f \in \mathcal{S}(PG_{2n})$ et $g \in PG_{2n}$. De plus, on a l'égalité de mesure suivante :
\begin{equation}
\label{mesurePlanch}
d\mu_{PG_{2n}}(\pi) = \frac{\gamma^*(0, \pi, \overline{Ad}, \psi)}{|S_\pi|} d\pi,
\end{equation}
où $\gamma^*(0, \pi, \overline{Ad}, \psi) = \lim_{s \rightarrow 0} (s log(q_F))^{-n_{\pi, \overline{Ad}}} \gamma(s, \pi, \overline{Ad}, \psi)$, avec $n_{\pi, \overline{Ad}}$ l'ordre du zéro de $\gamma(s, \pi, \overline{Ad}, \psi)$ en $s=0$. Pour $\pi \in Temp(PG_{2n})$ isomorphe à $\pi_1 \times ... \times \pi_k$, avec $\pi_i \in \Pi_{2}(G_{n_i})$, le facteur $|S_{\pi}|$ est le produit $\prod_{i=1}^k n_i$.
\end{proposition}

\begin{proof}
La mesure de Plancherel d'un groupe réductif p-adique $G$ est de la forme $d\mu_G(\pi) = \mu_G(\pi) d\pi$. Rappelons que la densité de Plancherel $\mu_G$ a été calculé par Waldspurger et Harish-Chandra \cite{waldspurger} sous la forme
\begin{equation}
\mu_G(\pi) = d(\sigma)j(\sigma)^{-1},
\end{equation}
où $\pi = i_P^G(\sigma)$ avec $P=MU$ un sous-groupe parabolique de $G$ et $\sigma \in \Pi_2(M)$, $d(\sigma)$ est le degré formel de $\sigma$ et $j(\sigma)$ est un scalaire produit d'opérateurs d'entralecements.

Par désintégration de la mesure sur $Temp(G_{2n})$ et inversion de Fourier sur $Z_{2n}$, on a $\mu_{PG_{2n}}(\pi) = \gamma^*(0,1, \psi)^{-1}\mu_{G_{2n}}(\pi)$ pour tout $\pi \in Temp(PG_{2n})$.
Pour $G = G_{2n}$, le degré formel $d(\sigma)$ a été calculé par Silberger-Zink \cite{silberger-zink} et le facteur $j(\sigma)^{-1}$ par Shahidi \cite{shahidi}. On renvoie à \cite[Proposition 2.13.2]{beuzart-plessis} pour le calcul du produit qui donne la relation \ref{mesurePlanch}.
\end{proof}

On note $\Phi(G)$ l'ensemble des paramètres de Langlands tempérés de $G$ et $Temp(G)/Stab$ le quotient de $Temp(G)$ par la relation d'équivalence $\pi \equiv \pi' \iff \varphi_\pi = \varphi_{\pi'}$, où $\varphi_\pi$ est le paramètre de Langlands associé à $\pi$.

Rappelons (section \ref{notations}) que la correspondance de Langlands locale pour $SO(2m+1)$ nous permet de définir une application de transfert $T : Temp(SO(2m+1))/Stab \rightarrow Temp(G_{2m})$. On sait caractériser l'image de l'application de transfert. Plus exactement,
\begin{equation}
\label{caracTransf}
\pi \in T(Temp(SO(2n+1))/Stab) \iff \pi = \left( \bigtimes_{i=1}^k \tau_i \times \widetilde{\tau_i} \right) \times \bigtimes_{j=1}^l \mu_i
\end{equation}
avec $\tau_i \in \Pi_2(G_{n_i})$ et $\mu_j \in T(Temp(SO(2m_j+1))/Stab) \cap \Pi_2(G_{2m_j})$ (de manière équivalente $\mu_j \in \Pi_2(G_{2m_j})$ et $\gamma(0, \mu_j, \Lambda^2, \psi) = 0$).

\begin{proposition}
\label{limitespectrale}
Soit $\phi$ une fonction régulière à support compact sur $Temp(PG_{2n})$, on a 

\begin{equation}
\begin{split}
& \lim_{s \rightarrow 0^+}  n \gamma(s, 1, \psi) \int_{Temp(PG_{2n})} \phi(\pi) \gamma(s, \pi, \Lambda^2, \psi)^{-1} d\mu_{PG_{2n}}(\pi) = \\
& \int_{Temp(SO_{2n+1}) / Stab} \phi(T(\sigma)) \frac{\gamma^*(0, \sigma, Ad, \psi)}{|S_\sigma|} d\sigma.
 \end{split}
\end{equation}

où pour $\sigma \in Temp(SO(2n+1))$ sous-représentation de $\pi_1 \times ... \times \pi_l \rtimes \sigma_0$, avec $\pi_i \in \Pi_{2}(G_{n_i})$ et $\sigma_0 \in \Pi_2(SO(2m+1))$, le facteur $|S_{\pi}|$ est le produit $|S_{\pi_1}|...|S_{\pi_l}||S_{\sigma_0}|$; où $|S_{\sigma_0}|=2^k$ tel que $T(\sigma_0) \simeq \tau_1 \times ... \times \tau_k$ avec $\tau_i \in \Pi_2(G_{m_i})$.
\end{proposition}

\begin{proof}
D'après la relation \ref{mesurePlanch}, on a
\begin{equation}
\label{int123}
\int_{Temp(PG_{2n})} \phi(\pi) \gamma(s, \pi, \Lambda^2, \psi)^{-1} d\mu_{PG_{2n}}(\pi) = 
\int_{Temp(PG_{2n})} \phi(\pi) \frac{\gamma^*(0, \pi, \overline{Ad}, \psi)}{|S_\pi|\gamma(s, \pi, \Lambda^2, \psi)} d\pi.
\end{equation}

Soit $\pi \in Temp(PG_{2n})$. En prenant des partitions de l'unité, on peut supposer que $\phi$ est à support dans un voisinage $U$ suffisamment petit de $\pi$. On écrit la représentation $\pi$ sous la forme
\begin{equation}
\pi = \left( \bigtimes_{i=1}^t \tau_i^{\times m_i} \times \widetilde{\tau_i}^{\times n_i} \right) \times \left( \bigtimes_{j=1}^u \mu_j^{\times p_j} \right) \times \left( \bigtimes_{k=1}^v \nu_k^{\times q_k}\right),
\end{equation}
où
\begin{itemize}
\item $\tau_i \in \Pi_2(G_{d_i})$ vérifie $\tau_i \not \simeq \widetilde{\tau_i}$ pour tout $1\leq i \leq t$. De plus, pour tous $1 \leq i < i' \leq t$, $\tau_i \not \simeq \tau_{i'}$ et $\tau_i \not \simeq \widetilde{\tau_{i'}}$.
\item $\mu_j \in \Pi_2(G_{e_j})$ vérifie $\mu_j \simeq \widetilde{\mu_j}$ et $\gamma(0, \mu_j, \Lambda^2, \psi) \neq 0$ pour tout $1\leq j \leq u$. De plus, pour tous $1 \leq j < j' \leq u$, $\mu_j \not \simeq \mu_{j'}$.
\item $\nu_k \in \Pi_2(G_{f_k})$ vérifie $\gamma(0, \nu_k, \Lambda^2, \psi) = 0$ ( et donc $\nu_k \simeq \widetilde{\nu_k}$ ) pour tout $1\leq k \leq v$. De plus, pour tous $1 \leq k < k' \leq v$, $\nu_k \not \simeq \nu_{k'}$.
\end{itemize}

On note $M = \left( \prod_{i=1}^t G_{d_i}^{m_i+n_i} \times \prod_{j=1}^u G_{e_j}^{p_j} \times \prod_{k=1}^v G_{f_k}^{q_k} \right) / Z_{2n}$ et $P$ un parabolique de $PG_{2n}$ de Levi $M$. Alors $\pi = i_P^{PG_{2n}}(\tau)$ pour une certaine représentation $\tau$ de $M$.

On note $X^*(M)$ le groupe des caractères algébriques de $M$. On note $\mathcal{A} \subset \prod_{i=1}^t (i\mathbb{R})^{m_i+n_i} \times \prod_{j=1}^u (i\mathbb{R})^{p_j} \times \prod_{k=1}^v (i\mathbb{R})^{q_k} = (i\mathbb{R})_M$ l'hyperplan défini par la condition que la somme des coordonnées est nulle. La forme linéaire "somme des coordonnées" induit un isomorphisme $(i\mathbb{R})_M/\mathcal{A} \simeq i\mathbb{R}$.

On équipe $(i\mathbb{R})_M$ du produit des mesures de Lebesgue sur $i\mathbb{R}$ et $\mathcal{A}$ de la mesure de Haar telle que la mesure quotient sur $(i\mathbb{R})_M/\mathcal{A} \simeq i\mathbb{R}$ (l'isomorphisme étant induit par la "somme des coordonnées") soit la mesure de Lebesgue.

Dans la suite, on notera les coordonnées de $\lambda \in \mathcal{A}$ de la manière suivante :
\begin{itemize}
\item $x_i(\lambda) = (x_{i,1}(\lambda), ..., x_{i, m_i}(\lambda), \widetilde{x_{i, 1}}(\lambda), ..., \widetilde{x_{i,n_i}}(\lambda)) \in (i\mathbb{R})^{m_i} \times (i\mathbb{R})^{n_i}$,
\item $y_j(\lambda) = (y_{j,1}(\lambda), ..., y_{j, p_j}(\lambda)) \in (i\mathbb{R})^{p_j}$,
\item $z_k(\lambda) = (z_{k,1}(\lambda), ..., z_{k, q_k}(\lambda)) \in (i\mathbb{R})^{q_k}$,
\end{itemize}
pour tout $\lambda \in \mathcal{A}$.

Pour $\lambda = \chi \otimes \mu \in i\mathcal{A}^*_M$, on note $\tau_\lambda = \tau \otimes |\chi|^\mu$. On rappelle que $W(PG_{2n}, \tau)$ est le sous-groupe de $W(PG_{2n}, M)$ fixant la représentation $\tau$. Soit $V_M$ un voisinage ouvert de $0$ suffisamment petit $W(PG_{2n},\tau)$-invariant dans $i\mathcal{A}^*_M$ telle que l'application $\lambda \in i\mathcal{A}^*_M \mapsto \pi_\lambda = i_{P}^{PG_{2n}}(\tau_\lambda) \in Temp(PG_{2n})$ induit un isomorphisme topologique entre $V_M \slash W(PG_{2n}, \tau)$ et un voisinage ouvert $U$ de $\pi$ dans $Temp(PG_{2n})$. On équipe $i\mathcal{A}^*_M = X^*(M) \otimes i\mathbb{R}$ de l'unique mesure de Haar qui vérifie la formule d'intégration suivante
\begin{equation}
\label{formuleIntegration}
\int_U f(\pi) d\pi = \frac{1}{W(PG_{2n}, \tau)} \int_{V_M} f(\pi_\lambda) d\lambda,
\end{equation}
pour toute fonction $f$ localement constante à support compact sur $Temp(PG_{2n})$.

Il existe un isomorphisme d'espace vectoriel $\mathcal{A} \simeq i\mathcal{A}^*_M$ tel que lorsqu'on le compose avec l'application $\lambda \in i\mathcal{A}^*_M \mapsto \pi_\lambda\in Temp(PG_{2n})$, on obtient l'application $\lambda \in \mathcal{A} \mapsto \pi_\lambda \in Temp(PG_{2n})$ où
\begin{equation}
\begin{split}
\pi_{\lambda} = &\left( \bigtimes_{i=1}^t \left( \bigtimes_{l=1}^{m_i} \tau_i \otimes |\det|^{\frac{x_{i,l}(\lambda)}{d_i}} \right) \times \left( \bigtimes_{l=1}^{n_i} \widetilde{\tau_i} \otimes |\det|^{\frac{\widetilde{x_{i,l}}(\lambda)}{d_i}} \right) \right) \\
& \times \left( \bigtimes_{j=1}^u \bigtimes_{l=1}^{p_j} \mu_j\otimes |\det|^{\frac{y_{j,l}(\lambda)}{e_j}} \right) \times \left( \bigtimes_{k=1}^v \bigtimes_{l=1}^{q_k} \nu_k \otimes |\det|^{\frac{z_{k,l}(\lambda)}{f_k}} \right).
\end{split}
\end{equation}

Quitte à restreindre $U$, cette dernière induit un homéomorphisme $U \simeq V/W(PG_{2n}, \tau)$, où $V$ est un voisinage de 0 dans $\mathcal{A}$.

L'isomorphisme $\mathcal{A} \simeq i\mathcal{A}^*_M$ envoie la mesure de Haar de $\mathcal{A}$ sur $\left(\frac{\log(q_F)}{2\pi}\right)^{\dim(A_M)}$ fois la mesure de Haar sur $i\mathcal{A}^*_M$. On en déduit, grâce à \ref{formuleIntegration}, que l'intégrale \ref{int123} est égale à
\begin{equation}
\frac{1}{|W(PG_{2n}, \tau)|} \left(\frac{log(q_F)}{2\pi}\right)^{dim(\mathcal{A})}\int_{V} \phi(\pi_\lambda) \frac{\gamma^*(0, \pi_\lambda, \overline{Ad}, \psi)}{|S_{\pi_\lambda}|\gamma(s, \pi_\lambda, \Lambda^2, \psi)} d\lambda,
\end{equation}
où $dim(\mathcal{A}) = dim(A_M) = \left(\sum_{i=1}^t m_i + n_i + \sum_{j=1}^u p_j + \sum_{k=1}^v q_k\right) - 1$. De plus, on a
\begin{equation}
|S_{\pi_\lambda}| = \prod_{i=1}^t d_i^{m_i+n_i} \prod_{j=1}^u e_j^{p_j} \prod_{k=1}^v f_k^{q_k}.
\end{equation}
On notera ce produit $P$ dans la suite. 

On en déduit l'égalité suivante :
\begin{equation}
\label{defvarphi}
\begin{split}
\int_{Temp(PG_{2n})} &\phi(\pi) \gamma(s, \pi, \Lambda^2, \psi)^{-1} d\mu_{PG_{2n}}(\pi) = \\
& \frac{1}{|W(PG_{2n}, \tau)|P} \left(\frac{log(q_F)}{2\pi}\right)^{dim(\mathcal{A})} 
\int_{\mathcal{A}} \varphi(\lambda) \frac{\gamma^*(0, \pi_\lambda, \overline{Ad}, \psi)}{\gamma(s, \pi_\lambda, \Lambda^2, \psi)} d\lambda,
\end{split}
\end{equation}
où $\varphi(\lambda) = \phi(\pi_\lambda)$ si $\lambda \in V$ et $0$ sinon. La fonction $\varphi$ est $W(PG_{2n}, \tau)$-équivariante à support compact.

Décrivons maintenant la forme des facteurs $\gamma$, on aura besoin des propriétés de ces derniers.
\begin{propriete}
Les facteurs $\gamma$ vérifient les propriétés suivantes :
\begin{itemize}
\item $\gamma(s, \pi_1 \times \pi_2, Ad, \psi) = \gamma(s, \pi_1, Ad, \psi)\gamma(s, \pi_2, Ad, \psi) \gamma(s, \pi_1 \times \widetilde{\pi_2}, \psi) \gamma(s, \widetilde{\pi_1} \times \pi_2, \psi)$,
\item $\gamma(s, \pi |\det |^x, Ad, \psi) = \gamma(s, \pi, Ad, \psi)$,
\item $\gamma(s, \pi, Ad, \psi)$ a un zéro simple en $s=0$,
\item $\gamma(s, \pi_1 \times \pi_2, \Lambda^2, \psi) = \gamma(s, \pi_1, \Lambda^2, \psi) \gamma(s, \pi_2, \Lambda^2, \psi) \gamma(s, \pi_1 \times \pi_2, \psi)$,
\item $\gamma(s, \pi |\det |^x, \Lambda^2, \psi) = \gamma(s + 2x, \pi, \Lambda^2, \psi)$,
\item $\gamma(s, \pi, \Lambda^2, \psi)$ a au plus un zéro simple en $s=0$ et $\gamma(0, \pi, \Lambda^2) = 0$ si et seulement si $\pi$ est dans l'image de l'application de transfert $T$,
\end{itemize}
pour tous $x,s \in \mathbb{C}$, $\pi \in \Pi_2(G_m)$ et $\pi_1, \pi_2 \in Temp(G_m)$.
\end{propriete}

On en déduit que
\begin{equation}
\begin{split}
&\gamma^*(0, \pi_\lambda, \overline{Ad}, \psi) = \left( \prod_{i=1}^t \prod_{1 \leq l \neq l' \leq m_i} (\frac{x_{i,l}(\lambda)-x_{i,l'}(\lambda)}{d_i}) \prod_{1 \leq l \neq l' \leq n_i} (\frac{\widetilde{x_{i,l}}(\lambda)-\widetilde{x_{i,l'}}(\lambda)}{d_i}) \right) \\
& \left( \prod_{j=1}^u \prod_{1 \leq l \neq l' \leq p_j} (\frac{y_{j,l}(\lambda)-y_{j,l'}(\lambda)}{e_j}) \right)
\left( \prod_{k=1}^v \prod_{1 \leq l \neq l' \leq q_k} (\frac{z_{k,l}(\lambda)-z_{k,l'}(\lambda)}{f_k}) \right) F(\lambda),
\end{split}
\end{equation}
où $F$ est une fonction $W(PG_{2n}, \tau)$-équivariante $C^\infty$ qui ne s'annule pas sur le voisinage $V$ (quitte à rétrécir $V$), il s'agit d'un produit de facteur $\gamma$ ne s'annulant pas sur $V$. De même, on a
\begin{equation}
\begin{split}
& \gamma(s, \pi_\lambda, \Lambda^2, \psi)^{-1} = \left( \prod_{i=1}^t \prod_{\substack{1 \leq l \leq m_i \\ 1\leq l' \leq n_i}} (s+\frac{x_{i,l}(\lambda)+\widetilde{x_{i,l'}}(\lambda)}{d_i})^{-1} \right) \\ &
\left( \prod_{j=1}^u \prod_{1 \leq l < l' \leq p_j} (s + \frac{y_{j,l}(\lambda)+y_{j,l'}(\lambda)}{e_j})^{-1} \right) 
 \left( \prod_{k=1}^v \prod_{1 \leq l \leq l' \leq q_k} (s+\frac{z_{k,l}(\lambda)+z_{k,l'}(\lambda)}{f_k})^{-1} \right) G(2\lambda+s),
 \end{split}
\end{equation}
où la fonction $G$ est une fonction $W(PG_{2n}, \tau)$-équivariante méromorphe sur $\mathcal{A} \otimes \mathbb{C}$ tel que ses diviseurs polaires ne rencontrent pas $2V+\mathcal{H}$ (quitte à rétrécir $V$); ici $\mathcal{H} = \{z \in \mathbb{C}, Re(z) > 0\} \cup \{0\}$ s'injecte dans $\mathcal{A} \otimes \mathbb{C}$ par l'application $s \in \mathcal{H} \mapsto \lambda_s \in \mathcal{A} \otimes \mathbb{C}$ dont les coordonnées sont $x_i(\lambda_s) = d_i(s, ..., s)$, $y_j(\lambda_s) = e_j(s, ..., s)$ et $z_k(\lambda_s) = f_k(s, ..., s)$.

On énonce maintenant le résultat fondamental de \cite{beuzart-plessis}, qui permet d'obtenir la proposition pour la représentation d'Asai. En reprenant les notations de \cite[Proposition 3.2.1]{beuzart-plessis}, on écrit
\begin{equation}
\varphi(\lambda) \frac{\gamma^*(0, \pi_\lambda, \overline{Ad}, \psi)}{\gamma(s, \pi_\lambda, \Lambda^2, \psi)} = \varphi_s(\lambda) 
\prod_{i=1}^t P_{m_i,n_i,s}(\frac{x_i(\lambda)}{d_i}) \prod_{j=1}^u Q_{p_j,s}(\frac{y_j(\lambda)}{e_j}) \prod_{i=1}^v R_{q_k,s}(\frac{z_k(\lambda)}{f_k}),
\end{equation}
où $\varphi_s(\lambda) = \varphi(\lambda) F(\lambda) G(2\lambda+s)$. 
La fonction $s \in \mathcal{H} \cup \{0\} \mapsto \varphi_s \in C^\infty_c(\mathcal{A})$ est continue.
De plus, $\varphi_s$ est $W(PG_{2n}, \tau)$-équivariante à support compact. Les lettres $P, Q, R$ désignent des fractions rationnelles qui apparaissent dans le quotient des facteurs $\gamma$ (voir \cite[section 3]{beuzart-plessis}).
\begin{proposition}[Beuzart-Plessis {\cite[Proposition 3.3.1]{beuzart-plessis}}]
\label{beuzart-plessis}
La limite
\begin{equation}
\lim_{s \rightarrow 0^+}  \frac{n s}{|W|} \int_{\mathcal{A}} \varphi_s(\lambda) 
\prod_{i=1}^t P_{m_i,n_i,s}(\frac{x_i(\lambda)}{d_i}) \prod_{j=1}^u Q_{p_j,s}(\frac{y_j(\lambda)}{e_j}) \prod_{i=1}^v R_{q_k,s}(\frac{z_k(\lambda)}{f_k}) d\lambda
\end{equation}
est nulle si $m_i \neq n_i$ pour un certain $i$ ou si l'un des $p_j$ est impair. De plus, dans le cas contraire, elle est égale à
\begin{equation}
\begin{split}
&\frac{D(2\pi)^{N-1}2^{-c}}{|W'|} \\
&\int_{\mathcal{A}'} \lim_{s \rightarrow 0^+} \varphi_s(\lambda') s^N \prod_{i=1}^t P_{m_i,n_i,s}(\frac{x_i(\lambda')}{d_i}) \prod_{j=1}^u Q_{p_j,s}(\frac{y_j(\lambda')}{e_j}) \prod_{i=1}^v R_{q_k,s}(\frac{z_k(\lambda')}{f_k}) d\lambda';
\end{split}
\end{equation}
où
\begin{itemize}
\item $D = \prod_{i=1}^t d_i^{n_i} \prod_{j=1}^u e_j^{\frac{p_j}{2}} \prod_{k=1}^v f_k^{\lceil \frac{q_k}{2} \rceil}$,
\item c est le cardinal des $1 \leq k \leq t$ tel que $q_k \equiv 1 \mod 2$,
\item $N = \sum_{i=1}^t n_i + \sum_{j=1}^u \frac{p_j}{2} + \sum_{k=1}^v \lceil \frac{q_k}{2} \rceil$,
\item $W$ et $W'$ sont définis de manière intrinsèque dans \cite[section 3.3]{beuzart-plessis}, $W$ est isomorphe à $W(PG_{2n}, \tau)$ et on introduira un groupe isomorphe à $W'$ plus loin.
\end{itemize}

De plus, $\mathcal{A}'$ est le sous-espace de $\mathcal{A}$ défini par les relations :
\begin{itemize}
\item $x_{i,l}(\lambda) + \widetilde{x_{i,l}}(\lambda) = 0$ pour tous $1 \leq i \leq t$ et $1 \leq l \leq n_i$,
\item $y_{j,l}(\lambda) + y_{j,p_j+1-l}(\lambda) = 0$ pour tous $1 \leq j \leq u$ et $1 \leq l \leq  \frac{p_j}{2} $,
\item $z_{k,l}(\lambda) + z_{k,q_k+1-l}(\lambda) = 0$ pour tous $1 \leq k \leq v$ et $1 \leq l \leq \lceil \frac{q_k}{2} \rceil$.
\end{itemize}
On équipe $\mathcal{A}'$ de la mesure Lebesgue provenant de l'isomorphisme
\begin{equation}
\mathcal{A}' \simeq \prod_{i=1}^t (i\mathbb{R})^{n_i} \prod_{j=1}^u (i\mathbb{R})^{\frac{p_j}{2}} \prod_{k=1}^v (i\mathbb{R})^{\lfloor \frac{q_k}{2} \rfloor}
\end{equation}
qui correspond à la projection consistant à supprimer les coordonnées redondantes : $\widetilde{x_{i,l}}$ pour tous $1 \leq i \leq t$ et $1 \leq l \leq n_i$, $y_{j,l}$ pour tous $1 \leq j \leq u$ et $\frac{p_j}{2} < l \leq p_j$, $z_{k,l}$ pour tous $1 \leq k \leq v$ et $\lceil \frac{q_k}{2} \rceil < l \leq q_k$.
\end{proposition}

\begin{proof}
Il nous faut donner une petite explication par rapport à la proposition 3.3.1 dans \cite{beuzart-plessis}. En effet, on l'utilise pour une famille $\varphi_s$ alors qu'elle est énoncée pour une unique fonction $\varphi$. Notons, comme dans \cite[section 3.3]{beuzart-plessis},
\begin{equation}
D_s(\varphi) = \int_{\mathcal{A}} \varphi(\lambda) 
\prod_{i=1}^t P_{m_i,n_i,s}(\frac{x_i(\lambda)}{d_i}) \prod_{j=1}^u Q_{p_j,s}(\frac{y_j(\lambda)}{e_j}) \prod_{i=1}^v R_{q_k,s}(\frac{z_k(\lambda)}{f_k}) d\lambda,
\end{equation}
\begin{equation}
\begin{split}
D'(\varphi) &= \frac{D(2\pi)^{N-1}2^{-c}}{n} \\
&\int_{\mathcal{A}'} \lim_{s \rightarrow 0^+} \varphi(\lambda') s^N \prod_{i=1}^t P_{m_i,n_i,s}(\frac{x_i(\lambda')}{d_i}) \prod_{j=1}^u Q_{p_j,s}(\frac{y_j(\lambda')}{e_j}) \prod_{i=1}^v R_{q_k,s}(\frac{z_k(\lambda')}{f_k}) d\lambda',
\end{split}
\end{equation}
pour tout $\varphi \in C^\infty_c(\mathcal{A})$.

On se contente du cas où $m_i = n_i$ pour tout $1 \leq i \leq t$ et que $p_j$ est pair pour tout $1 \leq j \leq u$, l'autre cas se fait en utilisant le même raisonnement en remplaçant $D'$ par $0$. On note $C^\infty_c(\mathcal{A})^W$ le sous-espace de $C^\infty_c(\mathcal{A})$ des fonctions $W$-équivariantes. D'après la proposition 3.3.1 de \cite{beuzart-plessis}, pour tout $\varphi \in C^\infty_c(\mathcal{A})^W$, on a
\begin{equation}
\lim_{s \rightarrow 0^+} sD_s(\varphi) = \frac{|W|}{|W'|}D'(\varphi).
\end{equation}

Les formes linéaires $D_s : C^\infty_c(\mathcal{A})^W \rightarrow \mathbb{C}$ et la fonction $s \in \mathcal{H} \cup \{0\} \mapsto \varphi_s \in C^\infty_c(\mathcal{A})$ sont continues. D'après le théorème de Banach-Steinhaus, on en déduit que
\begin{equation}
\lim_{s \rightarrow 0^+} sD_s(\varphi_s) = \frac{|W|}{|W'|}\lim_{s \rightarrow 0^+} D'(\varphi_s),
\end{equation}
ce qui est bien notre proposition.
\end{proof}

Supposons tout d'abord que $\pi$ n'est pas de la forme $T(\sigma)$ pour un certain $\sigma \in Temp(SO(2n+1))/Stab$. D'après la caractérisation \ref{caracTransf}, il existe $1 \leq i \leq r$ tel que $m_i \neq n_i$ ou $p_j$ est impair (on vérifie aisément que les autres cas se mettent sous la forme qui apparait dans \ref{caracTransf}). Alors en prenant $U$ suffisamment petit, on peut supposer que $U$ ne rencontre pas l'image de l'application de transfert $T$. Autrement dit, le terme de droite de la proposition est nul; d'après la proposition \ref{beuzart-plessis}, le terme de gauche l'est aussi.

Supposons maintenant qu'il existe $\sigma \in Temp(SO(2n+1))/Stab$ tel que $\pi = T(\sigma)$. Alors $m_i = n_i$ pour tout $1 \leq i \leq t$ et les $p_j$ sont pairs. De plus, on peut écrire
\begin{equation}
\sigma = \left( \bigtimes_{i=1}^t \tau_i^{\times n_i} \times \bigtimes_{j=1}^u \mu_j^{\times \frac{p_j}{2}} \times \bigtimes_{k=1}^v \nu_k^{\times \lfloor \frac{q_k}{2}\rfloor} \right) \rtimes \sigma_0,
\end{equation}
où $\sigma_0$ est une représentation de $SO(2m+1)$ pour un certain $m$ tel que
\begin{equation}
\label{sigma0}
T(\sigma_0) = \bigtimes_{\substack{k=1 \\ q_k \equiv 1 \mod 2}}^v \nu_k.
\end{equation}

On note $L = \prod_{i=1}^t G_{d_i}^{n_i} \times \prod_{j=1}^u G_{e_j}^{\frac{p_j}{2}} \times \prod_{k=1}^v G_{f_k}^{\lfloor \frac{q_k}{2} \rfloor} \times SO(2m+1)$ et $P$ un sous-groupe parabolique de $SO(2n+1)$ de Levi $L$.
On a $\sigma = i_P^{SO(2n+1)}(\Sigma)$, où $\Sigma \in \Pi_2(L)$. Le groupe $W'$ de la proposition \ref{beuzart-plessis} est isomorphe à $W(SO(2n+1), \sigma)$, où $W(SO(2n+1), \sigma)$ est le sous-groupe de $W(SO(2n+1), L)$ fixant la classe d'isomorphisme de $\sigma$.

On équipe $i\mathcal{A}^*_L = X^*(L) \otimes i\mathbb{R}$ de l'unique mesure de Haar telle que $i\mathcal{A}^*_L \slash (\frac{2i\pi}{\log(q_F)})X^*(L)$ a pour volume $1$.
Pour $\lambda' = \chi \otimes \mu \in i\mathcal{A}^*_L$, on note $\Sigma_{\lambda'} = \Sigma \otimes |\chi|^\mu$. Il existe un isomorphisme d'espace vectoriel $\mathcal{A'} \simeq i\mathcal{A}^*_L$ tel que lorsqu'on le compose avec l'application $\lambda' \in i\mathcal{A}^*_L \mapsto \sigma_{\lambda'} = i_P^{SO(2n+1)}(\Sigma_{\lambda'})\in Temp(SO(2n+1))$, on obtient l'application $\lambda' \in \mathcal{A}' \mapsto \sigma_{\lambda'} \in Temp(SO(2n+1))$, avec
\begin{equation}
\begin{split}
\sigma_{\lambda'} &= \left( \bigtimes_{i=1}^t \bigtimes_{l=1}^{n_i} \tau_i \otimes |\det|^{\frac{x_{i,l}(\lambda')}{d_i}} \right) \times \left( \bigtimes_{j=1}^u \bigtimes_{l=1}^{\frac{p_j}{2}}\mu_j \otimes |\det|^{\frac{y_{j,l}(\lambda')}{e_j}} \right) \\
&\times \left(\bigtimes_{k=1}^v \bigtimes_{l=1}^{\lfloor \frac{q_k}{2} \rfloor} \nu_k \otimes |\det|^{\frac{z_{k,l}(\lambda')}{f_k}} \right) \rtimes \sigma_0.
\end{split}
\end{equation}

L'isomorphisme $\mathcal{A'} \simeq i\mathcal{A}^*_L$ envoie la mesure de Haar de $\mathcal{A}$ sur $\left(\frac{\log(q_F)}{2\pi}\right)^{\dim(\mathcal{A}')}$ fois la mesure de Haar sur $i\mathcal{A}^*_L$. De plus, quitte à rétrécir $V$, pour $\lambda \in V$, $\pi_\lambda \in T(SO(2n+1)/Stab)$ si et seulement si $\lambda \in \mathcal{A}'$; d'après l'équivalence \ref{caracTransf}. Dans ce cas $\pi_\lambda = T(\sigma_\lambda)$.

En utilisant cette caractérisation et la définition de la fonction $\varphi$ (équation \ref{defvarphi}), on obtient
\begin{equation}
\label{membreDroite}
\begin{split}
& \int_{Temp(SO(2n+1))/Stab} \phi(T(\sigma)) \frac{\gamma^*(0, \sigma, Ad, \psi)}{|S_\sigma|} d\sigma \\
&= \frac{1}{|W'|} \left(\frac{log(q_F)}{2\pi}\right)^{dim(\mathcal{A}')} \int_{\mathcal{A}'} \phi(T(\sigma_{\lambda'})) \frac{\gamma^*(0, \sigma_{\lambda'}, Ad, \psi)}{|S_{\sigma_{\lambda'}}|} d\lambda' \\
&= \frac{1}{|W'|} \left(\frac{log(q_F)}{2\pi}\right)^{dim(\mathcal{A}')} \int_{\mathcal{A}'} \varphi(\lambda') \frac{\gamma^*(0, \sigma_{\lambda'}, Ad, \psi)}{|S_{\sigma_{\lambda'}}|} d\lambda'.
\end{split}
\end{equation}

De plus,
\begin{equation}
\label{Ssigma}
|S_{\sigma_{\lambda'}}| = \prod_{i=1}^t {d_i}^{n_i} \prod_{j=1}^u {e_j}^{\frac{p_j}{2}} \prod_{k=1}^v {f_k}^{\lfloor \frac{q_k}{2} \rfloor} |S_{\sigma_0}| = 2^c \frac{P}{D},
\end{equation}
d'après les notations de la proposition \ref{beuzart-plessis} et la relation \ref{sigma0}. D'autre part, d'après la proposition \ref{beuzart-plessis} et l'équation \ref{defvarphi}, on a
\begin{equation}
\begin{split}
& \lim_{s \rightarrow 0^+}  n \gamma(s, 1, \psi) \int_{Temp(PG_{2n})} \phi(\pi) \gamma(s, \pi, \lambda^2, \psi)^{-1} d\mu_{PG_{2n}}(\pi) = \frac{D(2\pi)^{N-1}2^{-c} \gamma^*(0, 1, \psi)log(q_F)}{|W'|P} \\
&\left(\frac{log(q_F)}{2\pi}\right)^{dim(\mathcal{A})}\int_{\mathcal{A}'} \lim_{s \rightarrow 0^+} \varphi_s(\lambda') s^N \prod_{i=1}^t P_{m_i,n_i,s}(\frac{x_i(\lambda')}{d_i}) \prod_{j=1}^u Q_{p_j,s}(\frac{y_j(\lambda')}{e_j}) \prod_{i=1}^v R_{q_k,s}(\frac{z_k(\lambda')}{f_k}) d\lambda'.
\end{split}
\end{equation}
Cette dernière intégrale est égale à
\begin{equation}
\int_{\mathcal{A}'} \varphi(\lambda') \lim_{s \rightarrow 0^+} s^N \frac{\gamma^*(0, \pi_{\lambda'}, \overline{Ad}, \psi)}{\gamma(s, \pi_{\lambda'}, \Lambda^2, \psi)} d\lambda'.
\end{equation}
De plus, on remarque que $s \mapsto \gamma(s, \pi_{\lambda'}, \Lambda^2, \psi)^{-1}$ a un pôle d'ordre $N$ en $s=0$. Notre membre de gauche est donc égal à
\begin{equation}
\label{finalgauche}
\frac{D\left(2\pi\right)^{N-1}2^{-c}log(q_F)}{|W'|P} \left(\frac{log(q_F)}{2\pi}\right)^{dim(\mathcal{A})} \int_{\mathcal{A}'} \varphi(\lambda') \frac{\gamma^*(0, \sigma_{\lambda'}, Ad, \psi)}{log(q_F)^N} d\lambda'.
\end{equation}
On a utilisé les relations $\gamma^*(0, 1, \psi)\gamma^*(0, \pi_{\lambda'}, \overline{Ad}, \psi) = \gamma^*(0, \pi_{\lambda'}, Ad, \psi)$ et
\begin{equation}
\frac{\gamma(s, T(\sigma_{\lambda'}), Ad, \psi)}{\gamma(s, T(\sigma_{\lambda'}), \Lambda^2, \psi)} = \gamma(s, \sigma_{\lambda'}, Ad, \psi).
\end{equation}

Dans l'expression \ref{finalgauche}, le facteur $\frac{log(q_F)}{2\pi}$ apparait avec un exposant $dim(\mathcal{A})-N+1 = dim(\mathcal{A}')$; on en déduit que \ref{finalgauche} est égal au membre de droite \ref{membreDroite}, d'après l'égalité \ref{Ssigma}.
\end{proof}

\section{Une formule d'inversion de Fourier}
\label{formInv}
On note $H_n$ le sous-groupe des matrices de la forme $\sigma_n \begin{pmatrix}
1 & X \\
0 & 1
\end{pmatrix}\begin{pmatrix}
g & 0 \\
0 & g
\end{pmatrix} \sigma_n^{-1}$ où $X$ est dans $M_n$ et $g$ dans $G_n$. On pose $H^P_n = H_n \cap P_{2n}$. On note $\theta$ le caractère sur $H_n$ qui envoie $\sigma_n \begin{pmatrix}
1 & X \\
0 & 1
\end{pmatrix}\begin{pmatrix}
g & 0 \\
0 & g
\end{pmatrix} \sigma_n^{-1}$ sur $\psi(Tr(X))$.

On équipe $H_n$, $H_n \cap N_{2n} \backslash{H_n}$ et $H^P_n \cap N_{2n} \backslash{H^P_n}$ des mesures suivantes :
\begin{itemize}
\item $\int_{H_n} f(s) ds = \int_{G_n} \int_{M_n} f\left(\sigma_n \begin{pmatrix}
1 & X \\
0 & 1
\end{pmatrix}\begin{pmatrix}
g & 0 \\
0 & g
\end{pmatrix} \sigma_n^{-1}\right) dX dg,$ pour $f \in \mathcal{S}(G_{2n}),$
\item $\int_{H_n \cap N_{2n} \backslash{H_n}} f(\xi) \theta(\xi)^{-1} d\xi = \int_{N_n \backslash G_n} \int_{V_n} f\left(\sigma_n \begin{pmatrix}
1 & X \\
0 & 1
\end{pmatrix}\begin{pmatrix}
g & 0 \\
0 & g
\end{pmatrix} \sigma_n^{-1}\right) dX dg,$ pour $f \in C^w(N_{2n} \backslash G_{2n}, \psi)$ (voir section \ref{fonctiontemperees} pour la definition de l'espace $C^w(N_{2n} \backslash G_{2n}, \psi)$),
\item $\int_{H^P_n \cap N_{2n} \backslash{H^P_n}} f(\xi) \theta(\xi)^{-1}d\xi = \int_{N_n \backslash P_n} \int_{V_n} f\left(\sigma_n \begin{pmatrix}
1 & X \\
0 & 1
\end{pmatrix}\begin{pmatrix}
g & 0 \\
0 & g
\end{pmatrix} \sigma_n^{-1}\right) dX dg,$ pour $f \in C^w(N_{2n} \backslash G_{2n}, \psi)$.
\end{itemize}

\begin{proposition}
\label{unfolding}
Soit $f \in \mathcal{S}(G_{2n})$, alors on a
\begin{equation}
\int_{H_n} f(s) \theta(s)^{-1} ds = \int_{H^P_n \cap N_{2n} \backslash{H^P_n}} \int_{H_n \cap N_{2n} \backslash{H_n}} W_f(\xi_p, \xi) \theta(\xi)^{-1} \theta(\xi_p) d\xi d\xi_p,
\end{equation}
où $W_f$ est la fonction de $G_{2n} \times G_{2n}$ définie par
\begin{equation}
W_f(g_1,g_2) = \int_{N_{2n}} f(g_1^{-1}ug_2) \psi(u)^{-1} du
\end{equation}
pour tous $g_1, g_2 \in G_{2n}$. De plus, l'intégrale double est absolument convergente.
\end{proposition}

\begin{proof}
On montre la proposition par récurrence sur $n$. Pour $n=1$, $\sigma_n$ est trivial, $H_1 = N_2 Z(G_2)$ et $H^P_1 = N_2$ donc $H^P_1 \cap N_2 \backslash{H^P_1}$ est trivial. Le membre de droite est alors
\begin{equation}
\int_{F^*} W_f \left(1, \begin{pmatrix}
z & 0 \\
0 & z
\end{pmatrix} \right) dz = \int_{F^*} \int_{N_2} f \left(u\begin{pmatrix}
z & 0 \\
0 & z
\end{pmatrix} \right) \psi(u)^{-1} du dz.
\end{equation}

Ce qui est bien l'égalité voulue. Supposons maintenant que $n > 1$ et que la proposition soit vraie au rang $n-1$.

Le sous groupe $\Omega_n$ des matrices de la forme
$\sigma_n \begin{pmatrix}
1 & Y \\
0 & 1
\end{pmatrix}\begin{pmatrix}
h & 0 \\
0 & h
\end{pmatrix} \sigma_n^{-1}$ où $Y$ est une matrice triangulaire inférieure stricte de taille $n$ et $h \in \overline{B}_n$ le sous-groupe des matrices triangulaires inférieures inversible, s'identifie à un ouvert dense du quotient $H_n \cap N_{2n} \backslash{H_n}$. On injecte $\Omega_{n-1}$ dans $\Omega_n$, en rajoutant des 0 sur la dernière ligne et colonne de $Y$ et voyant $h$ comme un élément de $\overline{B}_n$. On note $\widetilde{\Omega}_n$ l'ensemble des matrices de la
forme $\sigma_n \begin{pmatrix}
1 & \widetilde{Y} \\
0 & 1
\end{pmatrix}\begin{pmatrix}
\widetilde{h} & 0 \\
0 & \widetilde{h}
\end{pmatrix} \sigma_n^{-1}$
où $\widetilde{Y}$ est de la forme $\begin{pmatrix}
0_{n-1} & 0 \\
\widetilde{y} & 0
\end{pmatrix}$ avec $\widetilde{y} \in F^{n-1}$ et $\widetilde{h}$ de la forme $\begin{pmatrix}
1_{n-1} & 0 \\
\widetilde{l} & \widetilde{l}_n
\end{pmatrix}$ avec $\widetilde{l} \in F^{n-1}$ et $\widetilde{l}_n \in F^*$. Dans la suite, on fera l'identification de $F^{n-1} \times F^{n-1} \times F^*$ et $\widetilde{\Omega}_n$ à
travers l'isomorphisme $(\widetilde{y}, \widetilde{l}, \widetilde{l}_n) \in F^{n-1} \times F^{n-1} \times F^*\mapsto \sigma_n \begin{pmatrix}
1 & \widetilde{Y} \\
0 & 1
\end{pmatrix}\begin{pmatrix}
\widetilde{h} & 0 \\
0 & \widetilde{h}
\end{pmatrix} \sigma_n^{-1} \in \widetilde{\Omega}_n$ où $\widetilde{Y} = \begin{pmatrix}
0_{n-1} & 0 \\
\widetilde{y} & 0
\end{pmatrix}$ et $\widetilde{h} = \begin{pmatrix}
1_{n-1} & 0 \\
\widetilde{l} & \widetilde{l}_n
\end{pmatrix}$. On en déduit que $\Omega_n = \Omega_{n-1} \widetilde{\Omega}_n$. 

De même, on dispose d'une décomposition, $\Omega^P_n = \Omega^P_{n-1} \widetilde{\Omega}^P_n$, où $\Omega^P_n$ est l'ensemble des matrices de $\Omega_n$ avec $h \in P_n$ et $\widetilde{\Omega}^P_n$ est l'ensemble des matrices de la
forme $\sigma_n \begin{pmatrix}
1 & \widetilde{Z} \\
0 & 1
\end{pmatrix}\begin{pmatrix}
\widetilde{p} & 0 \\
0 & \widetilde{p}
\end{pmatrix} \sigma_n^{-1}$
où $\widetilde{Y}$ est de la forme $\begin{pmatrix}
0_{n-1} & 0 \\
\widetilde{z} & 0
\end{pmatrix}$ avec $\widetilde{z} \in F^{n-1}$ et $\widetilde{p}$ de la forme $\begin{pmatrix}
1_{n-2} & 0 & 0 \\
\widetilde{l} & \widetilde{l}_{n-1} & 0 \\
0 & 0 & 1
\end{pmatrix}$ avec $\widetilde{l} \in F^{n-2}$ et $\widetilde{l}_{n-1} \in F^*$. De plus, $\Omega^P_n$ s'identifie à un ouvert dense du quotient $H^P_n \cap N_{2n} \backslash{H^P_n}$. Dans la suite, on fera l'identification de $F^{n-1} \times F^{n-2} \times F^*$ et $\widetilde{\Omega}^P_n$ à
travers l'isomorphisme $(\widetilde{z}, \widetilde{l}, \widetilde{l}_{n-1}) \in F^{n-1} \times F^{n-2} \times F^* \mapsto \sigma_n \begin{pmatrix}
1 & \widetilde{Z} \\
0 & 1
\end{pmatrix}\begin{pmatrix}
\widetilde{p} & 0 \\
0 & \widetilde{p}
\end{pmatrix} \sigma_n^{-1} \in \widetilde{\Omega}^P_n$ où $\widetilde{Z} = \begin{pmatrix}
0_{n-1} & 0 \\
\widetilde{z} & 0
\end{pmatrix}$ et $\widetilde{p} = \begin{pmatrix}
1_{n-2} & 0 & 0 \\
\widetilde{l} & \widetilde{l}_{n-1} & 0 \\
0 & 0 & 1
\end{pmatrix}$.

On équipe $\Omega_n$, $\widetilde{\Omega}_n$, $\Omega^P_n$, $\widetilde{\Omega}^P_n$ des mesures suivantes :
\begin{itemize}
\item $\int_{\Omega_n} f(\xi) d\xi = \int_{\overline{B}_n} \int_{V_n} f\left(\sigma_n \begin{pmatrix}
1 & Y \\
0 & 1
\end{pmatrix}\begin{pmatrix}
h & 0 \\
0 & h
\end{pmatrix} \sigma_n^{-1}\right) dY dh,$ pour $f \in \mathcal{S}(G_{2n}),$

\item $\int_{\widetilde{\Omega}_n} f(\widetilde{\xi}) d\widetilde{\xi} = \int_{F_{n-1} \times F^*} \int_{F^{n-1}} f\left(\sigma_n \begin{pmatrix}
1 & \widetilde{Y} \\
0 & 1
\end{pmatrix}\begin{pmatrix}
\widetilde{h} & 0 \\
0 & \widetilde{h}
\end{pmatrix} \sigma_n^{-1}\right) d\widetilde{Y} d\widetilde{h},$ pour $f \in \mathcal{S}(G_{2n}),$

\item $\int_{\Omega^P_n} f(\xi_p) d\xi_p = \int_{\overline{B}_n \cap P_n} \int_{V_n} f\left(\sigma_n \begin{pmatrix}
1 & Z \\
0 & 1
\end{pmatrix}\begin{pmatrix}
p & 0 \\
0 & p
\end{pmatrix} \sigma_n^{-1}\right) dZ dp,$ pour $f \in \mathcal{S}(G_{2n}),$

\item $\int_{\widetilde{\Omega}^P_n} f(\widetilde{\xi}_p) d\widetilde{\xi}_p = \int_{F_{n-2} \times F^*} \int_{F^{n-1}} f\left(\sigma_n \begin{pmatrix}
1 & \widetilde{Z} \\
0 & 1
\end{pmatrix}\begin{pmatrix}
\widetilde{p} & 0 \\
0 & \widetilde{p}
\end{pmatrix} \sigma_n^{-1}\right) d\widetilde{Z} d\widetilde{p},$ pour $f \in \mathcal{S}(G_{2n}).$
\end{itemize}

On utilise ces décompositions pour écrire le membre de droite de la proposition sous la forme
\begin{equation}
\int_{\widetilde{\Omega}^P_n} \int_{\Omega^P_{n-1}} \int_{\widetilde{\Omega}_n} \int_{\Omega_{n-1}} W_f(\xi_p'\widetilde{\xi}_p, \xi'\widetilde{\xi}) |\det \xi_p'\xi'|^{-1} d\xi' d\widetilde{\xi} d\xi_p' d\widetilde{\xi}_p,
\end{equation}
On a choisi les représentants des matrices $Y$ et $\widetilde{Y}$ de sorte que le caractère $\theta$ soit trivial.

On fixe $\widetilde{\xi}_p \in \widetilde{\Omega}_n^P$ et $\widetilde{\xi} \in \widetilde{\Omega}_n$. On pose $f' = L(\widetilde{\xi}_p)R(\widetilde{\xi})f$, on a alors
 \begin{equation}
 \begin{split}
 & \int_{\Omega^P_{n-1}} \int_{\Omega_{n-1}} W_f(\xi_p'\widetilde{\xi}_p, \xi'\widetilde{\xi}) |\det \xi_p'\xi'|^{-1} d\xi' d\xi_p'= \\
 & \int_{\Omega^P_{n-1}} \int_{\Omega_{n-1}} W_{f'}(\xi_p', \xi') |\det \xi_p'\xi'|^{-1} d\xi' d\xi_p'.
 \end{split}
 \end{equation}

De plus,
 \begin{equation}
 W_{f'}(\xi_p', \xi') = \int_{N_{2n-2}} \int_V f'({\xi'}_p^{-1} v u \xi') \psi(u)^{-1}\psi(v)^{-1} dv du,
 \end{equation}
 où $V$ est le sous-groupe des matrices de $N_{2n}$ avec seulement les deux dernières colonnes non triviales, on dispose donc d'une décomposition $N_{2n} = N_{2n-2}V$. On effectue le changement de variable $v \mapsto {\xi'}_p v {\xi'}_p^{-1}$, ce qui donne
 \begin{equation}
 W_{f'}(\xi_p', \xi') = |\det \xi_p'|^{2}\int_{N_{2n-2}} \int_V f'(v {\xi'}_p^{-1} u \xi') \psi(u)^{-1}\psi(v)^{-1} dv du.
 \end{equation}

On note $\widetilde{f}'(g) = |\det g|^{-1}\int_V f'\left(v\begin{pmatrix}
g & 0 \\
0 & I_2
\end{pmatrix}\right) \psi(v)^{-1} dv$ pour $g \in G_{2n-2}$; alors $\widetilde{f}' \in \mathcal{S}(G_{2n-2})$. On obtient ainsi l'égalité
\begin{equation}
W_{f'}(\xi_p', \xi') = |\det \xi_p' \xi'| W_{\widetilde{f}'}(\xi_p', \xi').
\end{equation}

Appliquons l'hypothèse de récurrence,
 \begin{equation}
 \begin{split}
 & \int_{\Omega^P_{n-1}} \int_{\Omega_{n-1}} W_{f'}(\xi_p', \xi') |\det \xi_p'\xi'|^{-1} d\xi' d\xi_p' = \\
 & \int_{\Omega^P_{n-1}} \int_{\Omega_{n-1}} W_{\widetilde{f}'}(\xi_p', \xi') d\xi' d\xi_p' = \int_{H_{n-1}} \widetilde{f}'(s) \theta(s)^{-1} ds = \\
 & \int_{H_{n-1}} |\det s|^{-1} \int_V f(\widetilde{\xi}_p^{-1}v s \widetilde{\xi}) \theta(s)^{-1} \psi(v)^{-1} dv ds.
 \end{split}
 \end{equation}

Il nous faut maintenant intégrer sur $\widetilde{\xi}_p$ et $\widetilde{\xi}$ pour revenir à notre membre de droite. Explicitons l'intégrale sur $\widetilde{\xi}_p$ en le décomposant sous la forme $\sigma_n \begin{pmatrix}
1 & \widetilde{Z} \\
0 & 1
\end{pmatrix}\begin{pmatrix}
\widetilde{p} & 0 \\
0 & \widetilde{p}
\end{pmatrix} \sigma_n^{-1}$. On rappelle que l'on identifie $F^{n-1} \times F^{n-2} \times F^*$ et $\widetilde{\Omega}^P_n$ à
travers l'isomorphisme $(\widetilde{z}, \widetilde{l}, \widetilde{l}_{n-1}) \in F^{n-1} \times F^{n-2} \times F^* \mapsto \sigma_n \begin{pmatrix}
1 & \widetilde{Z} \\
0 & 1
\end{pmatrix}\begin{pmatrix}
\widetilde{p} & 0 \\
0 & \widetilde{p}
\end{pmatrix} \sigma_n^{-1} \in \widetilde{\Omega}^P_n$ où $\widetilde{Z} = \begin{pmatrix}
0_{n-1} & 0 \\
\widetilde{z} & 0
\end{pmatrix}$ et $\widetilde{p} = \begin{pmatrix}
1_{n-2} & 0 & 0 \\
\widetilde{l} & \widetilde{l}_{n-1} & 0 \\
0 & 0 & 1
\end{pmatrix}$. On obtient alors
\begin{equation}
\begin{split}
\int_{F^{n-2} \times F^*} \int_{F^{n-1}} \int_{\widetilde{\Omega}_n} \int_{H_{n-1}} |\det s|^{-1} \int_V &f\left(\sigma_n \begin{pmatrix}
\widetilde{p}^{-1} & 0 \\
0 & \widetilde{p}^{-1}
\end{pmatrix} \begin{pmatrix}
1 & -\widetilde{Z} \\
0 & 1
\end{pmatrix} \sigma_n^{-1} v s \widetilde{\xi}\right) \\
&\theta(s)^{-1} \psi(v)^{-1} dv ds d\widetilde{\xi} d\widetilde{Z} d\widetilde{p}.
\end{split}
\end{equation}

La conjugaison de $v$ par $\sigma_n^{-1}$ s'écrit sous la forme $\begin{pmatrix}
n_1 & y \\
t & n_2
\end{pmatrix}$ où $n_1, n_2$ sont dans $U_n$, les coefficients de $y$ sont nuls sauf la dernière colonne et $t$ est de la forme $\begin{pmatrix}
0_{n-1} & * \\
0 & 0
\end{pmatrix}$. Le caractère $\psi(v)$ devient après conjugaison $\psi(Tr(y)+Ts(t))$, où $Ts(t) = t_{n-1,n}$. Les changements de variables $\widetilde{Z} \mapsto \widetilde{p}\widetilde{Z}\widetilde{p}^{-1}$, $n_1 \mapsto \widetilde{p}n_1\widetilde{p}^{-1}$, $n_2 \mapsto \widetilde{p}n_2\widetilde{p}^{-1}$,
$t \mapsto \widetilde{p}t\widetilde{p}^{-1}$ et $y \mapsto \widetilde{p}y\widetilde{p}^{-1}$ transforme l'intégrale précédente en
\begin{equation}
\begin{split}
\int_{F^{n-2} \times F^*} \int_{F^{n-1}} \int_{\widetilde{\Omega}_n} \int_{H_{n-1}} & |\det s|^{-1}\int_{\sigma_n^{-1}V\sigma_n} f\left(\sigma_n \begin{pmatrix}
1 & -\widetilde{Z} \\
0 & 1
\end{pmatrix}  \begin{pmatrix}
n_1 & y \\
t & n_2
\end{pmatrix} \begin{pmatrix}
\widetilde{p}^{-1} & 0 \\
0 & \widetilde{p}^{-1}
\end{pmatrix} \sigma_n^{-1} s \widetilde{\xi}\right) \\
& \theta(s)^{-1} \psi(-Tr(y)) \psi(-Ts(\widetilde{p}t\widetilde{p}^{-1}))|\det \widetilde{p}|^3  d\begin{pmatrix}
n_1 & y \\
t & n_2
\end{pmatrix} ds d\widetilde{\xi} d\widetilde{Z} d\widetilde{p}.
\end{split}
\end{equation}

On explicite maintenant l'intégrale sur $s$ ce qui donne que $\sigma_n^{-1}s \sigma_n$ est de la forme $\begin{pmatrix}
1 & X \\
0 & 1
\end{pmatrix} \begin{pmatrix}
g & 0 \\
0 & g
\end{pmatrix}$ avec $X$ une matrice de taille $n$ dont la dernière ligne et dernière colonne sont nulles et $g \in G_{n-1}$ vu comme élément de $G_n$.
Le changement de variable $X \mapsto \widetilde{p}X\widetilde{p}^{-1}$ donne
\begin{equation}
\begin{split}
& \int_{F^{n-2} \times F^*} \int_{F^{n-1}} \int_{\widetilde{\Omega}_n} \int_{M_{n-1}} \int_{G_{n-1}}  |\det \widetilde{p}^{-1}g|^{-2}\int_{\sigma_n^{-1}V\sigma_n} \\
& f\left(\sigma_n \begin{pmatrix}
1 & -\widetilde{Z} \\
0 & 1
\end{pmatrix}  \begin{pmatrix}
n_1 & y \\
t & n_2
\end{pmatrix} \begin{pmatrix}
1 & X \\
0 & 1
\end{pmatrix} \begin{pmatrix}
\widetilde{p}^{-1} g & 0 \\
0 & \widetilde{p}^{-1} g
\end{pmatrix} \sigma_n^{-1} \widetilde{\xi}\right) \\
& \psi(-Tr(X)) \psi(-Tr(y)) \psi(-Ts(\widetilde{p}t\widetilde{p}^{-1}))  |\det \widetilde{p}| d\begin{pmatrix}
n_1 & y \\
t & n_2
\end{pmatrix} dg dX d\widetilde{\xi} d\widetilde{Z} d\widetilde{p}.
\end{split}
\end{equation}

On effectue maintenant le changement de variables $g \mapsto \widetilde{p}g$, notre intégrale devient alors
\begin{equation}
\begin{split}
& \int_{F^{n-2} \times F^*} \int_{F^{n-1}} \int_{\widetilde{\Omega}_n} \int_{M_{n-1}} \int_{G_{n-1}}  |\det g|^{-2}\int_{\sigma_n^{-1}V\sigma_n} \\
& f\left(\sigma_n \begin{pmatrix}
1 & -\widetilde{Z} \\
0 & 1
\end{pmatrix}  \begin{pmatrix}
n_1 & y \\
t & n_2
\end{pmatrix} \begin{pmatrix}
1 & X \\
0 & 1
\end{pmatrix} \begin{pmatrix}
g & 0 \\
0 & g
\end{pmatrix} \sigma_n^{-1} \widetilde{\xi}\right) \\
& \psi(-Tr(X)) \psi(-Tr(y)) \psi(-Ts(\widetilde{p}t\widetilde{p}^{-1}))  |\det \widetilde{p}| d\begin{pmatrix}
n_1 & y \\
t & n_2
\end{pmatrix} dg dX d\widetilde{\xi} d\widetilde{Z} d\widetilde{p}.
\end{split}
\end{equation}

\begin{lemme}
Soit $F \in \mathcal{S}(M_n)$, alors
\begin{equation}
\int_{F^{n-2} \times F^*} \int_{Lie(U_n)} F(t) \psi(-Ts(\widetilde{p}t\widetilde{p}^{-1}))|\det \widetilde{p}| dt d\widetilde{p} = F(0).
\end{equation}

On rappelle que l'on identifie $F^{n-2} \times F^*$ à l'ensemble des matrices de la forme $\begin{pmatrix}
1_{n-2} & 0 \\
\widetilde{l} & \widetilde{l}_{n-1}
\end{pmatrix}$ avec $\widetilde{l} \in F^{n-2}$ et $\widetilde{l}_{n-1} \in F^*$. 
\end{lemme}

\begin{proof}
La mesure $|\det \widetilde{p}| d\widetilde{p}$ correspond à la mesure additive sur $F^{n-1}$. En remarquant que $Ts(\widetilde{p}t\widetilde{p}^{-1})$ n'est autre que le produit scalaire des vecteurs dans $F^{n-1}$ correspondant à $\widetilde{p}$ et $t$, le lemme n'est autre qu'une
formule d'inversion de Fourier.
\end{proof}

Le lemme précédent nous permet de simplifier notre intégrale en
\begin{equation}
\begin{split}
\int_{F^{n-1}} \int_{\widetilde{\Omega}_n} \int_{M_{n-1}} \int_{G_{n-1}}  & |\det g|^{-2}\int_{\sigma_n^{-1}V_0\sigma_n} f\left(\sigma_n \begin{pmatrix}
1 & -\widetilde{Z} \\
0 & 1
\end{pmatrix}  \begin{pmatrix}
n_1 & y \\
0 & n_2
\end{pmatrix} \begin{pmatrix}
1 & X \\
0 & 1
\end{pmatrix} \begin{pmatrix}
g & 0 \\
0 & g
\end{pmatrix} \sigma_n^{-1} \widetilde{\xi}\right) \\
& \psi(-Tr(X)) \psi(-Tr(y))  d\begin{pmatrix}
n_1 & y \\
0 & n_2
\end{pmatrix} dg dX d\widetilde{\xi} d\widetilde{Z},
\end{split}
\end{equation}
où $\sigma_n^{-1}V_0\sigma_n$ est le sous-groupe de $\sigma_n^{-1}V\sigma_n$ où $t=0$.

On explicite l'intégration sur $\widetilde{\xi}$ de la forme $\sigma_n \begin{pmatrix}
1 & \widetilde{Y} \\
0 & 1
\end{pmatrix}\begin{pmatrix}
\widetilde{h} & 0 \\
0 & \widetilde{h}
\end{pmatrix} \sigma_n^{-1}$
où $\widetilde{Y}$ est une matrice de la forme $\begin{pmatrix}
0_{n-1} & 0 \\
\widetilde{y} & 0
\end{pmatrix}$ avec $\widetilde{y} \in F^{n-1}$ et $\widetilde{h} \in F^{n-1} \times F^*$ que l'on identifie avec un élément de $G_n$ dont seule la dernière ligne est non triviale.
Ce qui nous permet d'identifier $F^{n-1} \times F^{n-1} \times F^*$ et $\widetilde{\Omega}_n$. L'intégrale devient

\begin{equation}
\begin{split}
&\int_{F^{n-1}} \int_{F^{n-1}} \int_{F^{n-1} \times F^*} \int_{G_{n-1}} \int_{M_{n-1}}  |\det g|^{-2}\int_{\sigma_n^{-1} V_0 \sigma_n} \\
& f\left(\sigma_n \begin{pmatrix}
1 & -\widetilde{Z} \\
0 & 1
\end{pmatrix}  \begin{pmatrix}
n_1 & y \\
0 & n_2
\end{pmatrix} \begin{pmatrix}
1 & X \\
0 & 1
\end{pmatrix} \begin{pmatrix}
g & 0 \\
0 & g
\end{pmatrix} \begin{pmatrix}
1 & \widetilde{Y} \\
0 & 1
\end{pmatrix}\begin{pmatrix}
\widetilde{h} & 0 \\
0 & \widetilde{h}
\end{pmatrix} \sigma_n^{-1} \right)  \\
& \psi(-Tr(X)) \psi(-Tr(y)) d\begin{pmatrix}
n_1 & y \\
0 & n_2
\end{pmatrix} dX dg d\widetilde{h} d\widetilde{Y} d\widetilde{Z}.
\end{split}
\end{equation}

On remarque que l'on a
\begin{equation}
\begin{pmatrix}
n_1 & y \\
0 & n_2
\end{pmatrix} \begin{pmatrix}
1 & X \\
0 & 1
\end{pmatrix} \begin{pmatrix}
g & 0 \\
0 & g
\end{pmatrix} \begin{pmatrix}
1 & \widetilde{Y} \\
0 & 1
\end{pmatrix} = \begin{pmatrix}
n_1 & 0 \\
0 & n_2
\end{pmatrix}\begin{pmatrix}
1 & n_1^{-1}y + X + g\widetilde{Y}g^{-1} \\
0 & 1
\end{pmatrix}\begin{pmatrix}
g & 0 \\
0 & g
\end{pmatrix},
\end{equation}
 On effectue les changement de variable $y \mapsto n_1y$ et $\widetilde{Y} \mapsto g^{-1}\widetilde{Y} g$ et on combine les intégrales sur $X$, $y$ et $\widetilde{Y}$ en une intégration sur $M_n$ dont on note encore la variable $X$. On obtient alors

\begin{equation}
\label{combX}
\begin{split}
&\int_{F^{n-1}} \int_{F^{n-1} \times F^*} \int_{G_{n-1}} \int_{M_n}  |\det g|^{-1}\int_{U_n^2} \\
& f\left(\sigma_n \begin{pmatrix}
1 & -\widetilde{Z} \\
0 & 1
\end{pmatrix}  \begin{pmatrix}
n_1 & 0 \\
0 & n_2
\end{pmatrix} \begin{pmatrix}
1 & X \\
0 & 1
\end{pmatrix} \begin{pmatrix}
g\widetilde{h} & 0 \\
0 & g\widetilde{h}
\end{pmatrix} \sigma_n^{-1} \right)  \psi(-Tr(X)) d(n_1,n_2) dX dg d\widetilde{h} d\widetilde{Z}.
\end{split}
\end{equation}

On effectue le changement de variable $n_2 \mapsto n_2n_1$ et on remarque que l'on a
\begin{equation}
\begin{pmatrix}
1 & -\widetilde{Z} \\
0 & 1
\end{pmatrix}  \begin{pmatrix}
n_1 & 0 \\
0 & n_2n_1
\end{pmatrix} \begin{pmatrix}
1 & X \\
0 & 1
\end{pmatrix} = \begin{pmatrix}
1 & n_1Xn_1^{-1}-\widetilde{Z}n_2 \\
0 & n_2
\end{pmatrix}  \begin{pmatrix}
n_1 & 0 \\
0 & n_1
\end{pmatrix}.
\end{equation}

Le changement de variables $X \mapsto n_1^{-1}(X + \widetilde{Z}n_2)n_1$ nous donne alors
\begin{equation}
\begin{split}
\int_{F^{n-1}} \int_{F^{n-1} \times F^*} \int_{G_{n-1}} \int_{M_n}  & |\det g|^{-1}\int_{U_n^2} f\left(\sigma_n \begin{pmatrix}
1 & X \\
0 & n_2
\end{pmatrix} \begin{pmatrix}
n_1g\widetilde{h} & 0 \\
0 & n_1g\widetilde{h}
\end{pmatrix} \sigma_n^{-1} \right) \\
& \psi(-Tr(X))  \psi(-Tr(\widetilde{Z}n_2)) d(n_1,n_2) dX dg d\widetilde{h} d\widetilde{Z}.
\end{split}
\end{equation}

On reconnait une formule d'inversion de Fourier selon les variables $\widetilde{Z}$ et $n_2$ ce qui nous permet de simplifier notre intégrale en
\begin{equation}
\label{combg}
\begin{split}
\int_{F^{n-1} \times F^*} \int_{G_{n-1}} \int_{M_n}  |\det g|^{-1}\int_{U_n} & f\left(\sigma_n \begin{pmatrix}
1 & X \\
0 & 1
\end{pmatrix} \begin{pmatrix}
n_1g\widetilde{h} & 0 \\
0 & n_1g\widetilde{h}
\end{pmatrix} \sigma_n^{-1} \right) \\
& \psi(-Tr(X))  dn_1 dX dg d\widetilde{h}.
\end{split}
\end{equation}

Après combinaison des intégrations sur $n_1$, $g$, $\widetilde{h}$; on trouve bien notre membre de gauche
\begin{equation}
\int_{G_n} \int_{M_n}  f\left(\sigma_n \begin{pmatrix}
1 & X \\
0 & 1
\end{pmatrix} \begin{pmatrix}
g & 0 \\
0 & g
\end{pmatrix} \sigma_n^{-1} \right) \psi(-Tr(X))  dX dg.
\end{equation}

On remarquera que l'on a pris garde à ne pas échanger l'intégrale sur $V$ avec les intégrales sur $\widetilde{\Omega}_n$, $\widetilde{\Omega}^P_n$ et $H_{n-1}$ qui chacune est absolument convergente mais l'intégrale totale ne l'est pas. On s'est contenté d'échanger des intégrales sur $\widetilde{\Omega}_n$, $\widetilde{\Omega}^P_n$ et $H_{n-1}$ d'une part, d'échanger des intégrales sur les $n_1$, $n_2$, $t$, $y$ qui compose l'intégrale sur $V$ d'autre part. On doit seulement vérifier qu'il n'y a pas de problème de convergence lorsque l'on combine l'intégration en $X$ sur $M_n$ (cf. intégrale \ref{combX}) et lorsque l'on échange l'intégrale sur $U_n$ et $M_n$ (cf. intégrale \ref{combg}). Pour ce qui est de la dernière intégrale, on intègre sur un sous-groupe fermé  et $f \in \mathcal{S}(G_{2n})$ donc l'intégrale est absolument convergente. Pour ce qui est de l'intégrale \ref{combX}, à part l'intégration sur $\widetilde{Z}$, on intègre sur un sous-groupe fermé donc on peut bien combiner les intégrales.

Finissons par montrer la convergence absolue de notre membre de droite. Notons $r(g) = 1 + ||e_{2n}g||_\infty$. On a
\begin{equation}
\begin{split}
& W_{r^N |\det|^{-\frac{1}{2}} f}\left(\sigma_n \begin{pmatrix}
1 & X' \\
0 & 1
\end{pmatrix} \begin{pmatrix}
a'k' & 0 \\
0 & a'k'
\end{pmatrix} \sigma_n^{-1}, \sigma_n \begin{pmatrix}
1 & X \\
0 & 1
\end{pmatrix} \begin{pmatrix}
ak & 0 \\
0 & ak
\end{pmatrix} \sigma_n^{-1}\right) = \\
& (1+|a_n|)^N |\det a(a')^{-1}|^{-1} W_f\left(\sigma_n \begin{pmatrix}
1 & X' \\
0 & 1
\end{pmatrix} \begin{pmatrix}
a'k' & 0 \\
0 & a'k'
\end{pmatrix} \sigma_n^{-1}, \sigma_n \begin{pmatrix}
1 & X \\
0 & 1
\end{pmatrix} \begin{pmatrix}
ak & 0 \\
0 & ak
\end{pmatrix} \sigma_n^{-1}\right),
\end{split}
\end{equation}
pour tous $a \in A_n$, $a' \in A_{n-1}$, $X, X' \in V_n$, $k \in K_n$ et  $k' \in K_{n-1}$.

Il suffit de vérifier la convergence de l'intégrale
\begin{equation}
\label{conv}
\begin{split}
&\int_{V_n} \int_{A_{n-1}} \int_{V_n} \int_{A_n} (1+|a_n|)^{-N} |\det a(a')^{-1}| \\
& W_{r^N |\det|^{-\frac{1}{2}}f}\left(\sigma_n \begin{pmatrix}
1 & X' \\
0 & 1
\end{pmatrix} \begin{pmatrix}
a'k' & 0 \\
0 & a'k'
\end{pmatrix} \sigma_n^{-1}, \sigma_n \begin{pmatrix}
1 & X \\
0 & 1
\end{pmatrix} \begin{pmatrix}
ak & 0 \\
0 & ak
\end{pmatrix} \sigma_n^{-1}\right) \\
&\delta_{B_n}(a)^{-1} \delta_{B_{n-1}}(a')^{-1}da dX da' dX'
\end{split}
\end{equation}
pour $N$ suffisamment grand. On note $u_X  = \sigma_n \begin{pmatrix}
1 & X \\
0 & 1
\end{pmatrix} \sigma_n^{-1}$ et $u_{X'}  = \sigma_n \begin{pmatrix}
1 & X' \\
0 & 1
\end{pmatrix} \sigma_n^{-1}.$ On a alors
\begin{equation}
\sigma_n \begin{pmatrix}
1 & X \\
0 & 1
\end{pmatrix} \begin{pmatrix}
ak & 0 \\
0 & ak
\end{pmatrix} \sigma_n^{-1} = b u_{a^{-1}Xa}\sigma_n diag(k,k) \sigma_n^{-1},
\end{equation}
où $b = diag(a_1, a_1, a_2, a_2, ...)$ et \begin{equation}
\sigma_n \begin{pmatrix}
1 & X' \\
0 & 1
\end{pmatrix} \begin{pmatrix}
a'k' & 0 \\
0 & a'k'
\end{pmatrix} \sigma_n^{-1} = b' u_{{a'}^{-1}X'a'}\sigma_n diag(k' ,k') \sigma_n^{-1},
\end{equation}
où $b' = diag(a_1', a_1', a_2', a_2', ...)$.

On effectue les changements de variables $X \mapsto aXa^{-1}$ et $X' \mapsto a'X'{a'}^{-1}$. D'après le lemme \ref{majtemp} et la preuve du lemme \ref{convtemp}, il existe $d > 0$ tel que pour tout $N \geq 1$, l'intégrale \ref{conv} est alors majorée à une constante près par
\begin{equation}
\begin{split}
\int_{V_n} \int_{A_{n-1}} \int_{V_n} \int_{A_n} (1+|a_n|)^{-N} |\det a(a')^{-1}| m(X)^{-\alpha N} \prod_{i=1}^{n-1} (1 + |\frac{a_i}{a_{i+1}}|)^{-N} \delta^{\frac{1}{2}}_{B_{2n}}(bt_X)\log(||bt_X||)^d \\
m(X')^{-\alpha' N} \prod_{i=1}^{n-1} (1 + |\frac{a'_i}{a'_{i+1}}|)^{-N} \delta^{\frac{1}{2}}_{B_{2n}}(b't_{X'})\log(||b't_{X'}||)^d \delta_{B_n}^{-2}(a) \delta_{B_{n-1}}^{-2}(a') da dX da' dX'.
\end{split}
\end{equation}

Les quantités $m(X)$, $m(X')$, $\alpha$ et $\alpha'$ sont celles que l'on obtient par l'application de la proposition \ref{maj_mX}. On rappelle que $m(X) = sup(1, ||X||)$, où $||X|| = sup_{i,j} |X_{i,j}|$.
On a $\delta^{\frac{1}{2}}_{B_{2n}}(b')\delta_{B_{n-1}}^{-2}(a') = |\det a'|^2$. On en déduit que cette dernière intégrale est majorée (à constante près) par le maximum du produit des intégrales
 \begin{equation}
 \int_{V_n} m(X)^{-\alpha N} \delta^{\frac{1}{2}}_{B_{2n}}(t_X)\log(||t_X||)^{d-j} dX,
 \end{equation}
 
 \begin{equation}
 \int_{V_n} m(X')^{-\alpha' N} \delta^{\frac{1}{2}}_{B_{2n}}(t_{X'})\log(||t_{X'}||)^{d-j'} dX',
 \end{equation}
 
 \begin{equation}
 \int_{A_n}  \prod_{i=1}^{n-1} (1+ |\frac{a_i}{a_{i+1}}|)^{-N} (1+|a_n|)^{-N}\log(||b||)^j|\det a| da,
 \end{equation}
 
 et
 \begin{equation}
 \int_{A_{n-1}}  \prod_{i=1}^{n-2} (1+ |\frac{a_i'}{a_{i+1}'}|)^{-N} (1+|a_{n-1}'|)^{-N}\log(||b'||)^{j'}|\det a'|da',
 \end{equation}
 pour $j, j'$ compris entre $0$ et $d$. Ces dernières intégrales convergent pour $N$ assez grand, voir \cite[proposition 5.5]{jacquet-shalika} pour les deux premières intégrales et le lemme \ref{convergenceAn} pour les deux dernières.

 \end{proof}

 \section{Formules de Plancherel}
 
 \label{plancherel}
 Pour $W \in C^w(N_{2n} \backslash G_{2n}, \psi)$ (voir section \ref{fonctiontemperees}) , on note
\begin{equation}
\label{beta}
\beta(W) = \int_{H^P_n \cap N_{2n} \backslash H^P_n} W(\xi_p) \theta(\xi_p)^{-1} d\xi_p,
\end{equation}
voir la section \ref{formInv} pour la mesure $d\xi_p$.

\begin{lemme}
\label{betann}
L'intégrale \ref{beta} est absolument convergente. La forme linéaire $W \in \mathcal{C}^w(N_{2n} \backslash G_{2n}) \mapsto \beta(W)$ est continue.
\end{lemme}

\begin{proof}
D'après la décomposition d'Iwasawa, $P_n = N_n A_{n-1} K^P_n$, où $K^P_n$ est un sous-groupe compact, il suffit de montrer la convergence de l'intégrale
\begin{equation}
\int_{V_n} \int_{A_{n-1}} \left|W\left(\sigma_n \begin{pmatrix}
1 & X \\
0 & 1
\end{pmatrix} \begin{pmatrix}
ak & 0 \\
0 & ak
\end{pmatrix}\sigma_n^{-1}\right)\right| \delta_{B_{n-1}}(a)^{-1} da dX,
\end{equation}
pour tout $k \in K^P_n$. Par un argument similaire à la preuve du lemme \ref{convtemp}, on obtient la majoration suivante :
\begin{equation}
\begin{split}
\int_{V_n} \int_{A_{n-1}} &\prod_{i=1}^{n-2}(1+\frac{|a_i|}{|a_{i+1}|})^{-N} m(X)^{-\alpha N} \delta_{B_{2n}}(bt_X)^{\frac{1}{2}} \\
&\log(||bt_X||)^d \delta_{B_n}(a)^{-1} \delta_{B_{n-1}}(a)^{-1} da dX,
\end{split}
\end{equation}
pour tout $N \geq 1$. Cette dernière intégrale est convergente pour $N$ suffisamment grand par le même argument que dans la preuve du lemme \ref{convtemp}.
\end{proof}

\begin{proposition}
\label{lemmebeta}
Pour $\pi = T(\sigma)$ avec $\sigma \in Temp(SO(2n+1))$, la restriction de $\beta$ à $\mathcal{W}(\pi, \psi)$ est un élément de $\Hom_{H_n}(\mathcal{W}(\pi, \psi), \theta)$.
\end{proposition}

La preuve de cette proposition se fera après quelques préliminaires. On commence par prouver un lemme et introduire des notations.

On note $\mathcal{S}(Z_{2n}N_{2n} \backslash G_{2n})$ l'ensemble des fonctions lisse sur $G_{2n}$, $Z_{2n}N_{2n}$-équivariante à gauche et à support compact modulo $Z_{2n}N_{2n}$.
\begin{lemme}
\label{limitezeta}
Pour $W \in \mathcal{S}(Z_{2n} N_{2n} \backslash G_{2n})$ et $\phi \in \mathcal{S}(F^n)$, on a
\begin{equation}
\lim_{s\rightarrow 0^+} \gamma(ns, 1, \psi) J(s, W, \phi) = \phi(0) \int_{Z_{2n}(H_n \cap N_{2n}) \backslash H_n} W(\xi) \theta(\xi)^{-1} d\xi.
\end{equation}
\end{lemme}

\begin{proof}
On a
\begin{equation}
\begin{split}
\gamma(ns, 1, \psi) J(s, W, \phi) &= \int_{A_{n-1}} \int_{K_n} \int_{V_n} W\left(\sigma_n\begin{pmatrix}
1 & X \\
0 & 1
\end{pmatrix} \begin{pmatrix}
ak & 0 \\
0 & ak
\end{pmatrix} \sigma_n^{-1}\right) dX\\
& \gamma(ns, 1, \psi) \int_{Z_n}\phi(e_nzk) |\det z|^s dz dk |\det a|^s \delta_{B_n}(a)^{-1} da
\end{split}
\end{equation}

De plus, d'après la thèse de Tate, on a
\begin{equation}
\gamma(ns, 1, \psi) \int_{Z_n} \phi(e_n zk) |\det z|^s dz = \int_{F} \widehat{\phi_k}(x)|x|^{-ns} dx,
\end{equation}
où l'on a posé $\phi_k(x) = \phi(xe_nk)$ pour tous $x \in F$ et $k \in K_n$. Ce qui nous donne par convergence dominée
\begin{equation}
\lim_{s \rightarrow 0+} \gamma(ns, 1, \psi)\int_{Z_n} \phi(e_nzk) |\det z|^s dz = \int_{F} \widehat{\phi_k}(x)dx = \phi(0).
\end{equation}

D'autre part, $W$ est à support compact modulo $Z_{2n}N_{2n}$ et l'intersection de $Z_{2n}N_{2n}$ avec l'ensemble des matrices de la forme $\sigma_n\begin{pmatrix}
1 & X \\
0 & 1
\end{pmatrix} \begin{pmatrix}
ak & 0 \\
0 & ak
\end{pmatrix} \sigma_n^{-1}$ pour $X \in V_n, a \in A_{n-1}, k \in K_n$ est à support compact. On en déduit que l'intégrale
\begin{equation}
\int_{A_{n-1}} \int_{V_n} W\left(\sigma_n\begin{pmatrix}
1 & X \\
0 & 1
\end{pmatrix} \begin{pmatrix}
ak & 0 \\
0 & ak
\end{pmatrix} \sigma_n^{-1}\right) dX |\det a|^s \delta_{B_n}(a)^{-1} da,
\end{equation}
converge absolument pour tout $k \in K_n$ et tout $s \in \mathbb{C}$.

On en déduit par convergence dominée, que $\lim_{s \rightarrow 0^+}\gamma(ns, 1, \psi) J(s, W, \phi)$ est égal à
\begin{equation}
\phi(0)\int_{Z_n \backslash A_n} \int_{K_n} \int_{V_n}
W\left(\sigma_n\begin{pmatrix}
1 & X \\
0 & 1
\end{pmatrix} \begin{pmatrix}
ak & 0 \\
0 & ak
\end{pmatrix} \sigma_n^{-1}\right)dX  dk \delta_{B_n}(a)^{-1} da,
\end{equation}
ce qui nous permet de conclure.
\end{proof}

Pour tout $g \in G_{2n}$, on pose $||g|| = max(||g||_\infty, ||g^{-1}||_\infty)$ et $\sigma(g) = \log(||g||)$, avec $||g||_\infty = \sup_{i,j} |g_{i,j}|$ où les $g_{i,j}$ sont les coefficients de $g$. On note $\Xi$ la fonction Xi d'Harish-Chandra (voir \cite{waldspurger}).

On introduit l'espace $C^w(G_{2n})$ des fonctions tempérées sur $G_{2n}$ comme l'ensemble des fonctions $f : G_{2n} \rightarrow \mathbb{C}$ qui sont biinvariante par un sous-groupe compact ouvert et vérifient qu'il existe $d \geq 1$ et $C > 0$ tel que
\begin{equation}
\label{ineqTemp}
|f(g)| \leq C \Xi(g) \sigma(g)^d,
\end{equation}
pour tout $g \in G_{2n}$. On note $C^w_{d,K}(G_{2n})$ le sous-espace de $C^w(G_{2n})$ des fonctions qui vérifient \ref{ineqTemp} et qui sont biinvariante par $K$, où $K$ est un sous-groupe compact ouvert de $G_{2n}$. Ce dernier est muni de la norme $\sup_{g \in G_{2n}} \frac{|f(g)|}{\Xi(g) \sigma(g)^d}$ qui en fait un espace de Banach. On munit $C^w(G_{2n}) = \bigcup_{d,K} C^w_{d,K}(G_{2n})$ de la topologie limite inductive.

On étend la forme linéaire $f \in \mathcal{S}(G_{2n}) \mapsto \int_{N_{2n}} f(u)\psi(u)^{-1} du$ par continuité en une forme linéaire sur $C^w(G_{2n})$ \cite{beuzart-plessis}, que l'on note
\begin{equation}
f \in C^w(G_{2n}) \mapsto \int_{N_{2n}}^* f(u)\psi(u)^{-1} du.
\end{equation}

Pour $f \in C^w(G_{2n})$, on peut ainsi définir $W_f$ par la formule
\begin{equation}
W_f(g_1, g_2) = \int_{N_{2n}}^* f(g_1^{-1}ug_2)\psi(u)^{-1} du,
\end{equation}
pour tous $g_1, g_2 \in G_{2n}$.

Soient $f \in \mathcal{S}(G_{2n})$ et $\pi \in Temp(G_{2n})$, on pose $W_{f, \pi} = W_{f_\pi}$ où $f_\pi(g) = Tr(\pi(g)\pi(f^\vee))$ pour tout $g \in G_{2n}$.

On introduit l'espace $C^w(N_{2n} \times N_{2n} \backslash G_{2n} \times G_{2n}, \psi \otimes \psi ^{-1})$ de manière analogue à l'espace des fonctions tempérées sur $N_{2n} \backslash G_{2n}$ dans la section \ref{fonctiontemperees}. On a
\begin{equation}
C^w(N_{2n} \times N_{2n} \backslash G_{2n} \times G_{2n}, \psi \otimes \psi ^{-1}) = C^w(N_{2n} \backslash G_{2n}, \psi) \widehat{\otimes} C^w(N_{2n}  \backslash G_{2n}, \psi ^{-1}),
\end{equation}
où $\widehat{\otimes}$ est le produit tensoriel complété \cite[A.5]{bpannexe}.

\begin{proposition}[Beuzart-Plessis {\cite[Proposition 2.14.2]{beuzart-plessis}}]
\label{decspecwhitt}
Soit $f \in \mathcal{S}(G_{2n})$. On pose $\widetilde{f}(g) = \int_{Z_n} f(zg) dz$, alors $\widetilde{f} \in \mathcal{S}(PG_{2n})$. On a $\widetilde{f}_\pi = f_\pi$ pour tout $\pi \in Temp(PG_{2n})$. De plus, on a l'égalité
\begin{equation}
W_{\widetilde{f}} = \int_{Temp(PG_{2n})} W_{f, \pi} d\mu_{PG_{2n}}(\pi)
\end{equation}
dans $C^w(N_{2n} \times N_{2n} \backslash G_{2n} \times G_{2n}, \psi \otimes \psi ^{-1})$. On renvoie à \cite[Annexe A.2]{bpannexe} pour la définition de l'intégrale à valeur vectorielle.
\end{proposition}

\begin{lemme}
\label{choixf}
Soit $W \in \mathcal{W}(\pi, \psi)$, alors il existe $f \in \mathcal{S}(G_{2n})$ tel que $W_{f, \pi}(1,.) = W$.
\end{lemme}

\begin{proof}
On a
\begin{equation}
W_{f, \pi}(1, .) = \int^*_{N_{2n}} f_\pi(u.) \psi(u)^{-1} du.
\end{equation}

D'autre part, soit $f \in \mathcal{S}(G_{2n})$ alors $f$ est bi-invariante par un sous-groupe ouvert compact $K$. On a une décomposition $V_{\pi} = V_{\pi}^K \oplus V_\pi(K)$, où $V_\pi^K$ est l'espace des vecteurs $K$-invariants. Comme $\pi$ est admissible, $V_\pi^K$ est de dimension finie. On note $\mathcal{B}_\pi^K$ une base de cet espace. Alors pour tout $g \in G_{2n}$, on a $f_\pi(g) = Tr(\pi(g) \pi(f^\vee)) = \sum_{v \in \mathcal{B}_\pi^K} <\pi(g)\pi(f^\vee)v, v^\vee>$, où $(v^\vee)_{v \in \mathcal{B}_\pi^K}$ est la base duale de $\mathcal{B}_\pi^K$. On en déduit que $f_\pi$ est une somme (finie) de coefficient matriciel.

On note $Coeff^K = \{g \mapsto <\pi(g)v, \widetilde{v}>, v \in V_\pi^K, \widetilde{v} \in V_{\widetilde{\pi}}^K\}$.
Alors toute combinaison linéaire de $Coeff^K$ est de la forme $f_\pi$ avec $f \in \mathcal{S}(G_{2n}, K)$, où l'on a noté $\mathcal{S}(G_{2n}, K)$ le sous espace de $\mathcal{S}(G_{2n}, K)$ des fonctions bi-invariante par $K$. En effet, $f \in \mathcal{S}(G_{2n}, K) \mapsto \pi(f^\vee) \in End(V_\pi^K)$ est surjective. La surjectivité est une conséquence du lemme de Burnside et du fait que $V_\pi^K$ est un $\mathcal{S}(G_{2n}, K)$-module irréductible de dimension finie. Par adjonction, on a $End(V_\pi^K) \simeq \pi^K \boxtimes \widetilde{\pi}^K$, d'où le résultat.

Pour montrer le lemme, il nous faut montrer qu'il existe un coefficient matriciel $c = <\pi(.)v, \widetilde{v}>$ tel que $W = \int^*_{N_{2n}} c(u.) \psi(u)^{-1}du$. Or
\begin{equation}
v \mapsto \int^*_{N_{2n}} c(u.) \psi(u)^{-1} du = \int^*_{N_{2n}} <\pi(u.)v, \widetilde{v}> \psi(u)^{-1} du
\end{equation}
est une fonctionnelle de Whittaker. Il suffit donc de montrer que l'on peut choisir $\widetilde{v}$ pour que cette fonctionnelle soit non nulle.
C'est le contenu de \cite[Théorème 6.4.1]{sak-ven}.
\end{proof}

Pour $W \in C^w(N_{2n} \backslash G_{2n}, \psi^{-1})$, on note
\begin{equation}
\widetilde{\beta}(W) = \int_{H_n^P \cap N_{2n} \backslash H_n^P} W(\xi_p) \theta(\xi_p) d\xi_p.
\end{equation}
On dispose du lemme \ref{betann} pour la forme linéaire $\widetilde{\beta}$ avec la même preuve.

Pour $\sigma \in Temp(SO(2n+1))$, on pose $\pi = T(\sigma)$. On montre dans la proposition \ref{constbeta} qu'il existe un signe $c_\beta(\sigma)$ tel que
\begin{equation}
\widetilde{\beta}(\rho(w_{n,n})\widetilde{W}) = c_\beta(\sigma) \beta(W),
\end{equation}
pour tout $W \in \mathcal{W}(\pi, \psi)$.

\begin{corollaire}[de la limite spectrale]
\label{corolim}
Soient $f \in \mathcal{S}(G_{2n})$ et $g \in G_{2n}$, alors
\begin{equation}
\begin{split}
\int_{H_n \cap N_{2n} \backslash H_n} W_f(g, \xi) \theta(\xi)^{-1} d\xi = &\int_{Temp(SO(2n+1))/Stab} \beta(W_{f,T(\sigma)}(g,.)) \\
&\frac{\gamma^*(0, \sigma, Ad, \psi)}{|S_\sigma|} c(T(\sigma))^{-1} c_\beta(\sigma) d\sigma,
\end{split}
\end{equation}
où $c(T(\sigma))$ est la constante qui provient du théorème \ref{egalitegamma}.
\end{corollaire}

\begin{proof}
On peut supposer que $g = 1$ en remplaçant $f$ par $L(g)f$. On pose $\widetilde{f}(g) = \int_{Z_n} f(zg) dz$, alors $\widetilde{f} \in \mathcal{S}(PG_{2n})$. On a donc
\begin{equation}
\int_{H_n \cap N_{2n} \backslash H_n} W_f(1, \xi) \theta(\xi)^{-1} d\xi = \int_{Z_{2n}(H_n \cap N_{2n}) \backslash H_n} W_{\widetilde{f}}(1, \xi) \theta(\xi)^{-1} d\xi.
\end{equation}

On choisit $\phi \in \mathcal{S}(F^n)$ tel que $\phi(0) = 1$. D'après le lemme \ref{limitezeta}, la proposition \ref{decspecwhitt} et le lemme \ref{convtemp}, on a
\begin{equation}
\begin{split}
\int_{Z_{2n}(H_n \cap N_{2n}) \backslash H_n} W_{\widetilde{f}}(1, \xi) \theta(\xi)^{-1} d\xi &= \lim_{s\rightarrow 0^+} n\gamma(s, 1, \psi) J(s, W_{\widetilde{f}}(1, .), \phi) \\
&= \lim_{s\rightarrow 0^+} n\gamma(s, 1, \psi) \int_{Temp(PG_{2n})}J(s, W_{f, \pi}(1, .), \phi) d\mu_{PG_{2n}}(\pi).
\end{split}
\end{equation}

D'après l'équation fonctionnelle \ref{funcloc} et le théorème \ref{egalitegamma}, on a
\begin{equation}
\begin{split}
&\int_{H_n \cap N_{2n} \backslash H_n} W_f(1, \xi) \theta(\xi)^{-1} d\xi = \\
&\lim_{s\rightarrow 0^+} n\gamma(s, 1, \psi) \int_{Temp(PG_{2n})}J(1-s, \rho(w_{n,n})\widetilde{W_{f, \pi}(1, .)}, \widehat{\phi}) c(\pi)^{-1} \gamma(s, \pi, \Lambda^2, \psi)^{-1} d\mu_{PG_{2n}}(\pi).
\end{split}
\end{equation}

La proposition \ref{limitespectrale}, nous permet d'obtenir la relation
\begin{equation}
\label{relspec}
\begin{split}
&\int_{H_n \cap N_{2n} \backslash H_n} W_f(1, \xi) \theta(\xi)^{-1} d\xi =\\
&\int_{Temp(SO(2n+1)/Stab} J(1, \rho(w_{n,n})\widetilde{W_{f, T(\sigma)}(1,.)}, \widehat{\phi}) c(T(\sigma))^{-1} \frac{\gamma^*(0, \sigma, Ad, \psi)}{|S_\sigma|} d\sigma.
\end{split}
\end{equation}

En remplaçant $f$ par $\rho(h)f$, $h \in H_n$, dans le membre de gauche; cela revient à multiplier par $\theta(h)$. On en déduit la même relation pour le membre de droite. Ce qui signifie que
\begin{equation}
\int_{Temp(SO(2n+1)/Stab} J(1, \rho(w_{n,n})\widetilde{W_{\rho(\xi)f, T(\sigma)}(1,.)}, \widehat{\phi})-\theta(\xi)J(1, \rho(w_{n,n})\widetilde{W_{f, T(\sigma)}(1,.)}, \widehat{\phi}) d\mu(\sigma) = 0,
\end{equation}
pour tout $\xi \in H_n$, où $d\mu(\sigma) = c(T(\sigma))^{-1} \frac{\gamma^*(0, \sigma, Ad, \psi)}{|S_\sigma|} d\sigma$.

D'après le lemme de séparation spectrale \cite[Lemme 5.7.2]{beuzart-plessis} et la continuité de $\sigma \mapsto J(1, \rho(w_{n,n})\widetilde{W_{f, T(\sigma)}}(1, .), \widehat{\phi})$, on en déduit que
$J(1, \rho(w_{n,n})\widetilde{W_{\rho(\xi)f, T(\sigma)}(1,.)}, \widehat{\phi}) = \theta(\xi)J(1, \rho(w_{n,n})\widetilde{W_{f, T(\sigma)}(1,.)}, \widehat{\phi})$ pour tous $\xi \in H_n$, $\sigma \in Temp(SO(2n+1)$ et donc que $f \mapsto J(1, \rho(w_{n,n})\widetilde{W_{f, T(\sigma)}(1,.)}, \widehat{\phi})$ est $(H_n, \theta)$-équivariante.

\begin{lemme}
\label{zetabeta}
Soient $\sigma \in Temp(SO(2n+1))$ et $\pi = T(\sigma)$. Alors
\begin{equation}
J(1, \rho(w_{n,n})\widetilde{W}, \widehat{\phi}) = \phi(0)c_\beta(\sigma)\beta(W),
\end{equation}
pour tous $W \in \mathcal{W}(\pi, \psi)$ et $\phi \in \mathcal{S}(F^n)$. Rappelons que $c_{\beta}(\sigma)$ est le signe qui vérifie la relation
\begin{equation}
\widetilde{\beta}(\rho(w_{n,n})\widetilde{W}) = c_\beta(\sigma) \beta(W),
\end{equation}
pour tous $W \in \mathcal{W}(\pi, \psi)$.
\end{lemme}

\begin{proof}
En effet, soit $W \in \mathcal{W}(\pi, \psi)$, on a
\begin{equation}
\begin{split}
J(1, \widetilde{W}, \widehat{\phi}) = \int_{N_n \backslash G_n} \int_{V_n} &\widetilde{W}\left(\sigma_n\begin{pmatrix}
1 & X \\
0 & 1
\end{pmatrix} \begin{pmatrix}
g & 0 \\
0 & g
\end{pmatrix} \sigma_n^{-1}\right)dX \\
& \widehat{\phi}(e_ng) |\det g| dg.
\end{split}
\end{equation}

D'après le lemme \ref{choixf}, on choisit $f \in \mathcal{S}(G_{2n})$ tel que $W_{f,\pi}(1,.) = W$. D'après ce que l'on vient de dire précédemment, on en déduit que $J(1, \rho(w_{n,n})\widetilde{W}, \widehat{\phi})$ vérifie la relation $J(1, \rho(w_{n,n})\widetilde{\rho(\xi)W}, \widehat{\phi}) = \theta(\xi)J(1, \rho(w_{n,n})\widetilde{W}, \widehat{\phi})$, ou encore $J(1, \rho(\xi)\rho(w_{n,n})\widetilde{W}, \widehat{\phi}) = \theta(\xi)^{-1}J(1, \rho(w_{n,n})\widetilde{W}, \widehat{\phi})$, pour tout $\xi \in H_n$.

Comme $\widehat{\phi}(e_ng)$ est arbitraire parmi les fonctions invariante à gauche par $P_n$, on en déduit que
\begin{equation}
\begin{split}
\widetilde{W} \in \mathcal{W}(\widetilde{\pi}, \psi^{-1}) \mapsto &\int_{N_n \backslash P_n} \int_{V_n} \widetilde{W}\left(\sigma_n\begin{pmatrix}
1 & X \\
0 & 1
\end{pmatrix} \begin{pmatrix}
g & 0 \\
0 & g
\end{pmatrix} \sigma_n^{-1}\right) dX dg \\
& = \widetilde{\beta}(\widetilde{W})
\end{split}
\end{equation}
est $(H_n, \theta^{-1})$-équivariante. Ce qui nous permet d'en déduire que $\widetilde{\beta}$ restreint à $\mathcal{W}(\widetilde{\pi}, \psi^{-1})$ est $(H_n, \theta^{-1})$-équivariante. En remplaçant $\psi$ par $\psi^{-1}$ et comme $\widetilde{\pi} \simeq \pi$ d'après la caractérisation \ref{caracTransf}, on en déduit que $\beta$ restreint à $\mathcal{W}(\pi, \psi)$ est $(H_n, \theta)$-équivariante, ce qui termine la preuve de la proposition \ref{lemmebeta}.

\begin{remarque}
Cette preuve que $\beta$ restreint à $\mathcal{W}(\pi, \psi)$ est $(H_n, \theta)$-équivariante est quelque peu détournée dû au fait qu'il nous manque un résultat. On conjecture que, pour $\pi$ tempérée, $Hom_{H_n \cap P_{2n}}(\pi, \theta)$ est de dimension au plus $1$. Comme $\pi = T(\sigma)$, d'après \ref{caracTransf}, la classification des séries discrètes admettant un modèle de Shalika \cite{matringe3} et le passage à l'induction parabolique \cite{matringe2}, on obtient que $\pi$ admet un modèle de Shalika. Autrement dit, on a $Hom_{H_n}(\pi, \theta) \neq 0$. Ce dernier est un sous-espace de $Hom_{H_n \cap P_{2n}}(\pi, \theta)$. On en déduirait alors que la restriction de $\beta$ à $\mathcal{W}(\pi, \psi)$, qui est bien $(H_n \cap P_{2n}, \theta)$-équivariante, est un élément de $Hom_{H_n}(\pi, \theta)$. Ce qui simplifierait légèrement la preuve à condition de prouver le résultat de dimension $1$.
\end{remarque}

\begin{proposition}
\label{constbeta}
Soit $\sigma \in Temp(SO(2n+1))$, on pose $\pi = T(\sigma)$ le transfert de $\sigma$ dans $Temp(G_{2n})$. La forme linéaire $W \in \mathcal{W}(\pi, \psi) \mapsto \widetilde{\beta}(\rho(w_{n,n})\widetilde{W})$ est un élément de $\Hom_{H_n}(\mathcal{W}(\pi, \psi), \theta)$. Il existe un signe $c_\beta(\sigma) = c_\beta(\pi)$ tel que
\begin{equation}
\label{eqconstbeta}
\widetilde{\beta}(\rho(w_{n,n})\widetilde{W}) = c_\beta(\sigma)\beta(W),
\end{equation}
pour tout $W \in \mathcal{W}(\pi, \psi)$.
\end{proposition}

\begin{proof}

On a montré que $\widetilde{W} \in \mathcal{W}(\widetilde{\pi}, \psi^{-1}) \mapsto \widetilde{\beta}(\widetilde{W})$ est $(H_n, \theta^{-1})$-équivariante. L'isomorphisme $W \in \mathcal{W}(\pi, \psi) \mapsto \rho(w_{n,n})\widetilde{W} \in \mathcal{W}(\widetilde{\pi}, \psi^{-1})$ transforme le caractère $\theta$ en $\theta^{-1}$.
On en déduit que $W \in \mathcal{W}(\pi, \psi) \mapsto \widetilde{\beta}(\rho(w_{n,n})\widetilde{W})$ est $(H_n, \theta)$-équivariante.

Montrons maintenant l'égalité \ref{eqconstbeta}. En effet, $\Hom_{H_n}(\pi, \theta)$ est de dimension au plus 1, d'après l'unicité du modèle de Shalika \cite{jacquet-rallis}. On en déduit l'existence de $c_\beta(\pi) \in \mathbb{C}$ vérifiant $\widetilde{\beta}(\rho(w_{n,n})\widetilde{W}) = c_\beta(\sigma)\beta(W)$
pour tout $W \in \mathcal{W}(\pi, \psi)$. En remplaçant $\psi$ par $\psi^{-1}$ l'équation \ref{eqconstbeta} devient $\beta(\rho(w_{n,n})\widetilde{W}) = c_\beta(\pi) \widetilde{\beta}(W)$ pour tout $W \in \mathcal{W}(\pi, \psi^{-1})$. L'isomorphisme $W \in \mathcal{W}(\pi, \psi^{-1}) \mapsto \rho(w_{n,n})\widetilde{W} \in \mathcal{W}(\widetilde{\pi}, \psi^{-1})$ donne alors $\beta(W) = c_\beta(\pi) \widetilde{\beta}(\rho(w_{n,n})\widetilde{W})$. Comme $\pi \simeq \widetilde{\pi}$, on identifie $\mathcal{W}(\pi, \psi)$ et $\mathcal{W}(\widetilde{\pi}, \psi)$. On en déduit que $\widetilde{\beta}(\rho(w_{n,n})\widetilde{W}) = c_\beta(\pi)^2 \widetilde{\beta}(\rho(w_{n,n})\widetilde{W})$. Si $\widetilde{\beta}$ est nulle sur $\mathcal{W}(\widetilde{\pi}, \psi^{-1})$ le résultat est vrai, sinon $c_\beta(\pi)$ est un signe d'après la relation précédente.
\end{proof}

Finissons la preuve du lemme \ref{zetabeta}, on remarque que l'on a
\begin{equation}
\begin{split}
&\int_{N_n \backslash G_n} \int_{V_n} \widetilde{W}\left(\sigma_n\begin{pmatrix}
1 & X \\
0 & 1
\end{pmatrix} \begin{pmatrix}
g & 0 \\
0 & g
\end{pmatrix} \sigma_n^{-1}\right) dX \widehat{\phi}(e_ng) |\det g| dg \\
&= \int_{P_n \backslash G_n} \int_{N_n \backslash P_n} \int_{V_n} \widetilde{W}\left(\sigma_n\begin{pmatrix}
1 & X \\
0 & 1
\end{pmatrix} \begin{pmatrix}
ph & 0 \\
0 & ph
\end{pmatrix} \sigma_n^{-1}\right) dX dp \widehat{\phi}(e_nh) |\det h| dh.
\end{split}
\end{equation}

De plus,
\begin{equation}
\begin{split}
\int_{N_n \backslash P_n} \int_{V_n} &\widetilde{W}\left(\sigma_n\begin{pmatrix}
1 & X \\
0 & 1
\end{pmatrix} \begin{pmatrix}
ph & 0 \\
0 & ph
\end{pmatrix} \sigma_n^{-1}\right) dX dp \\
&= \widetilde{\beta}\left(\rho\left(\sigma_n \begin{pmatrix}
h & 0 \\
0 & h
\end{pmatrix} \sigma_n^{-1}\right) \widetilde{W}\right) \\
&= \widetilde{\beta}(\widetilde{W}),
\end{split}
\end{equation}
puisque $\widetilde{\beta}$ est $(H_n, \theta^{-1})$-équivariante. D'autre part,
\begin{equation}
\begin{split}
\int_{P_n \backslash G_n}  \widehat{\phi}(e_nh) |\det h| dh &= \int_{F^n} \widehat{\phi}(x) dx \\
&= \phi(0).
\end{split}
\end{equation}

On en déduit que $J(1, \rho(w_{n,n})\widetilde{W}, \widehat{\phi}) = \phi(0)\widetilde{\beta}(\rho(w_{n,n}\widetilde{W})$.
On conclut grâce à la proposition \ref{constbeta}.
\end{proof}

Pour finir la preuve du corollaire, il suffit d'utiliser le lemme \ref{zetabeta} dans la relation \ref{relspec}.
\end{proof}

\subsection{Formule de Plancherel explicite sur $H_n \backslash G_{2n}$}
\label{secPlanchExpl}

On note $Y_n = H_n \backslash G_{2n}$ munie de la mesure quotient. On note $\mathcal{S}(Y_n, \theta)$ l'ensemble des fonctions lisses sur $G_{2n}$, $(H_n, \theta)$-équivariante à gauche et à support compact modulo $H_n$. 

On dispose d'une surjection $f \in \mathcal{S}(G_{2n}) \mapsto \varphi_f \in \mathcal{S}(Y_n, \theta)$ avec
\begin{equation}
\varphi_f(y) = \int_{H_n} f(hy) \theta(h)^{-1} dh,
\end{equation}
pour tout $y \in G_{2n}$. 

Soient $\varphi_1, \varphi_2 \in \mathcal{S}(Y_n, \theta)$, il existe $f_1, f_2 \in \mathcal{S}(G_{2n})$ tels que $\varphi_i = \varphi_{f_i}$ pour $i = 1,2$. On a
\begin{equation}
\label{psf}
(\varphi_1, \varphi_2)_{L^2(Y_n)} = \int_{H_n} f(h) \theta(h)^{-1} dh,
\end{equation}
où $f = f_1 * f_2^{*}$, on note $f_2^*(g) = \overline{f_2(g^{-1})}$. 

En effet,
\begin{equation}
(\varphi_1, \varphi_2)_{L^2(Y_n)} = \int_{Y_n} \int_{H_n \times H_n} f_1(h_1 y) \overline{f_2(h_2 y)} \theta(h_1)^{-1} \theta(h_2) dh_1 dh_2 dy.
\end{equation}

L'intégrale double est absolument convergente. On effectue le changement de variable $h_1 \mapsto h_1h_2$ et on combine les intégrales selon $y$ et $h_2$ en une intégrale sur $G_{2n}$. Ce qui donne
\begin{equation}
\begin{split}
(\varphi_1, \varphi_2)_{L^2(Y_n)} &= \int_{G_{2n}} \int_{H_n} f_1(h_1 y) \overline{f_2(y)} \theta(h_1)^{-1} dh_1 dy \\
&= \int_{H_n} f(h) \theta(h)^{-1} dh,
\end{split}
\end{equation}
puisque $f(h) = \int_{G_{2n}} f_1(h y) \overline{f_2(y)} dy$ et que l'on peut échanger l'ordre d'intégration.

On pose
\begin{equation}
\label{defps}
(\varphi_1, \varphi_2)_{Y_n, \pi} = (f_1, f_2)_{Y_n, \pi} = \int_{H^P_n \cap N_{2n} \backslash H^P_n} \beta\left(W_{f,\pi}(\xi_p,.)\right) \theta(\xi_p) d\xi_p,
\end{equation}
pour tout $\pi \in T(Temp(SO(2n+1)))$.

On note $\mathcal{S}(Y_n, \theta)_\pi$ le quotient de $\mathcal{S}(Y_n, \theta)$ par l'intersection des noyaux de toutes les applications $\mathcal{S}(Y_n, \theta) \rightarrow \pi$ linéaires $G_{2n}$-équivariantes.

\begin{proposition}
Supposons $\pi = T(\sigma)$ avec $\sigma \in Temp(SO(2n+1))$.
La forme sesquilinéaire $(., .)_{Y_n, \pi}$ sur $\mathcal{S}(G_{2n})$ est une forme hermitienne semi-definie positive qui se factorise par $\mathcal{S}(Y_n, \theta)_\pi$.
\end{proposition}

\begin{proof}
Commençons par le
\begin{lemme}
\label{decbase}
Soit $\pi \in Temp(G_{2n})$. On introduit un produit scalaire sur $\mathcal{W}(\pi, \psi)$ :
\begin{equation}
(W, W')^{Wh} = \int_{N_{2n} \backslash P_{2n}} W(p)\overline{W'(p)} dp,
\end{equation}
pour tous $W, W' \in \mathcal{W}(\pi, \psi)$.

L'opérateur $\pi(f^{\vee}) : \mathcal{W}(\pi, \psi) \rightarrow \mathcal{W}(\pi, \psi)$ est de rang fini. Notons $\mathcal{B}(\pi, \psi)_f$ une base finie orthonormée de son image. Alors
\begin{equation}
W_{f,\pi} = \sum_{W' \in \mathcal{B}(\pi, \psi)_f} \overline{\pi(\overline{f})W'} \otimes W'.
\end{equation}
\end{lemme}

\begin{proof}
Le produit scalaire $(.,.)^{Wh}$ est $P_{2n}$-équivariant, d'après Bernstein \cite{bernstein}, il est aussi $G_{2n}$-équivariant.

Pour $W \in \mathcal{W}(\pi, \psi)$, la décomposition de $\pi(f^{\vee})W$ selon ce produit scalaire est
\begin{equation}
\label{decpif}
\begin{split}
\pi(f^{\vee})W &= \sum_{W' \in \mathcal{B}(\pi, \psi)_f} (\pi(f^{\vee})W, W')^{Wh}W' \\
&= \sum_{W' \in \mathcal{B}(\pi, \psi)_f} (W, \pi(\overline{f})W')^{Wh}W'.
\end{split}
\end{equation}

Cette égalité nous permet grâce au produit scalaire $(.,.)^{Wh}$ de faire l'identification
\begin{equation}
\pi(f^\vee) = \sum_{W' \in \mathcal{B}(\pi, \psi)_f} W' \otimes \overline{\pi(\overline{f})W'}.
\end{equation}

On en déduit, d'après \ref{decpif}, que
\begin{equation}
\begin{split}
W_{f, \pi}(g_1, g_2) &= \int_{N_{2n}}^* Tr(\pi(g_1^{-1}ug_2)\pi(f^\vee)) \psi(u)^{-1} du \\
& = \sum_{W' \in \mathcal{B}(\pi, \psi)_f} \int_{N_{2n}}^* (\pi(ug_2)W', \pi(g_1)\pi(\overline{f})W')\psi(u)^{-1}du \\
&= \sum_{W' \in \mathcal{B}(\pi, \psi)_f} W'(g_2) \overline{\pi(\overline{f})W'}(g_1),
\end{split}
\end{equation}
pour tous $g_1, g_2 \in G_{2n}$. La dernière égalité provient de \cite[Lemme 4.4]{lapid-mao}.
\end{proof}

La définition \ref{defps} et le lemme \ref{decbase} donne la relation
\begin{equation}
\begin{split}
(f_1, f_2)_{Y_n, T(\sigma)} &= \sum_{W' \in \mathcal{B}(\pi, \psi)_f} \int_{H^P_n \cap N_{2n} \backslash H^P_n} \beta(W') \overline{\pi(\overline{f})W'}(\xi_p)\theta(\xi_p) d\xi_p\\
& = \sum_{W' \in \mathcal{B}(T(\sigma), \psi)_f} \beta(W') \overline{\beta(T(\sigma)(\overline{f}_1)T(\sigma)(f_2^\vee)W')}
\end{split}
\end{equation}
qui ne dépend que de $\varphi_1$ et $f_2$ puisque la restriction de $\beta$ à $\mathcal{W}(T(\sigma), \psi)$ est $(H_n, \theta)$-équivariante, d'après la proposition \ref{lemmebeta}. En échangeant les rôles de $\varphi_1$ et $\varphi_2$, on voit que $(f_1, f_2)_{Y_n, T(\sigma)}$ ne dépend que de $\varphi_1$ et $\varphi_2$. De plus, $(f_1, f_2)_{Y_n, T(\sigma)}$ dépend uniquement de $T(\sigma)(f_1)$ et $T(\sigma)(f_2)$. On en déduit que $(.,.)_{Y_n, \pi}$ se factorise par $\mathcal{S}(Y_n, \theta)_{T(\sigma)}$.
\end{proof}

On remarque pour la suite que l'on a
\begin{equation}
(\varphi_1, \varphi_2)_{Y_n, T(\sigma)} = (f_1, f_2)_{Y_n, T(\sigma)} = (\widetilde{\beta}\otimes\beta)(W_{f,\pi}).
\end{equation}

\begin{theoreme}
\label{thPlanch}
Soient $\varphi_1, \varphi_2 \in \mathcal{S}(Y_n, \theta)$. On a
\begin{equation}
(\varphi_1, \varphi_2)_{L^2(Y_n)} = \int_{Temp(SO(2n+1))/Stab} (\varphi_1, \varphi_2)_{Y_n, T(\sigma)} \frac{|\gamma^*(0, \sigma, Ad, \psi)|}{|S_\sigma|}d\sigma.
\end{equation}
\end{theoreme}

\begin{proof}
D'après \ref{unfolding} et \ref{corolim}, on a
\begin{equation}
\label{intfin}
\begin{split}
\int_{H_n} f(h) \theta(h)^{-1} dh = &\int_{H_n \cap N_{2n} \backslash H^P_n} \int_{Temp(SO(2n+1))/Stab} \beta\left(W_{f,T(\sigma)}(\xi_p,.)\right) \\
& \theta(\xi_p) \frac{\gamma^*(0, \sigma, Ad, \psi)}{|S_\sigma|}c(T(\sigma))^{-1}c_\beta(\sigma) d\sigma d\xi_p.
\end{split}
\end{equation}

\begin{lemme}
La fonction $\sigma \mapsto \beta\left(W_{f,T(\sigma)}(\xi_p,.)\right)$ est à support compact.
\end{lemme}

\begin{proof}
D'après la définition de $f_\pi$, $W_{f,\pi}$ est nul dès que $\pi(f^\vee)$ l'est.

Soit $K$ un sous-groupe ouvert compact (que l'on choisit distingué dans $GL_{2n}(\mathbb{Z}_p)$) tel que $f^\vee$ est biinvariant par $K$. Alors $\pi(f^\vee) \neq 0$, seulement lorsque $\pi$ admet des vecteurs $K$-invariant non nuls.

Pour $M$ un sous-groupe de Levi de $G_{2n}$, d'après Harish-Chandra \cite[Théorème VIII.1.2]{waldspurger}, il n'y a qu'un nombre fini de représentations $\tau \in \Pi_2(M)$ modulo $X^*(M) \otimes i\mathbb{R}$ qui admettent des vecteurs $(K \cap M)$-invariant non nuls.

Comme toute représentation $\pi \in Temp(G_{2n})$ est une induite d'une telle représentation $\tau$ pour un bon choix de sous-groupe de Levi $M$ et que $\pi$ admet des vecteurs $K$-invariant non nuls si et seulement si $\tau$ admet des vecteurs $(K \cap M)$-invariants non nuls, on en déduit le lemme.
\end{proof}

On en déduit que
\begin{equation}
\int_{Temp(SO(2n+1))/Stab} \beta\left(W_{f,T(\sigma)}(\xi_p,.)\right) \theta(\xi_p) \frac{\gamma^*(0, \sigma, Ad, \psi)}{|S_\sigma|}c(T(\sigma))^{-1}c_\beta(\sigma) d\sigma
\end{equation}
est absolument convergente.

De plus, l'intégration extérieure $\int_{H_n \cap N_{2n} \backslash H^P_n} \theta(\xi_p)d\xi_p$ n'est autre que la forme linéaire continue $\widetilde{\beta}$, on en déduit que l'intégrale \ref{intfin} devient
\begin{equation}
\begin{split}
&\widetilde{\beta}\left(x_p \mapsto \int_{Temp(SO(2n+1))/Stab} \beta\left(W_{f,T(\sigma)}(\xi_p,.)\right) \frac{\gamma^*(0, \sigma, Ad, \psi)}{|S_\sigma|}c(T(\sigma))^{-1}c_\beta(\sigma) d\sigma\right) \\
&= \int_{Temp(SO(2n+1))} (\widetilde{\beta}\otimes\beta)(W_{f, T(\sigma)}) \frac{\gamma^*(0, \sigma, Ad, \psi)}{|S_\sigma|} c(T(\sigma))^{-1}c_\beta(\sigma)d\sigma \\
&= \int_{Temp(SO(2n+1))/Stab} (\varphi_1, \varphi_2)_{Y_n, T(\sigma)} \frac{\gamma^*(0, \sigma, Ad, \psi)}{|S_\sigma|} c(T(\sigma))^{-1}c_\beta(\sigma)d\sigma.
\end{split}
\end{equation}

Pour finir, \cite[prop 4.1.1]{beuzart-plessis} nous dit que les formes sesquilinéaires $(\varphi_1, \varphi_2) \mapsto (\varphi_1, \varphi_2)_{Y_n, T(\sigma)} \frac{\gamma^*(0, \sigma, Ad, \psi)}{|S_\sigma|} c(T(\sigma))^{-1}c_\beta(\sigma)$ sont automatiquement définies positives. On en déduit que 
\begin{equation}
\gamma^*(0, \sigma, Ad, \psi) c(T(\sigma))^{-1}c_\beta(\sigma) = |\gamma^*(0, \sigma, Ad, \psi)|.
\end{equation}
\end{proof}

\begin{corollaire}
\label{planchab}
On a une décomposition de Plancherel abstraite sur $L^2(H_n \backslash G_{2n})$ :
\begin{equation}
L^2(H_n \backslash G_{2n}) = \int^{\oplus}_{Temp(SO(2n+1))/Stab} T(\sigma) d\sigma.
\end{equation}
\end{corollaire}

\begin{proof}
C'est une conséquence du théorème \ref{thPlanch} et du fait que $\frac{|\gamma^*(0, \sigma, Ad, \psi)|}{|S_\sigma|}$ est presque partout non nul.
\end{proof}

\subsection{Formule de Plancherel abstraite sur $G_n \times G_n \backslash G_{2n}$}

\begin{lemme}
\label{lemmeiso}
On dispose d'un isomorphisme $G_{2n}$-équivariant d'espaces de Hilbert
\begin{equation}
L^2(G_n \times G_n \backslash G_{2n}) \simeq L^2(H_n \backslash G_{2n}, \theta).
\end{equation}
\end{lemme}

\begin{proof}
On note $\mathcal{S}(G_n \times G_n \backslash G_{2n})$ l'ensemble des fonctions lisse sur $G_{2n}$, $G_n \times G_n$-équivariante et à support compact modulo $G_n \times G_n$.
On considère l'application $f \in \mathcal{S}(H_n \backslash G_{2n}, \theta) \mapsto \widetilde{f} \in \mathcal{S}(G_n \times G_n \backslash G_{2n})$, où $\widetilde{f}$ est définie par
\begin{equation}
\widetilde{f}(g) = \int_{G_n} f\left( \sigma_n \begin{pmatrix}
\gamma & 0 \\
0 & 1_n
\end{pmatrix}g\right) d\gamma
\end{equation}
pour tout $g \in G_{2n}$.

Commençons par montrer que l'application est bien définie. En effet, pour $g' \in G_n$ et $X \in M_n$, on a
\begin{equation}
\begin{pmatrix}
g' & X \\
0 & g'
\end{pmatrix}\begin{pmatrix}
\gamma & 0 \\
0 & 1_n
\end{pmatrix} = \begin{pmatrix}
g' \gamma & X \\
0 & g'
\end{pmatrix}.
\end{equation}

On note $K$ un compact tel que $supp(f) \subset H_nK$. On en déduit que $f\left( \sigma_n \begin{pmatrix}
\gamma & 0 \\
0 & 1_n
\end{pmatrix}g \right)$ est nul sauf si il existe $g' \in G_n$ tel que $\begin{pmatrix}
g' \gamma & X \\
0 & g'
\end{pmatrix} \in \sigma_n^{-1}K$. On en déduit alors que $ \gamma \in G_n \mapsto f\left( \sigma_n \begin{pmatrix}
\gamma & 0 \\
0 & 1_n
\end{pmatrix}g \right)$ est à support compact. L'intégrale est donc absolument convergente. De plus,
pour tous $g_1, g_2 \in G_n$ et $g \in G_{2n}$, on a
\begin{equation}
\begin{split}
\widetilde{f}\left( \begin{pmatrix}
g_1 & 0 \\
0 & g_2
\end{pmatrix}g \right) &= \int_{G_n} f\left( \sigma_n \begin{pmatrix}
\gamma & 0 \\
0 & 1_n
\end{pmatrix}\begin{pmatrix}
g_1 & 0 \\
0 & g_2
\end{pmatrix}g \right) d\gamma \\
&=_{\gamma \mapsto g_2 \gamma g_1^{-1}} \int_{G_n} f\left( \sigma_n \begin{pmatrix}
g_2 & 0 \\
0 & g_2
\end{pmatrix}\begin{pmatrix}
\gamma & 0 \\
0 & 1_n
\end{pmatrix}g \right) d\gamma \\
&= \int_{G_n} f\left( \sigma_n \begin{pmatrix}
\gamma & 0 \\
0 & 1_n
\end{pmatrix}g \right) d\gamma \\
&= \widetilde{f}(g).
\end{split}
\end{equation}

Pour finir, montrons que $\widetilde{f}$ est à support compact modulo $G_n \times G_n$. Grâce à la décomposition d'Iwasawa, écrivons $g$ sous la forme $\begin{pmatrix}
g_1 & x \\
0 & g_2
\end{pmatrix}k$ avec $g_1, g_2 \in G_n$, $x \in M_n$ et $k \in K$. Alors $\widetilde{f}(g) = \widetilde{f}\left( \begin{pmatrix}
1 & y \\
0 & 1
\end{pmatrix}k \right)$ avec $y = g_1^{-1}x$, on a alors
\begin{equation}
\begin{split}
\widetilde{f}(g) &= \int_{G_n} f\left( \sigma_n \begin{pmatrix}
1 & \gamma y \\
0 & 1
\end{pmatrix} \begin{pmatrix}
\gamma & 0 \\
0 & 1
\end{pmatrix} k \right) d\gamma \\
&= \int_{G_n} f\left( \sigma_n \begin{pmatrix}
\gamma & 0 \\
0 & 1
\end{pmatrix} k \right) \psi(Tr(\gamma y))d\gamma.
\end{split}
\end{equation}
Cette dernière intégrale est la transformée de Fourier d'une fonction à support compact sur $M_n$, à savoir la fonction $\phi_k$ définie par $\phi_k(x) = f\left( \sigma_n \begin{pmatrix}
x & 0 \\
0 & 1
\end{pmatrix} k \right)|\det x|^{-n}$ si $x \in G_n$ et $0$ sinon. Le facteur $|\det x|^{-n}$ provient de la transformation de la mesure multiplicative $d\gamma$ en une mesure additive. On en déduit que $\widetilde{f}$ est à support compact modulo $G_n \times G_n$.
Ce qui prouve que l'application $f \in \mathcal{S}(H_n \backslash G_{2n}, \theta) \mapsto \widetilde{f} \in \mathcal{S}(G_n \times G_n \backslash G_{2n})$ est bien définie.

Cette application est linéaire et injective. En effet, si $\widetilde{f} = 0$, alors $\phi_k = 0$ pour tout $k \in K$, donc $f\left( \sigma_n \begin{pmatrix}
\gamma & 0 \\
0 & 1
\end{pmatrix} k \right) = 0$ pour tout $\gamma \in G_n$ et $k \in K$. On en déduit que $f = 0$ car elle est $(H_n, \theta)$-équivariante.

Pour finir, montrons qu'il existe une constante $c > 0$ telle que $||f||_{L^2(H_n \backslash G_{2n}, \theta)} = c||\widetilde{f}||_{L^2(G_n \times G_n \backslash G_{2n})}$. Ce qui prouve que l'application $f \in C^\infty_c(H_n \backslash G_{2n}, \theta) \mapsto \widetilde{f} \in C^\infty_c(G_n \times G_n \backslash G_{2n})$ s'étend en un isomorphisme d'espaces de Hilbert $L^2(H_n \backslash G_{2n}, \theta) \simeq L^2(G_n \times G_n \backslash G_{2n})$.

En effet,
\begin{equation}
\begin{split}
||\widetilde{f}||_{L^2(G_n \times G_n \backslash G_{2n})} &= \int_{M_n \times K} |\widetilde{f}\left( \begin{pmatrix}
1 & x \\
0 & 1
\end{pmatrix} k \right)|^2 dx dk \\
&= \int_{M_n \times K} |\int_{G_n} f\left( \sigma_n \begin{pmatrix}
\gamma & 0 \\
0 & 1
\end{pmatrix} k \right) \psi(Tr(\gamma x)) d\gamma|^2 dx dk \\
&= \int_{M_n \times K} |\widehat{\phi}_k(x)|^2 dx dk.
\end{split}
\end{equation}

La transformé de Fourier conserve la norme $L^2$ avec un choix de constante appropriée, on en déduit qu'il existe une constante $c' > 0$ telle que
\begin{equation}
\begin{split}
||\widetilde{f}||_{L^2(H_n \backslash G_{2n}, \theta)} &= c' \int_{M_n \times K} |\phi_k(x)|^2 dx dk \\
&= c' \int_{K} \int_{G_n} |f\left( \sigma_n \begin{pmatrix}
\gamma & 0 \\
0 & 1
\end{pmatrix} k \right)|^2 \frac{d\gamma}{|\det \gamma|^n} dk.
\end{split}
\end{equation}
On met l'accent sur le fait que l'on a modifié la mesure additive sur $M_n$ restreinte à $G_n$ en une mesure multiplicative sur $G_n$. La mesure $\frac{d\gamma}{|\det \gamma|^n} dk$ est une mesure de invariante sur $G_n K \simeq H_n \backslash G_{2n}$. On en déduit bien qu'il existe une constante $c > 0$ telle que $||f||_{L^2(H_n \backslash G_{2n}, \theta)} = c||\widetilde{f}||_{L^2(G_n \times G_n \backslash G_{2n})}$.
\end{proof}

Cet isomorphisme d'espace $L^2$ nous permet de faire le lien entre les formules de Plancherel sur  $G_n \times G_n \backslash G_{2n}$ et sur $H_n \backslash G_{2n}$. En effet, on dispose du
\begin{theoreme}
Une décomposition de Plancherel abstraite sur $L^2(G_n \times G_n \backslash G_{2n})$ est obtenue par la relation
\begin{equation}
L^2(G_n \times G_n \backslash G_{2n}) = \int^{\oplus}_{Temp(SO(2n+1))/Stab} T(\sigma) d\sigma.
\end{equation}
\end{theoreme}

\begin{proof}
C'est une conséquence du lemme \ref{lemmeiso} et du corollaire \ref{planchab}.
\end{proof}

 \bibliographystyle{siam}
 \bibliography{article}

\end{document}